\documentclass[11pt]{amsart}
\usepackage{amsmath,amsthm}
\usepackage {latexsym}
\usepackage{amssymb}

\usepackage{xcolor}
\definecolor{darkgreen}{rgb}{0,0.7,0}
\usepackage{todonotes}

\everymath{\displaystyle}

\usepackage[utf8]{inputenc}
\usepackage{xspace}
\usepackage[colorlinks=true]{hyperref}

\newcommand{\Sph}{\mathbb{S}}

\newcommand{\R}{\mathbb{R}}

\newtheorem{thm}{THEOREM}[section]
\newtheorem{remark}[thm]{REMARK}
\newtheorem{lem}[thm]{LEMMA}
\newtheorem{defn}[thm]{DEFINITION}
\newtheorem{prop}[thm]{PROPOSITION}
\newtheorem{cor}[thm]{COROLLARY}

\date{\today}

\begin{document}
\title[Global control of the heat flow]{Global controllability to  harmonic maps of the heat flow from a circle to a sphere}

\author{Jean-Michel Coron}
\address{Sorbonne Universit\'{e}, Universit\'{e} Paris-Diderot SPC, CNRS, INRIA, Laboratoire Jacques-Louis Lions, LJLL,  \'{e}quipe CAGE, F-75005 Paris, France}
\email{\texttt{jean-michel.coron@sorbonne-universite.fr}}
\thanks{}

\author{Shengquan Xiang}
\address{School of Mathematical Sciences, Peking University, 100871, Beijing, P. R. China}
\email{\texttt{shengquan.xiang@math.pku.edu.cn}}
\thanks{}

\begin{abstract}
In this paper, we study the global controllability and stabilization problems of the harmonic map heat flow from a circle to a sphere. Combining ideas from control theory, heat flow, differential geometry, and asymptotic analysis, we  obtain several  important properties,  such as small-time local controllability, local quantitative rapid stabilization, obstruction to semi-global asymptotic stabilization, and global controllability to geodesics.  Surprisingly, due to the geometric feature of the equation we also discover the small-time global controllability between harmonic maps within the same homotopy class for general compact Riemannian manifold targets, which is to be compared with the analogous but longstanding problem for the nonlinear heat equations.
\end{abstract}
\subjclass[2010]{35K58,   35B40, 93C20}
\thanks{\textit{Keywords.} harmonic map heat flow, global controllability, quantitative rapid stabilization, degree theory.}
\maketitle

\setcounter{tocdepth}{2}
\tableofcontents

\section{Introduction}

 Harmonic maps are maps between Riemannian manifolds which are critical points of a certain energy functional, the Dirichlet energy. They have applications in various fields of mathematics and physics. Let us be given a domain $\Omega\subset \R^l$ and a compact Riemannian manifold $(\mathcal{N}, g)$.  Using the  Nash isometric embedding theorem, we may assume that $(\mathcal{N}, g)$ is a Riemannian submanifold of some $\mathbb{R}^n$ equipped with the Euclidean metric.  The harmonic map heat flow is a process that starts with an initial map and deforms it over time in such a way that it tends to decrease the energy of the map. It is governed by the following geometric heat-like equation
 \begin{equation*}
 u \in \mathcal{N} \text{ and }    \partial_t u- \Delta u\perp T_{u} \mathcal{N}.
 \end{equation*}
 This geometric equation finds extensive applications across a wide spectrum of disciplines, including physics, fluid dynamics, materials sciences, and even computer vision.  For example, it appears in the study of solutions to the Einstein field equations, and it plays a significant role in the understanding of soap films, elasticity, and various other phenomena that connect to minimal surfaces. Concerning fluid dynamics especially to the liquid crystals,  this equation is in principle investigated in the following special cases, such as the harmonic mappings from a three-dimensional domain to a sphere \cite{LESLIE19791}, the Oseen-Frank equation \cite{MR1369095}, and the Ericksen model \cite{MR3621817, Lin-Wang-2016}.  Harmonic maps and their related flows  have also been applied to image analysis, especially in tasks like image denoising, segmentation, and shape matching.

The control of the harmonic maps equation naturally finds applications to the preceding illustrated topics. In a physical example involving liquid crystals, the concept of harmonic maps comes into play through the modeling of the director field. The director field represents the preferred molecular orientation in a given region of a liquid crystal. Understanding how this field evolves and interacts with external influences (such as electric or magnetic fields) is critical to the design and optimization of liquid crystal-based devices, and these fields interact with the director field in a PDEs control background that can quantitatively or qualitatively describe the deformation and response.
Nevertheless, the rigorous mathematical investigation of the controllability of this geometric equation remains quite limited \cite{MR3621817,  Liu-2018}. For example, in \cite[Theorem 1]{Liu-2018} Liu has studied the global controllability of the harmonic map heat flow from $\Omega\subset \mathbb{R}^3$ to $\mathbb{S}^2$, but his control is not localized due to the use of a strong external magnetic field.  We also refer to the recent
study by the authors on the control of wave map equations \cite{Coron-Krieger-Xiang-1, Krieger-Xiang-2022}.
This  is, of course, related to the well-understood topic of control of the heat
equations, see for instance the one-dimensional results  \cite{2016-Coron-Nguyen-ARMA,  GHXZ-2022, 1971-Fattorini-Russell-ARMA,  1995-Guo-Littman, MR4153111, 1977-Jones-Frank-JMAA,  1978-Littman-ASNSP,  2014-Martin-Rosier-Rouchon-A,  2016-Martin-Rosier-Rouchon-SICON} and the multi-dimensional results  \cite{DZZ-2008, EZ-2011, MR1791879, MR1406566, Lebeau-Robbiano-CPDE}. But to our limited knowledge, there is no result in the literature on the study
of the global controllability\footnote{Using the standard terminology of
control theory of PDEs, by \textit{global} we mean that the scale of states can be
as large as we want. This differs from the widely used definition of ``global" in the well-posedness problem for
long-time behavior.  } of harmonic map heat flow with localized controls.    As the first paper
on the study of global control of harmonic map heat flow with localized control, we restrict
ourselves to the simplest case: the initial domain is the periodic
domain $\mathbb{T}^1:=\R/2\pi \mathbb{Z}$ while the target manifold is the  Euclidean sphere $\Sph^k:=\{u\in \R^{k+1}:\, |u|=1\}$ with $k\in \mathbb{N}\setminus\{0\}$.
Then the harmonic map heat flow takes the form
\begin{equation}\label{eq:freehmhf}
    \partial_t u- \Delta u= |\partial_x u|^2 u. \notag
\end{equation}
In this paper, we are interested in the control related properties of the controlled harmonic map heat flow equation.  For any given force term $f: \mathbb{R}^+\times \mathbb{T}^1
\rightarrow  \mathbb{R}^{k+1}$ the controlled harmonic map heat flow is written as
\begin{equation}\label{eq:controlhmhf}
    \partial_t u- \Delta u= |\partial_x u|^2 u+ \mathbf{1}_{\omega} f^{u^{\perp}},
\end{equation}
where the notation $f^{u^{\perp}}$ represents the projection of $f$ onto the tangent space $T_u \Sph^k\subset \R^{k+1}$ and $\mathbf{1}_{\omega}$ is the characteristic function of $\omega$:
\begin{equation*}
\mathbf{1}_{\omega}(x)=1 \text{ if $x\in \omega$ and $\mathbf{1}_{\omega}(x)=0$ if $x\in \mathbb{T}^1\setminus \omega$. }
\end{equation*}

\begin{figure}[t]
\centering
\includegraphics[width=0.95\linewidth, trim={0cm 0.0cm 0cm 0.0cm},clip]{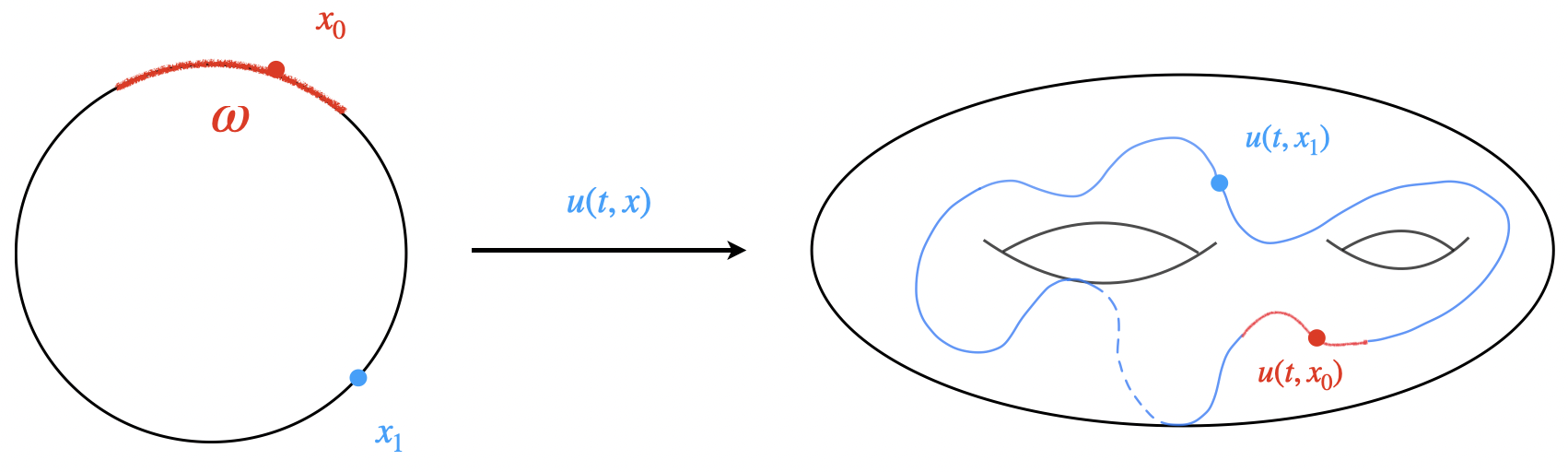}
\caption{The controlled harmonic map heat flow from $ \mathbb{T}^1$ to  a torus of genus 2.  The blue curve is the solution at a given time $t$, and the red part is the place where we are allowed to add extra control force.}
\label{pic1}
\end{figure}

One easily verifies that the solutions (if they exist) remain on the sphere $\mathbb{S}^k$ if the initial datum takes its values in $\mathbb{S}^k$.
As usual, in this paper, we consider localized control problems, which means that the force term $f$ can only act on some given open, nonempty, subdomain $\omega \subset \mathbb{T}^1$ (which can be very small). Physically, this means that our extra force only acts on a localized part of $\mathbb{T}^1$.  By choosing well the force term $f$ we are able to change the deformation of the solution, and one may wonder if, starting from a given initial state, one can then reach a given state. This is called the global controllability problem.   If the force term $f$ is chosen in the form of a feedback law that depends on the current state, for example $f(t, x):= u_x(t, x)$,  we call the system a closed-loop system. If, in addition, the stability of this new system becomes stronger than the original one, we call this process of selecting $f$ and enhancing the stability of the system {\it stabilization}.  This paper is devoted to the study of global controllability and stabilization of the controlled harmonic map heat flow equation.\\
Throughout this paper we define and consider the energy
\begin{equation}
    E(v(t, \cdot)):= \int_{\mathbb{T}^1} |\partial_x v(t, x)|^2 dx.
\end{equation}
When there is no risk of confusion, we simply denote the preceding energy by $E(v(t))$ or $E(t)$.
We also define the following energy level set.
\begin{defn}[Energy level set]
   Let $e>0$.  Let us denote by $\textbf{H}(e)$ the set of  states
    \begin{equation}
        \textbf{H}(e):= \{v:\mathbb{T}^1\rightarrow \R^{k+1}: v(x)\in \mathbb{S}^k\; \forall x\in \mathbb{T}^1, \; E(v(x))\leq e \}. \notag
    \end{equation}
\end{defn}

\begin{figure}[t]
\centering
\includegraphics[width=0.6\linewidth, trim={0cm 0.0cm 0cm 0.0cm},clip]{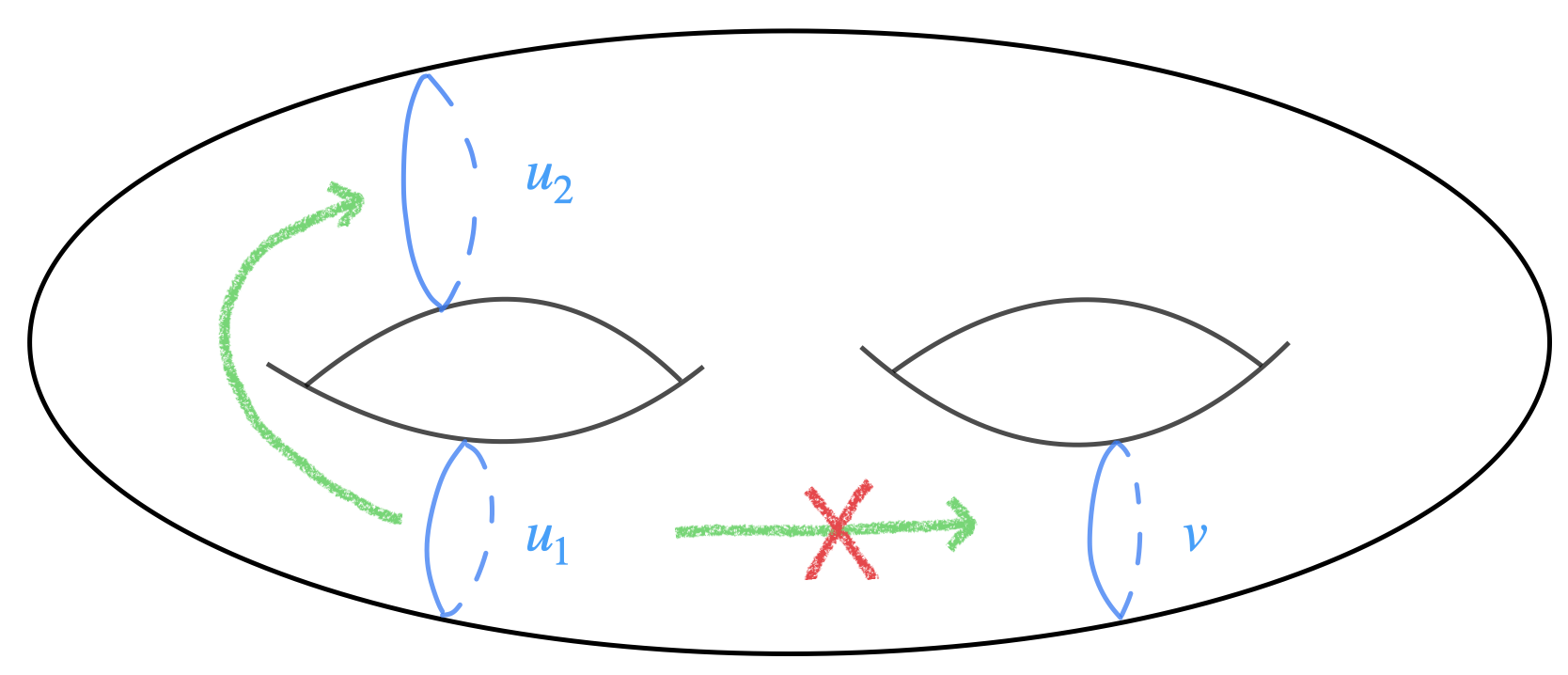}
\caption{Deformation of the curve in the same homotopy class. In this picture, the solution remains in the torus of genus 2. We observe that the state $u_1$ can be continuously deformed to the state $u_2$, but it cannot reach the state $v$. The {\it homotopy problem} states that every state is homotopic to a harmonic map. See Proposition   \ref{thm:homopotyconver} and the paper \cite{Ottarsson} for details.
}
\label{pictorus}
\end{figure}

\subsection{The main results}

\subsubsection{\textbf{Stability of the harmonic map heat flow}}
We start with the free harmonic map heat flow, namely no extra force is applied to the system:
\begin{equation}\label{eq:freehmhfs}
    \begin{cases}
    \partial_t u- \Delta u= |\partial_x u|^2 u, \\
    u(0, \cdot)= u_0(\cdot).
    \end{cases}
\end{equation}
One of the central problems in the study of the harmonic map heat flow is the so-called {\it homotopy problem}. Recall the following concepts about {\it approximate harmonic maps}.
\begin{defn}
Let $0<\varepsilon<1$. We call ``$\varepsilon$-approximate harmonic maps" the class of states $u: \mathbb{T}^1\rightarrow \Sph^k$ that belong to the set
\begin{equation}
\mathcal{Q}_{\varepsilon}:= \bigcup\limits_{\Phi: \textrm{ a harmonic map }} \left\{u: \mathbb{T}^1\rightarrow \Sph^k : \|u- \Phi\|_{H^1}\leq \varepsilon\right\}.
\end{equation}
\end{defn}
The  crucial  homotopy problem  asks about whether a given map $\phi_0: \mathcal{M}\rightarrow \mathcal{N}$ between two Riemannian manifolds can be deformed into a harmonic map   $\bar \phi: \mathcal{M}\rightarrow \mathcal{N}$.   One
 of the important methods to this classical problem is the harmonic map heat flow method:  for any given initial state $\phi_0$, one investigates the flow of the state and proves that the flow converges to a harmonic map.
\\

For instance, in our framework  one easily observes that when the initial energy is smaller than $1/2\pi$ the heat flow converges to a constant state.  Indeed, by performing the naive energy estimates for the free harmonic map heat flow one  observes that
\begin{equation}\label{ine:enedir}
    \frac{1}{2}\frac{d}{dt} \int_{\mathbb{T}^1} |u_x|^2(t, x) dx= -\int_{\mathbb{T}^1} |u_t|^2(t, x)dx= \int_{\mathbb{T}^1} (-|u_{xx}|^2+ |u_x|^4)(t, x) dx.
\end{equation}
The first part of the preceding equality also implies the global energy dissipation: $E(u(t))$ decreases with respect to time.
This equation, to be combined with a Sobolev interpolation inequality,
\begin{align*}
    - \|u_{xx}\|_{L^2}^2+ \|u_x\|_{L^4}^4 &\leq - \|u_{xx}\|_{L^2}^2+ \|u_x\|_{L^2}^2 \|u_x\|_{L^{\infty}}^2 \\
    &\leq \|u_{x}\|_{L^{\infty}}^2\left( - \frac{1}{\pi} + \|u_x\|_{L^2}^2\right) \\
     &\leq   -\frac{1}{2\pi} \|u_{x}\|_{L^{\infty}}^2 \;  \textrm{ provided that }   \|u_x\|_{L^2}^2\leq \frac{1}{2\pi},
\end{align*}
implies local exponential energy decay
\begin{equation}\label{eq:localexdecay}
    \frac{d}{dt} E(u(t))\leq - \frac{1}{2\pi^2} E(u(t))\; \forall t>0,
\end{equation}
provided that $E(u(0))\leq 1/2\pi$. One may ask if such a direct approach leads to global asymptotic dissipation of the energy.
However, the above simple analysis is based on a Sobolev interpolation inequality and the Poincar\'{e} inequality,  it is not clear that  there is uniform exponential stability for initial states belonging to $\mathbf{H}(2\pi-\varepsilon)$ for any $\varepsilon>0$. Besides, since harmonic maps are stationary states of the system, for an initial state with energy larger than $2\pi$ it is possible that the evolution of the system converges to some non-constant harmonic maps. The global asymptotic behavior of the system becomes more delicate.
\begin{figure}[t]
\centering
\includegraphics[width=0.7\linewidth, trim={0cm 0.0cm 0cm 0.0cm},clip]{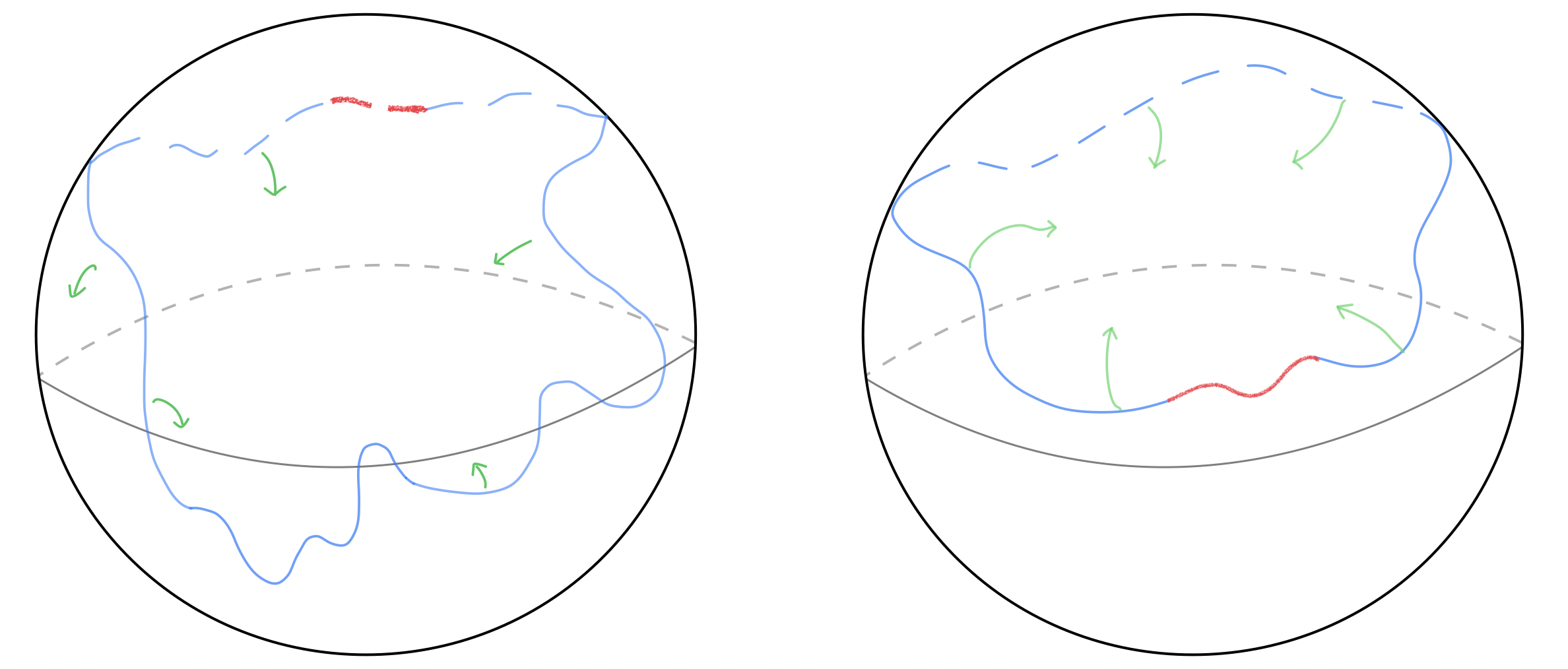}
\caption{The natural energy dissipation of the harmonic map heat flow from $\mathbb{T}^1$ to $\mathbb{S}^2$.  The green arrow indicates the deformation of the curve (solution). See Proposition \ref{thm:homopotyconver} for details on this convergence result.  \protect \\
The  picture on the left shows that the solution converges to a harmonic map having $2\pi$-energy. This picture is also related to \hyperref[Step1]{Step 1} of Section \ref{sec:strategy}.  \protect \\
The picture on the right shows that the solution converges to a constant state. This picture is also related to \hyperref[Step4]{Step 4} of Section \ref{sec:strategy}.}
\label{pic2}
\end{figure}

Many important works are devoted to the asymptotic convergence to harmonic maps of the harmonic map heat flow resulting in fruitful results. In
 \cite{Eells-Sampson-1964} Eells and Sampson proved this convergence when the Riemannian sectorial curvature of $\mathcal{N}$ is nonpositive and both $\mathcal{M}$ and $\mathcal{N}$ are compact manifolds without boundary. Then  in \cite{Hamilton-1975} Hamilton further showed the same result when $\mathcal{M}$ and $\mathcal{N}$ are compact manifolds with boundary. For the case where $\mathcal{M}$ is simply $\mathbb{T}^1$ and $\mathcal{N}$ is an arbitrary Riemannian manifold, Ottarsson proved the convergence result in \cite{Ottarsson}.     We refer to the review papers \cite{Eells-Lemaire, Eells-Lemaire-2} on this topic. It is noteworthy that in most of the literature the results obtained are the so-called subconvergence property, namely that  there exists a harmonic map  and a sequence of times $\{t_k\}_k$ tending to infinity such that $u(t_k)$ converges to a harmonic map. The ``unique asymptotic limit" problem  asks if and when we can replace the ``subconvergence'' by ``convergence''.  Regarding  this problem, in cases where the target manifold $\mathcal{N}$ is real analytic  and satisfies several other conditions, including having non-positive curvature, the result can be improved to a convergence result. See for example \cite[Corollary 2]{1983-Simon-AM}.   Notice that such asymptotic behavior result in our framework is covered by the general result given by Ottarsson in  \cite{Ottarsson}, more precisely, one has the following subconvergence result.
\begin{prop}[ \cite{Ottarsson}, convergence to harmonic maps of the heat flow]\label{thm:homopotyconver}
    For any initial state $u_0\in H^1(\mathbb{T}^1; \Sph^k)$, for any $\varepsilon>0$, there exists some time $T= T(\varepsilon, u_0)$ such that the unique solution of \eqref{eq:freehmhfs} becomes a  ``$\varepsilon$-approximate harmonic map" at time $T$.
\end{prop}

Although Proposition \ref{thm:homopotyconver} is well known in the literature even when $\mathbb{S}^k$ is replaced by an arbitrary Riemannian manifold $\mathcal{N}$, in this paper we present a direct proof of Proposition \ref{thm:homopotyconver} concerning this specific case for readers' convenience. This direct proof is based on the flux method introduced in \cite{Krieger-Xiang-2022}, and it is  easier to follow than previous proofs.
Its idea is to use the smallness of  flux to recover the smallness of  states (while, more precisely, the closeness of  states to harmonic maps in this
specific context); see, in particular, Lemma \ref{lem:flux}. This proof can be found in Section \ref{sec:globalconvergence}.\\

This convergence result, together with the continuous dependence property of the flow, implies that for any given initial state $u(0, \cdot)$ that is sufficiently close to $u_0$ in $H^1$-topology,  the unique solution becomes a  ``$2\varepsilon$-approximate harmonic maps" after a certain time $T(\varepsilon, u_0)$ of evolution.  Thanks to the strong smoothing effect of the harmonic map heat flow (see for example Lemma \ref{lem:freeheat}), for any $M>0$ there exists an explicit constant $C_M$ such that, for any initial state $u(0, \cdot)$ satisfying $\|u(0)\|_{H^1}\leq M$, the unique solution satisfies $\|u(1)\|_{H^2}\leq C_M$. Thus $u(1, \cdot)$ stays in a compact subspace of $H^1(\mathbb{T}^1)$, as the Sobolev embedding $H^2(\mathbb{T}^1)\hookrightarrow H^1(\mathbb{T}^1)$ is compact. This leads to the following uniform convergence result.
\begin{cor}
\label{cor-uniform}
Let $M>0$. For any $\varepsilon>0$, there exists some time $T= T(\varepsilon, M)$ such that,  for any initial state $u_0\in H^1(\mathbb{T}^1; \Sph^k)$ belongs to $\mathbf{H}(M)$,  the unique solution of \eqref{eq:freehmhfs} becomes a  ``$\varepsilon$-approximate harmonic maps" at some time $t_0\in (0, T)$.
\end{cor}

Moreover, by combining Lemma \ref{lem:flux} and the local exponential stability property \eqref{eq:localexdecay} we immediately obtain the following semi-global stability result, the proof of which we omit.
\begin{cor}[Semi-global exponential stability]\label{thm-expo-decay}
Let $\varepsilon>0$. There exists  $C= C(\varepsilon)$  such that   for any initial state $u_0\in H^1(\mathbb{T}^1; \mathbb{S}^k)$ belongs to $\textbf{H}(2\pi- \varepsilon)$ the unique solution of \eqref{eq:freehmhfs}   satisfies
\begin{equation*}
    E(u(t))\leq C e^{- \frac{1}{2\pi^2}t} E(u_0) \; \forall t>0.
\end{equation*}
\end{cor}

\subsubsection{\textbf{Stabilization of the harmonic map heat flow}}
Inspired by the recent works on the stabilization of wave maps equation \cite{Coron-Krieger-Xiang-1, Krieger-Xiang-2022}, we are also interested in the related stabilization problems for the harmonic map heat flow. According to Corollary \ref{thm-expo-decay} the system is locally exponentially stable.  However, since harmonic maps are steady states, the system is not globally stable.  Concerning this subject we successively prove the following  results:
\begin{itemize}
    \item Local rapid stabilization, namely, by adding a well-chosen control force one can enhance the dissipation of the system to make it as fast as we want provided that the initial state is not large enough.
    \item Obstruction to uniform asymptotic stabilization in $\mathbf{H}(2\pi)$.
\end{itemize}

Concerning energy  stabilization problems, we search for continuous time-varying (both local and non-local) feedback laws
\begin{align}\label{feedbacklaw-def}
\begin{array}{ccc}
    F: \mathbb{R}\times H^1(\mathbb{T}^1)&\rightarrow& L^2(\mathbb{T}^1)\\
    (t; u)&\mapsto&  F(t; u)
\end{array}
\end{align}
to enhance the stability  of the original system. Namely, what is the asymptotic behavior of the following closed-loop system
\begin{equation}\label{closed-loop-timevarying}
     \partial_t u(t,x)- \Delta u(t,x)=  |u_{ x}|^2(t,x) u(t,x)+  \mathbf{1}_{\omega} (F(t; u(t, \cdot)))^{u(t,x)^{\perp}} (t, x)?
\end{equation}
 The following natural definition is borrowed from \cite{Coron-Krieger-Xiang-1, 2017-Coron-Rivas-Xiang-APDE} (also refer to the monograph \cite[Definition 24.2, page 97]{1967-Hahn-book}).
\begin{defn}\label{DEF-usstab}
The system \eqref{closed-loop-timevarying} is called uniformly asymptotically stable in the energy level set $\textbf{H}(q)$ if there exists a $\mathcal{KL}$ function $h$, i.e. a continuous function $h: \mathbb{R}^+\times \mathbb{R}^+ \rightarrow \mathbb{R}^+$ satisfying
\begin{gather*}
  \textrm{for any } t \in [0, +\infty),\;   h(\cdot,t) \text{ is strictly increasing and vanishes at $0$,}
\\
  \textrm{for any } s \in [0, +\infty),\;  h(s,\cdot) \text{ is decreasing and } \lim_{t\rightarrow +\infty}h(s,t)=0,
\end{gather*}
 such that the energy decays uniformly as follows:
\begin{equation}\label{estEphit}
    E(u(t))\leq h(E(u(0)),t)  \; \forall t\in (0, +\infty),\;   \forall u(0) \in \textbf{H}(q).
\end{equation}
\end{defn}

Due to  Corollary \ref{thm-expo-decay}  the system is locally uniformly exponentially stable without adding any special feedback law. However, the exponential decay rate is bounded. To achieve better stability properties such as rapid stability  one relies on extra forces, namely feedback laws. The following result shows that with the help of some well-designed explicit feedback laws the system can decay as fast as we want.
\begin{thm}[Quantitative local rapid stabilization]\label{thm-rapid-stab}
There exists some effectively computable constant \footnote{Throughout this paper we use the expression \textit{effectively computable} to describe a number that can be explicitly calculated, and the notation $a\lesssim b$ to indicate that $a\leq C b$ where $C$ is some effectively computable constant.}  $C>0$  such that for any $\lambda>1$ and for any point $q\in \mathbb{S}^k$, one can design an explicit time-independent feedback law
\begin{align*}\label{feedbacklaw-rapid}
    F= F(\lambda, q):  H^1(\mathbb{T}^1)&\rightarrow L^2(\mathbb{T}^1), \\
    u &\mapsto  F(u)= F(\lambda, q)(u),
\end{align*}
such that for any initial state $u_0\in H^1(\mathbb{T}^1; \mathbb{S}^k)$ satisfying $\|u_0- q\|_{H^1}\leq e^{-C \sqrt{\lambda}}$
the unique solution of the system
  \begin{equation*}
    \partial_t u- \Delta u- |\partial_x u|^2 u=   \mathbf{1}_{\omega} (F u)^{u^{\perp}},
\end{equation*}
decays exponentially
\begin{equation*}
   \|u(t)\|_{\dot H^1}\leq e^{C \sqrt{\lambda}} e^{-\lambda t}  \|u_0\|_{\dot H^1} \; \forall t\in (0, +\infty)
\end{equation*}
where
$$
\|v\|_{\dot H^1}:= \sqrt{E(v)}.
$$
\end{thm}
The proof of this theorem is based on the stereographic projection as well as the frequency Lyapunov method introduced  by the second author in \cite{Xiang-heat-2020, Xiang-NS-2020}. We refer to Section \ref{sec:local-control-stabilization}, in particular to  Lemma \ref{lem-rapid-stab} and Remark \ref{rem:rapid-lem-thm}, for details.
Moreover, as a direct consequence of this quantitative rapid stabilization result, see Proposition \ref{prop:nullcontrol} for details, one further obtains the following small-time local controllability result.

\begin{thm}[Quantitative local null controllability]\label{thm:nullcontrol}
The controlled harmonic map heat flow equation
    \begin{equation}
    \begin{cases}
    \partial_t u- \Delta u- |\partial_x u|^2 u= \mathbf{1}_{\omega} f^{u^{\perp}}, \notag \\
    u(0, \cdot)= u_0(\cdot),
    \end{cases}
\end{equation}
is locally null controllable in small time in the sense that, there exists an effectively computable constant $C>1$ such that for any $T\in (0, 1)$, for any initial state $u_0\in H^1(\mathbb{T}^1; \mathbb{S}^k)$ and any point $p\in \mathbb{S}^k$ satisfying
\begin{equation*}
    \|u_0- p\|_{H^1(\mathbb{T}^1)}\leq e^{-\frac{C}{T}},
\end{equation*}
we can construct an explicit control $f\in L^{\infty}(0, T; L^2(\mathbb{T}^1))$ satisfying
\begin{equation}
\|f\|_{L^{\infty}(0, T; L^2(\mathbb{T}^1))}\leq e^{\frac{C}{T}}\|u_0- p\|_{H^1(\mathbb{T}^1)} \notag
\end{equation}
such that the unique solution $u\in C([0, T]; H^1(\mathbb{T}^1))\cap L^2(0, T; H^2(\mathbb{T}^1))$ of the Cauchy problem satisfies $u(T, \cdot)= p.$
\end{thm}

\begin{figure}[t]
\centering
\includegraphics[width=0.8\linewidth, trim={0cm 0.0cm 0cm 0.0cm},clip]{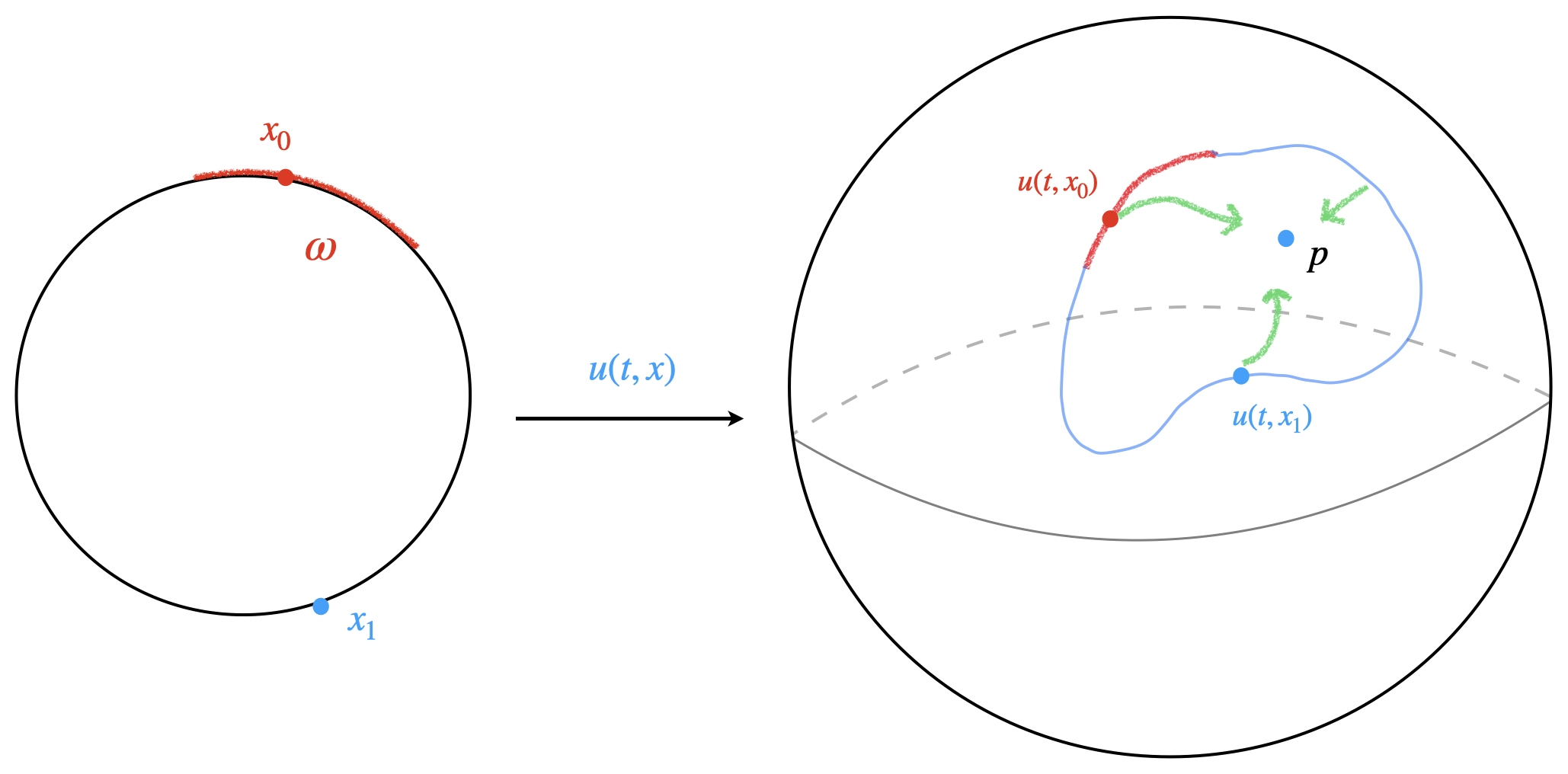}
\caption{The local quantitative rapid stabilization and local null controllability of the equation around some given point $p\in \mathbb{S}^2$. See Theorem \ref{thm-rapid-stab} and Theorem \ref{thm:nullcontrol} for details.  \protect  \\
This picture is also related to \hyperref[Step5]{Step 5} of Section \ref{sec:strategy}.
}
\label{pic4}
\end{figure}

\begin{remark}\label{rmk:Liu}
 The local controllability of the nonlinear heat equations is extensively studied in the literature, we refer to \cite{2016-Coron-Nguyen-ARMA,  1971-Fattorini-Russell-ARMA,  1995-Guo-Littman, MR4153111, 1977-Jones-Frank-JMAA,  1978-Littman-ASNSP,  2014-Martin-Rosier-Rouchon-A,  2016-Martin-Rosier-Rouchon-SICON} for the one-dimensional case, as well as to   \cite{DZZ-2008, EZ-2011, MR1791879, MR1406566, Lebeau-Robbiano-CPDE} for the multi-dimensional case.  However, the study of controllability of such a geometric heat equation is extremely limited, we have only found the recent paper by Liu \cite{Liu-2018} in the literature. In \cite[Proposition 2.6]{Liu-2018} Liu has proved the local controllability of the harmonic map heat flow equation from $\Omega\subset \mathbb{R}^3$ to $\mathbb{S}^2$, his proof is based on the stereographic projection, global Carleman estimates \cite{MR1406566, MR2224822}, and an iteration scheme \cite{MR3023058} due to Liu-Takahashi-Tucsnak. Compared with  \cite[Proposition 2.6]{Liu-2018}, our result highlights the advantage of presenting the quantitative cost estimates with respect to time.
\end{remark}

Again, due to Corollary \ref{thm-expo-decay}, one may ask for global stabilization. Inspired by the topological argument introduced in \cite{Coron-Krieger-Xiang-1}, this global stabilization does not hold even if the controlled area is the whole domain, $i. e. \; \omega= \mathbb{T}^1$, as shown by the following proposition.
\begin{thm}[Obstruction to uniform asymptotic stabilization in $\mathbf{H}(2\pi)$]\label{THM-nounifdecay}
  For any  time-varying feedback law F satisfying  conditions $(\mathcal{P}1)-(\mathcal{P}4)$ (we refer to  Section \ref{subsec:stabi:pre} for precise definitions),
    the closed-loop system \eqref{closed-loop-timevarying} is not uniformly asymptotically stable in
    $\textbf{H}(2\pi)$.
\end{thm}
This obstruction is due to the degree theory, which, roughly speaking, comes from the non-triviality of the homology group $H_k(\Sph^k)$. Actually,  the same arguments
 introduced in \cite[Section 3.2]{Coron-Krieger-Xiang-1} can be adapted to our framework.  For readers' convenience we sketch this proof here.
\begin{proof}[Proof of Theorem \ref{THM-nounifdecay}]
Case $k=1$ is trivial since  the degree of  $u(t, \cdot): \mathbb{T}^1\rightarrow \Sph^1$ does not depend  time. In the following we concentrate on the case $k\geq 2$. Thanks to \cite[Lemma 3.2 as well as equation (32)]{Coron-Krieger-Xiang-1}, there are  $A_{k-1}: \left(\mathbb{T}^1\right)^k\rightarrow \mathbb{S}^k$ as well as $\gamma_{k-1}: \left(\mathbb{T}^1\right)^{k-1}\rightarrow H^1(\mathbb{T}^1; \mathbb{S}^k)$ with
\begin{equation}
\gamma_{k-1}(s_1,\ldots , s_{k-1})(x):= A_{k-1}(s_1, s_2,\ldots , s_{k-1}, x) \; \forall s_1,\ldots , s_{k-1}, x \in \mathbb{T}^1  \notag
\end{equation}
 such that
\begin{equation}
\deg A_{k-1}= 2^{k-1}    \notag
\end{equation}
and that
\begin{equation}
E(\gamma_{k-1}(s_1,\ldots , s_{k-1})(\cdot))= 2\pi \sin{s_1}^2\ldots \sin{s_{k-1}}^2 \leq 2\pi\; \forall s_1,\ldots , s_{k-1}\in \mathbb{T}^1 .    \notag
\end{equation}
For instance, $A_1:  \left(\mathbb{T}^1\right)^2\rightarrow \mathbb{S}^2$ is defined as follows: (see Figure \ref{picdegree})
\begin{equation*}
A_1(s, x):=
\begin{cases}
\left(\sin s \cos x, \sin s \sin x, \cos s \right)^T \; \forall s\in [0, \pi], \forall x\in \mathbb{T}^1, \\
\left(-\sin s \cos x, \sin s \sin x, \cos s \right)^T \; \forall s\in (\pi, 2\pi), \forall x\in \mathbb{T}^1,
\end{cases}
\end{equation*}
and for $k\geq 2$ the map is successively constructed as follows, $A_k:  \left(\mathbb{T}^1\right)^{k+1}\rightarrow \mathbb{S}^{k+1}$ is defined as follows: (see Figure \ref{picdegree})
\begin{equation*}
A_k(s_1,..., s_k, x):=
\begin{cases}
\left(\sin s_1 A_{k-1}(s_2,...,s_k, x)^T, \cos s_1 \right)^T \; \forall s_1\in [0, \pi], \forall (s_2,..., s_k, x)\in \left(\mathbb{T}^1\right)^k, \\
\left(-\sin s_1 A_{k-1}(s_2,...,s_k, x)^T, \cos s_1 \right)^T \; \forall s_1\in [0, \pi], \forall (s_2,..., s_k, x)\in \left(\mathbb{T}^1\right)^k.
\end{cases}
\end{equation*}

Define the flow of the closed-loop system:
\begin{equation*}
\begin{array}{ccc}
    \Phi: \mathbb{R} \times H^1(\mathbb{T}^1;\mathbb{S}^k)&\rightarrow & H^1(\mathbb{T}^1;\mathbb{S}^k)\\
    (t, u_0)&\mapsto& u(t),
\end{array}
\end{equation*}
where $u(t)$ is the unique solution of equation \eqref{closed-loop-timevarying} with the initial state $u(0)= u_0$.
Suppose that the closed-loop system is uniformly asymptotically stable. Then, for any $\delta>0$,  there exists $T>0$ such that  $E(\Phi(T, \gamma_{k-1}(s_1,\ldots , s_{k-1})))\leq \delta \; \forall s_1,\ldots , s_{k_1}\in \mathbb{T}^1$. Thanks to the continuity   of the flow associated to the closed-loop system, for which we refer to Proposition \ref{lem:wellclosegene},
\begin{equation*}
\Phi(t, \gamma_{k-1}(s_1,\ldots , s_{k-1}))(x)\in C^0\left(\mathbb{T}^1_{s_1}\times\ldots \times \mathbb{T}^1_{s_{k-1}}; C^0([0, T]; C^0(\mathbb{T}^1; \mathbb{S}^{k}) \right).
\end{equation*}
\begin{figure}[t]
\centering
\includegraphics[width=0.99\linewidth, trim={0cm 0.0cm 0cm 0.0cm},clip]{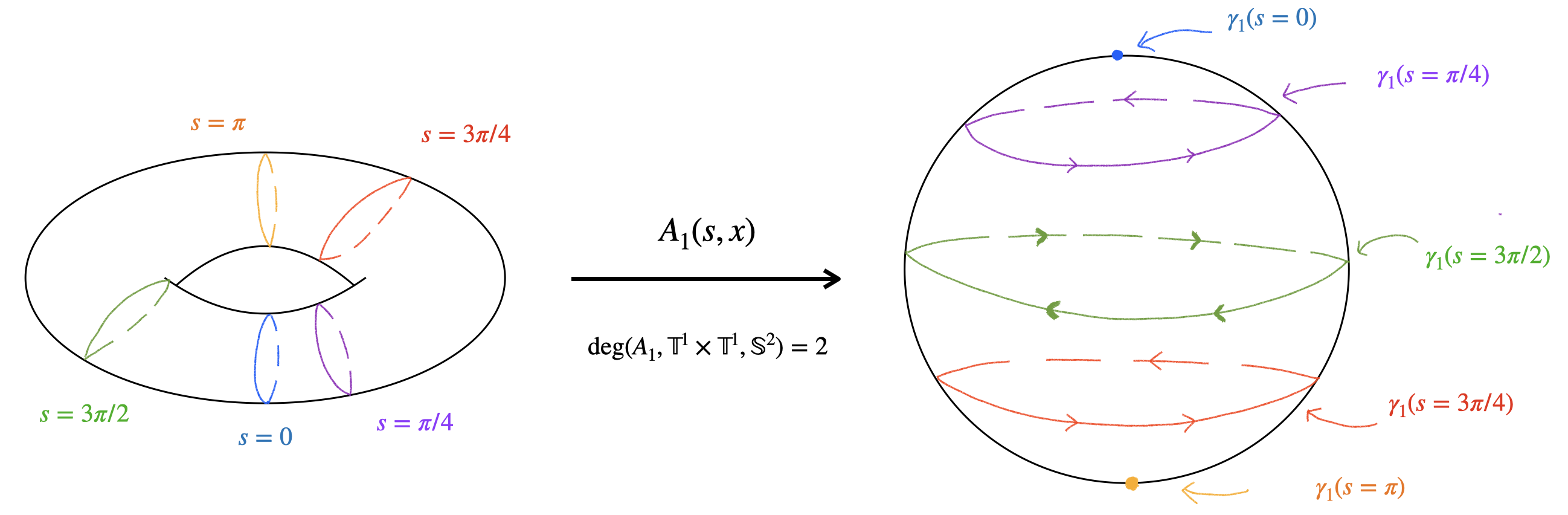}
\caption{Obstruction to uniform asymptotic stabilization. The constructed map $A_1: \mathbb{T}^1\times \mathbb{T}^1\rightarrow \mathbb{S}^2$ has nontrivial degree. This picture is  related to Theorem \ref{THM-nounifdecay}.
}
\label{picdegree}
\end{figure}

Then we define  $ G_1: [0, T+1]\times (\mathbb{T}^1)^{k-1}\times \mathbb{T}^1_x\rightarrow \mathbb{S}^{k}$ as
\begin{equation}
G_1(t; s_1,\ldots , s_{k_1}, x):=
\begin{cases}
 \Phi(t, \gamma_{k-1}(s_1,\ldots , s_{k-1}))(x)\; \forall t\in [0, T]\\
 \frac{(1-t+T)  \mathcal{H}_1(T; s_1, \ldots , s_{k-1}, x)+ (t- T) a(s)}{|(1-t+T)  \mathcal{H}_1(T; s_1, \ldots , s_{k-1}, x)+ (t- T) a(s_1, \ldots , s_{k-1})|} \; \forall t\in [T, T+1],
 \end{cases}   \notag
\end{equation}
where,  $a(s_1, \ldots , s_{k-1}):= \Phi(T, \gamma_{k-1}(s_1,\ldots , s_{k-1}))(0)$.
Note that, choosing $\delta >0$ small enough, $G_1: [0, T+1]\times (\mathbb{T}^1)^{k-1}\times \mathbb{T}^1_x\rightarrow \mathbb{S}^{k}$ is continuous. Hence,
 $ G_1(t=0): (\mathbb{T}^1)^{k-1}\times \mathbb{T}^1_x\rightarrow \Sph^k$ is homotopic to
$ G_1(t=T+1): (\mathbb{T}^1)^{k-1}\times \mathbb{T}^1_x\rightarrow \Sph^k$.  However,
\begin{align*}
\deg G_1(t=0)&= \deg A_{k-1}= 2^{k-1}, \\
\deg G_1(t=T+1)&= 0 \textrm{ as its value does not depend on $x$}.
\end{align*}
This leads to a contradiction and finishes the proof of Theorem \ref{THM-nounifdecay}.
\end{proof}

\subsubsection{\textbf{Small-time global exact controllability between harmonic maps and an open problem}}
Global exact controllability is one of the most difficult problems in control theory. For example the Lions problem on the global controllability of the Navier-Stokes equations with the no-slip Stokes condition and the optimal time for the global controllability of the semilinear wave equations are widely open.
In the case that (small-time) global (exact) controllability fails or is unknown, one continue to ask whether it is possible to deform between stationary states and whether it is possible to do so in a short time.

Actually, this natural  problem is still widely open for many important models, for instance even for the simplest one-dimensional nonlinear heat equation we do not know how to move between any stationary states in small time. More precisely, let us consider the following controlled system
\begin{equation*}
\begin{cases}
u_t- u_{xx}- u^3= 0, \\
u(t, 0)= 0,\\
u(t, L)= f(t),\\
u(0, \cdot)= u_0(\cdot),
\end{cases}
\end{equation*}
whose stationary states set is defined as
\begin{equation*}
\mathcal{S}_{H3}:= \{u\in C^2([0, L]): u_{xx}+ u^3= 0 \; \forall x\in [0, L], u(0)= 0\}.
\end{equation*}
The controllability  between stationary states in small time is still \textit{largely open} for this model. The main difficulty is that we have to deform the state as fast as we want,  otherwise the weaker property of controllability between stationary states in large time has been known for decades \cite{Coron-Trelat-SICON-2004}.  We refer to \cite[Open Problem 2]{Coron-CM-2007} for background on this important open problem.
\begin{figure}[t]
\centering
\includegraphics[width=0.9\linewidth, trim={0cm 0.0cm 0cm 0.0cm},clip]{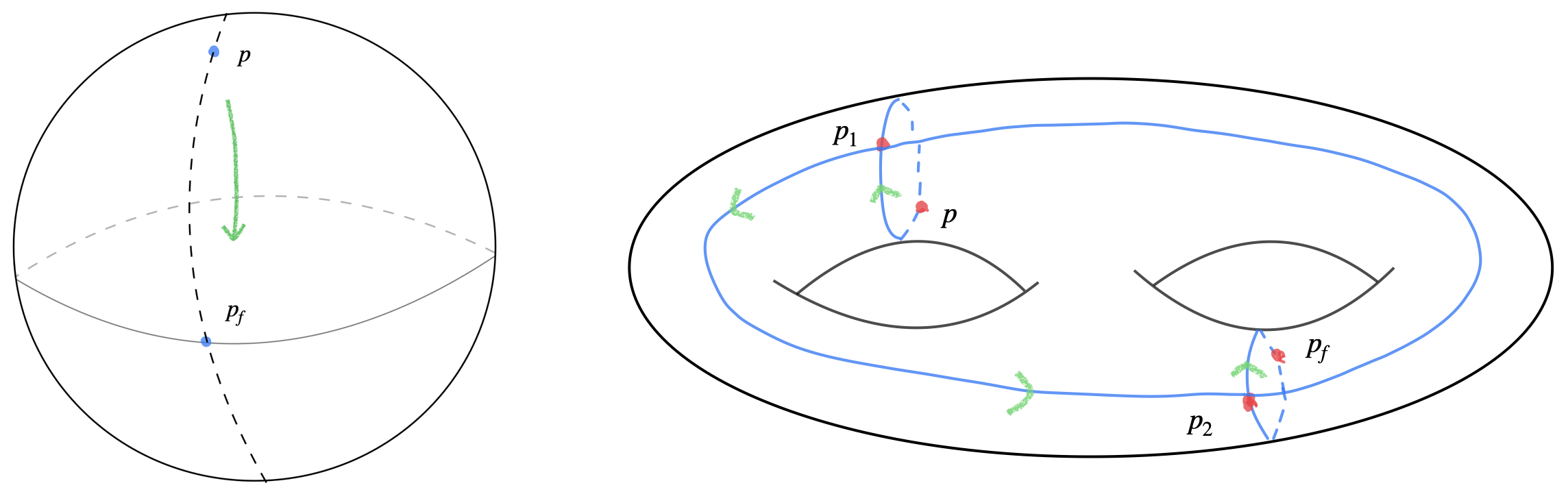}
\caption{Small-time global controllability on closed geodesics: the picture on the left shows that how to deform from a constant state $p$ to another constant state $p_f$ along geodesics on sphere. See Theorem \ref{lem:excconhm0} for details.  This picture is also related to \hyperref[Step6]{Step 6} of Section \ref{sec:strategy}.
 \protect  \\
 For general manifold targets $\mathcal{N}$, there may be no closed  geodesic that contains both $p$ and $p_f$.
 The picture on the right shows that we can  deform from a constant state $p$ to another constant state $p_f$  by steps along different closed geodesics. Alternatively, we can always find a complete geodesic, which is not necessarily a closed geodesic, that contains both $p$ and $p_f$. The idea of deforming on complete geodesics is used in the proof of Lemma \ref{heat-geodesic}. }
\label{pic5}
\end{figure}
\begin{figure}[t]
\centering
\includegraphics[width=0.99\linewidth, trim={0cm 0.0cm 0cm 0.0cm},clip]{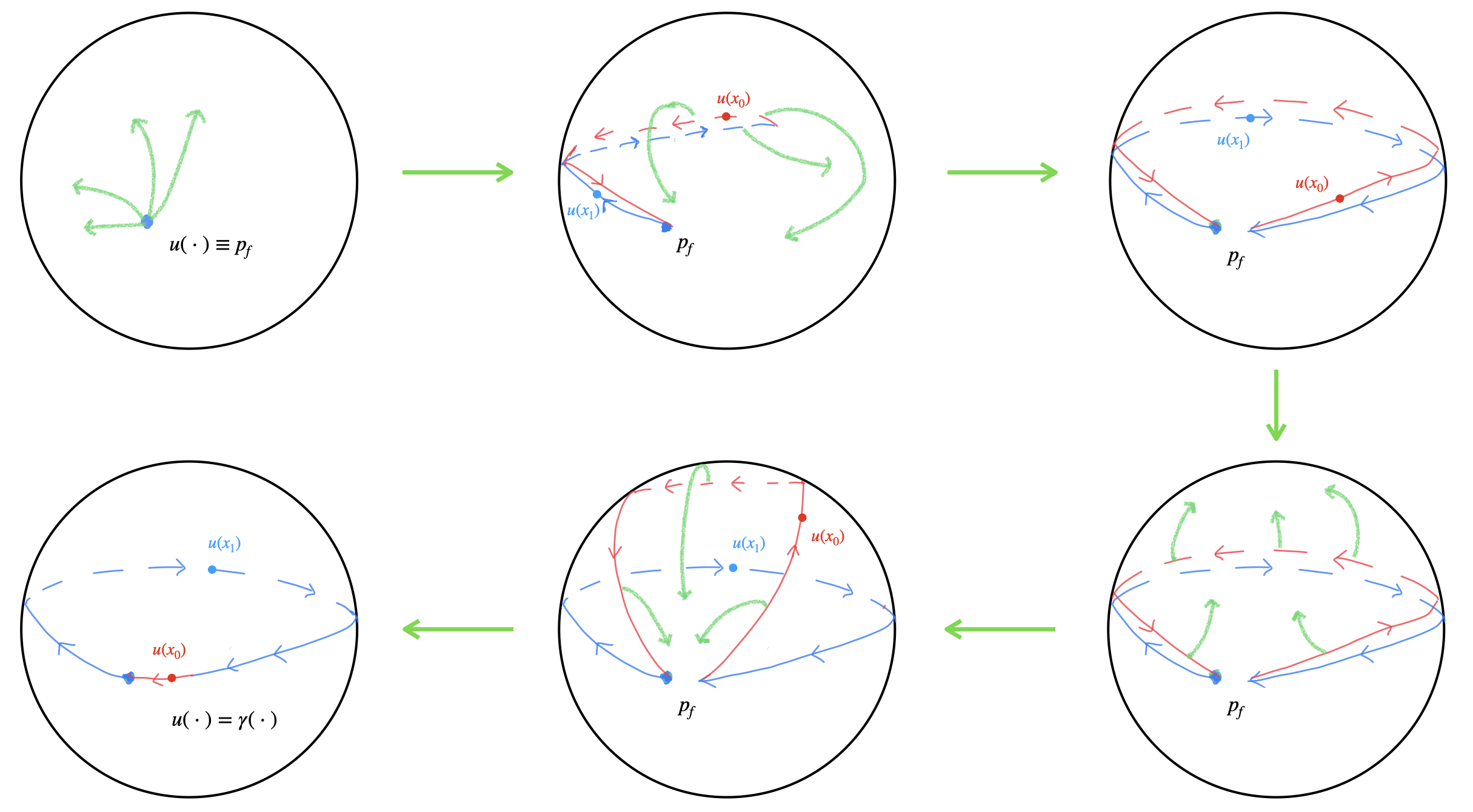}
\caption{Small-time global controllability between harmonic maps having different energy: deform from the constant state $p_f$ to a harmonic map $\gamma$ that crosses $p_f$. This is the simplest case of Theorem \ref{lem:excconhm}.  In these pictures, the blue colored curve is the value of $u(x)|_{x\in \mathbb{T}^1\setminus \omega}$, the red colored curve is the value of $u(x)|_{x\in \omega}$, and the green colored arrows represent the direction of the flow along with time. We observe that on the domain $\omega$ the speed of the curve $\dot{u}(x)$ is faster than the speed on $\mathbb{T}^1\setminus \omega$.
This picture is also related to \hyperref[Step7]{Step 7} of Section \ref{sec:strategy}.
}
\label{pic6}
\end{figure}

\begin{figure}[t]
\centering
\includegraphics[width=0.943\linewidth, trim={0cm 0.7cm 0.6cm 0cm},clip]{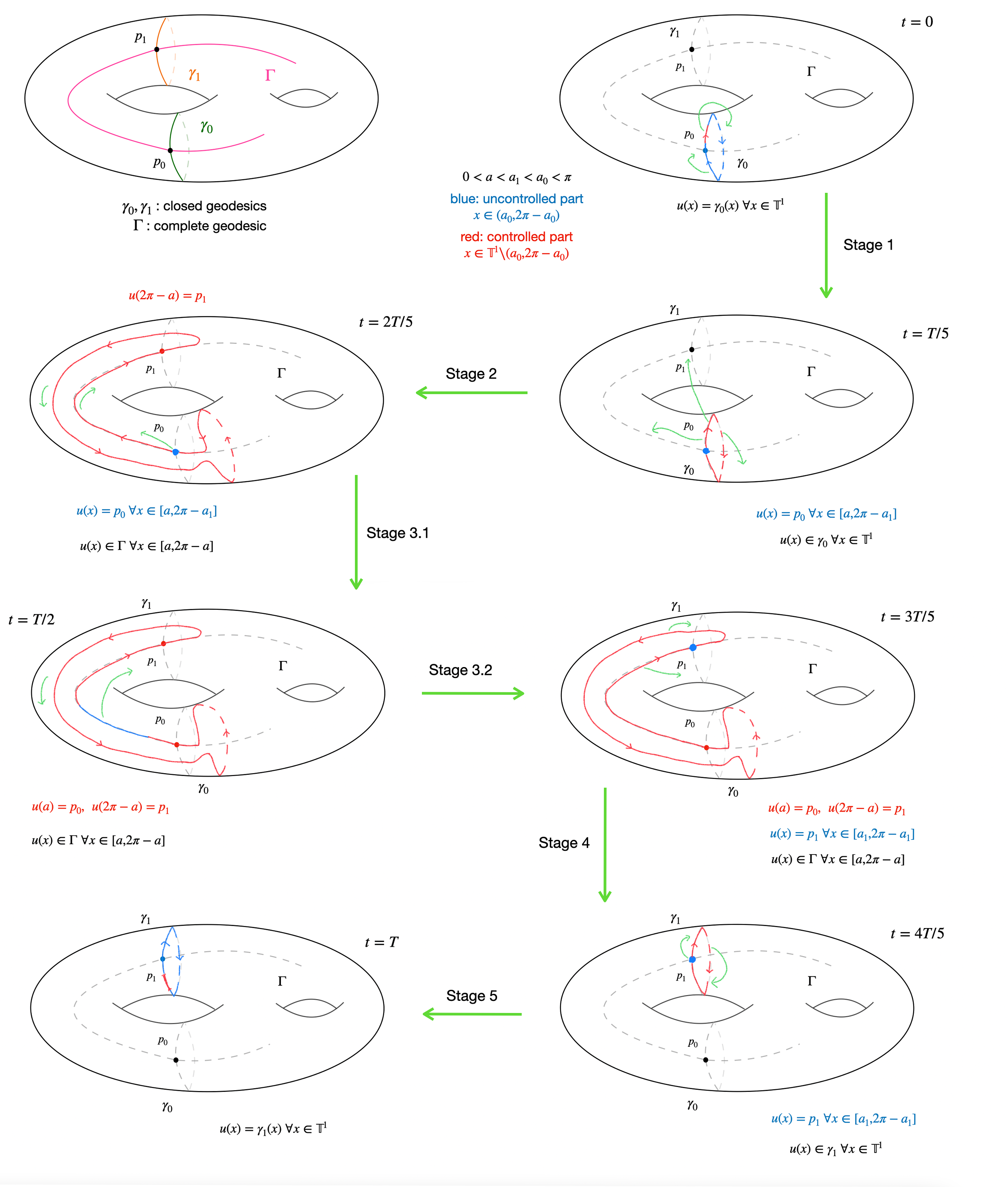}
\caption{Small-time global controllability between harmonic maps are that homotopic on general compact Riemannian manifold, Theorem \ref{th-small-time-global-harmonic} : deform from $\gamma_0$ to $\gamma_1$. The blue, red and green curves have the same meaning as in Figure \ref{pic6}.
Stage 1 and Stage 5 are based on the idea of  Theorem \ref{lem:excconhm0} to control the state on closed geodesics $\gamma_0$ and $\gamma_1$.
Stage 2 and Stage 4 are based on the idea of  Theorem \ref{lem:excconhm} to deform the state on $\mathcal{N}$ without moving the uncontrolled part.  See Figure \ref{pic6} for such a deformation on $\mathbb{S}^2$.
Stage 3 adapts  Lemma \ref{heat-geodesic} $(ii)$ and Lemma  \ref{lem:glue} to move the uncontrolled part on a given complete geodesic $\Gamma$.  }
\label{picma}
\end{figure}

Surprisingly, this global controllability in small time between stationary states has a positive answer for the harmonic map heat flow equation, thanks to geometric properties of the system.  We state three  theorems (Theorems \ref{lem:excconhm0} to \ref{th-small-time-global-harmonic}) in this direction and prove them in  Section \ref{sec:GECHM}. This new observation also works for other geometric equations with compact Riemannian target, such as the wave maps equations.  \\

\noindent (i)  \textit{Small-time global controllability on closed geodesics}

We first investigate the simplest case, namely, the harmonic map heat flow $\mathbb{T}^1\rightarrow \Sph^1$, for which, benefiting from the polar coordinates, the system is transformed into the linear heat equation with an interior control.
Next, we show that for a sphere $\Sph^k$, if the initial state is in a given closed geodesic and the control force is tangent to this geodesic, then the unique solution remains in the same closed geodesic. Moreover, under the polar coordinates, the system becomes again the controlled linear heat equation with an interior control, for which the  controllability is well-investigated.

\begin{thm}\label{lem:excconhm0}
Let $T>0$.  Let the curve $\mathcal{C}= \{\gamma(x)\in \mathbb{S}^k: x\in \mathbb{T}^1\}$ be a closed constant-speed geodesic on $\mathbb{S}^k$ with energy $2\pi$:
\begin{gather*}
   \gamma(x)\in  \mathbb{S}^k,\;
    \gamma_{xx}(x)+ |\gamma_x|^2(x) \gamma(x)= 0 \; \forall x\in \mathbb{T}^1,
   \\
   \int_{\mathbb{T}^1}|\gamma_x(x)|^2dx=2\pi.
\end{gather*}
 Let $k\geq 1$ and $N\in \mathbb{N}$.
For any initial state  $v_0(\cdot)\in H^1(\mathbb{T}^1; \mathbb{S}^k)$ with values into  $\mathcal{C}$ and satisfying
\begin{equation*}
    \deg (v_0, \mathbb{T}^1, \mathcal{C})= N,
\end{equation*}
and for any $r\in \mathbb{R}$,
there exists an explicit function $f\in L^{\infty}(0, T; L^2(\mathbb{T}^1; \mathbb{R}^{k+1}))$ such that the unique solution $u(t, x)|_{t\in (0, T), x\in \mathbb{T}^1}$ of the controlled harmonic map heat flow equation \eqref{eq:controlhmhf} with initial state $v_0$ satisfies
\begin{equation*}
    u(T, x)= \gamma(Nx+ r)\; \forall x\in \mathbb{T}^1.
\end{equation*}
\end{thm}
The proof of this theorem is based on the writing of the harmonic map heat flow in polar coordinates and the controllability of the 1-D linear heat equation. It can be found in  Section \ref{sec:GECHM}.\\

\noindent (ii)  \textit{Small-time global  controllability between harmonic maps having different energy}

\begin{thm}\label{lem:excconhm}
Let $T>0$.  Let the curve $\mathcal{C}= \{\gamma(x)\in \Sph^k: x\in \mathbb{T}^1\}$ be a closed constant-speed geodesic on $\mathbb{S}^k$ with energy $2\pi$:
\begin{gather*}
   \gamma(x)\in  \mathbb{S}^k,\;
    \gamma_{xx}(x)+ |\gamma_x|^2(x) \gamma(x)= 0 \; \forall x\in \mathbb{T}^1,
\\
 \int_{\mathbb{T}^1}|\gamma_x(x)|^2dx=2\pi.
\end{gather*}
Let $k\geq 2$.
Let   $v_0(\cdot)\in H^1(\mathbb{T}^1; \mathbb{S}^k)$ be an initial state and  $v_1(\cdot)\in H^1(\mathbb{T}^1; \mathbb{S}^k)$ be a target state both with values into  $\mathcal{C}$. Assume that there exists $N_1\in \mathbb{Z}$ and $r\in \mathbb{R}$ such that
\begin{gather*}
    v_1(x)= \gamma(N_1 x+ r) \; \forall x\in \mathbb{T}^1.
\end{gather*}
 Then, there exists an explicit function $f\in L^{\infty}(0, T; L^2(\mathbb{T}^1))$  such that the unique solution $u(t, x)|_{t\in (0, T), x\in \mathbb{T}^1}$ of the controlled harmonic map heat flow equation \eqref{eq:controlhmhf} with initial state $v_0$ satisfies
\begin{equation*}
    u(T, x)= v_1(x)\; \forall x\in \mathbb{T}^1.
\end{equation*}
\end{thm}

\noindent (iii)  \textit{General compact Riemannian targets}

Finally one can replace $\mathbb{S}^k$-target by a general compact Riemannian submanifold $\mathcal{N}$ of $\R^m$.
Indeed, one has the following general result on the controlled harmonic maps heat flow,
\begin{equation}
    u_t- u_{xx}= \beta_u (u_x, u_x)+  \mathbf{1}_{\omega} f^{u^{\perp}},
\end{equation}
 where $f^{u^{\perp}}$  represents now the projection of $f$ onto the tangent space $T_u \mathcal{N}\subset \R^m$.

\begin{thm}
\label{th-small-time-global-harmonic}
Let $T>0$. Let $\gamma_0:\mathbb{T}^1\rightarrow \mathcal{N}$ and $\gamma_1:\mathbb{T}^1\rightarrow \mathcal{N}$ be two harmonic maps which are in the same connected component of $C^0(\mathbb{T}^1; \mathcal{N})$. Then there exists a control $f\in L^{\infty}(0, T; L^2(\mathbb{T}^1))$ which steers the harmonic map heat flow  from $\gamma_0$ to $\gamma_1$ during the time interval $[0,T]$.
\end{thm}
We refer to Section \ref{sebsec:genR} and Figure \ref{picma} for the proof of this result.

\subsubsection{\textbf{Global controllability  of the harmonic map heat flow}}
 Finally we deal with the global controllability result. Let us recall that when $k=1$ the global controllability with $\mathbb{T}^1$-target is answered by Theorem \ref{lem:excconhm0}. For  $k\geq 2$, one has the following theorem.

\begin{thm}[Global controllability]\label{thm:globalcontrol}

Let $k\geq 2$. Let $M>0$. There exists some   constant $T= T(M)>0$ such that for any initial state $u_0\in H^1(\mathbb{T}^1; \mathbb{S}^k)$ belongs to $E(M)$, for  any target state $u_f\in H^1(\mathbb{T}^1; \mathbb{S}^k)$ that is a harmonic map, we can  construct an explicit control $f\in L^{\infty}(0, T; L^2(\mathbb{T}^1))$ such that the unique solution of the controlled harmonic map heat flow
    \begin{equation*}
    \begin{cases}
            \partial_t u- \Delta u- |\partial_x u|^2 u= \mathbf{1}_{\omega} f^{u^{\perp}}, \\
            u(0, \cdot)= u_0(\cdot),
  \end{cases}
\end{equation*}
 satisfies $u(T, \cdot)= u_f(\cdot)$.
\end{thm}

\begin{remark}
It is noteworthy that the time required to control the system only depends on the scale of the initial state, namely the energy of the target state can be as large as we want. This is the  consequence of the small-time global exact controllability between harmonic maps, which will be shown in Section \ref{sec:GECHM} Lemma \ref{lem:excconhm0}.  However, it is not clear whether we can obtain the stronger small-time global null controllability result.
\end{remark}

\subsection{The strategy for the proof of the global controllability result (Theorem  \ref{thm:globalcontrol}) }\label{sec:strategy}

Our proof is divided into several steps,  in which we benefit from different properties of the controlled harmonic map heat flow system, some of which are illustrated in the preceding context such as the convergence of the flow to harmonic maps, while the other properties will be explored later on.
\begin{itemize}
\item[\label{Step1} Step 1.]  {\it Benefit on the natural dissipation}.
\\
   In this step we do not add any specific control and let the system evolve itself for a period with length $T_1$ to be chosen later on
\begin{equation*}
    \begin{cases}
            \partial_t u- \Delta u- |\partial_x u|^2 u= 0, \; t\in (0, T_1), \\
            u(t= 0, \cdot)= u_0(\cdot).
  \end{cases}
\end{equation*}
Suppose that the initial state $ u_0(\cdot)$ belongs to  $ \mathbf{H}(J_0 2\pi- \delta_0)$, where $J_0$ is some integer. Due to the asymptotic property Proposition \ref{thm:homopotyconver}, after some time of evolution $T_1$ the solution becomes a ``$\varepsilon$-approximate harmonic map" and is therefore close to a harmonic map $u_1$ satisfying $E(u_1)= J_1 2\pi$ for some integer  $J_1$ that is strictly smaller than $ J_0$.   This proposition is proved in Section  \ref{sec:globalconvergence}.  See the  picture on the left of Figure \ref{pic2}.
\item[\label{Step2}  Step 2.]   {\it Across the critical energy level set}. \\ Because the harmonic map $u_1$ has energy $J_1 2\pi$, in general for the state $u(t= T_1)$ sufficiently  close to  $u_1$  the energy of the flow  can not pass the energy level set  $J_1 2\pi$ without any extra control force. In the case that $u_1$ is non-trivial, inspired by the recent work on wave maps \cite{Coron-Krieger-Xiang-1} and based on the power series expansion method,  we construct an explicit control $f(t)|_{t\in [T_1, T_1+ \tau_1]}= f_1(t)|_{t\in [T_1, T_1+ \tau_1]}$ with $\tau_1$  small such that the solution of
    \begin{equation*}
            \partial_t u- \Delta u- |\partial_x u|^2 u= \mathbf{1}_{\omega} f_1^{u^{\perp}},  \; t\in (T_1, T_1+\tau_1),
\end{equation*}
which is equal to $u(t= T_1)$ at time $T_1$ satisfies $E(u(t= T_1+ \tau_1))< E(u_1)$.
Thus $u(t= T_1+ \tau_1)\in \mathbf{H}((J_1) 2\pi- \delta_1)$ for some $\delta_1>0$.  This construction is provided by  Proposition \ref{Prop:pse} which will be shown in Section \ref{subsec:pse}. Otherwise, there is no need to add other specific control. See Figure \ref{pic3}.
 \item[\label{Step3} Step 3.]  {\it Iterate the preceding two-step procedure}. \\  We iterate the process, which yields
 \begin{gather*}
 u(t= T_i) \textrm{ is close to the harmonic map } u_i,\\
  \textrm{ where } E(u_i)= J_i 2\pi, \textrm{ for } i\in \{1, 2,\ldots , K\}, \\
  J_0> J_1>J_2>\ldots > J_K= 0,
 \end{gather*}
 until at time $t= T_K$ such that $u(t= T_K)$ is close to a trivial harmonic map $u_K$.  We shall keep in mind that the values of $T_1, T_2,\ldots , T_K$ may become large but only depend on the value of $E(u_0)$.
  \item[\label{Step4} Step 4.] {\it Exponential decay of the energy}.  \\ Again we do not add any specific control.
  Since the state is sufficiently small, thanks to Corollary \ref{thm-expo-decay}  the energy decreases exponentially with rate $1/(2\pi^2)$, for any $\varepsilon>0$ there exists a  certain time $t= T_K+  t_d$ such that
 \begin{equation}
 \|u(t= T_K+ t_d)- p\|_{H^1}\leq \varepsilon,  \notag
\end{equation}
for some $p\in \mathbb{S}^k$. See Figure \ref{pic2} (the  picture on the right).
 \item[\label{Step5} Step 5.]   {\it Local null controllability}. \\ We adapt the local null controllability result, Theorem  \ref{thm:nullcontrol}, to find some control $$f(t)|_{t\in (T_K+ t_d, T_K+ t_d+ \tilde T)}= \tilde f(t)|_{t\in (T_K+ t_d, T_K+ t_d+ \tilde T)}$$ that steers the state from $u(t= T_K+ t_d)$ to $u(t= T_K+ t_d+ \tilde T)= p$,
\begin{equation*}
            \partial_t u- \Delta u- |\partial_x u|^2 u= \mathbf{1}_{\omega} \tilde f^{u^{\perp}},  \; t\in (T_K+ t_d, T_K+ t_d+\tilde T).
\end{equation*}
The proof of Theorem  \ref{thm:nullcontrol} is shown in Section \ref{sec:local-control-stabilization}, which is further divided by three stages: in the first stage we obtain quantitative rapid stabilization result Lemma \ref{lem-rapid-stab}, namely for any given decay rate $\lambda$ we construct an explicit feedback control such that the energy exponential dissipation of is enhanced to this decay rate, then in the second stage we  cut the time interval by infinite many pieces $[0, T]= \bigcup_{n\in \mathbb{N}}[T_n, T_{n+1}]$ and adapt an infinite-step iteration scheme, finally using quantitative estimates we show that the unique solution converges to 0 at time $T$  and  get the null controllability result.
See Figure \ref{pic4}.
 \item[\label{Step6} Step 6.] {\it Move from any constant state to any other constant state}. \\ Suppose that the final state is a harmonic map $u_f$, and
  that a point $p_f$ belongs to $u_f$.  We can find   some control in a short  period  $f(t)|_{t\in (T_K+ t_d+ \tilde T, T_K+ t_d+ \tilde T+ \tilde T_1)}= \tilde f(t)|_{t\in (T_K+ t_d+ \tilde T, T_K+ t_d+ \tilde T+ \tilde T_1)}$  to drive the state from $p$ to $p_f$. This task can be achieved via two different approaches. \\  Due to the  local controllability result around points, Theorem  \ref{thm:nullcontrol},  we  find a sequence of points $\{p_i: i= 0, 1, \ldots , n\}$ satisfying $p_0= p$ and $p_n= p_f$ such that one can move the solution from $p_i$ to $p_{i+1}$ in a short time for each $i$.\\
  Alternatively, we define a closed geodesic $\mathcal{C}$ (an equator) that connects  $p$ and $p_f$. By restricting the evolution of the system on this equator, which is possible if the initial state belongs the equator and if the control stays in the subtangent space $T\mathcal{C}\subset T\mathbb{R}^{k+1}$. According to Theorem \ref{lem:excconhm0} that will be proved in Section \ref{sec:GECHM},  this inhomogeneous geometric heat equation with  constraint becomes   the linear inhomogeneous heat equation without  constraint. Thus the state $p_f$ is a  reachable state from $p$.  Notice that by restricting the flow in the geodesic $\mathcal{C}$ the degree of the flow $u(t): \mathbb{T}^1\rightarrow \mathcal{C}$ cannot be modified with respect to time.  This gives a  necessary condition for two states that can be deformed from one to another. However,  one shall also note that this topological condition is not a  sufficient condition to find the reachable set due to the strong smoothing effect of the heat equation.  See Figure \ref{pic5}.
  \item[\label{Step7} Step 7.]  {\it Move from any constant state to any harmonic map that passes through this constant}. \\   Relying on the previous step based on  Theorem \ref{lem:excconhm0}  one can not change the degree of the flow if it is restricted to a closed geodesic. In this final step we show that, with the help of Theorem \ref{lem:excconhm} which is proved in Section \ref{sec:GECHM}, during a short period one can move the state from one harmonic map to another even if their degree do not coincide but $k\geq 2$.  This finishes the proof of Theorem \ref{thm:globalcontrol}. See Figure \ref{pic6}.
\end{itemize}

\subsection{The organization of the paper}
This paper is organized as follows. After the preliminary Section \ref{sec:prep}, all the proofs of technical theorems can be found in Sections  \ref{sec:GAC}--\ref{sec:GECHM}. More precisely,
\begin{itemize}
\item
Section \ref{sec:GAC} is devoted to the so-called ``global approximate controllability". More precisely, we first prove Proposition \ref{thm:homopotyconver}  and then show  Proposition \ref{Prop:pse}. This part is also related to \hyperref[Step1]{Step 1}--\hyperref[Step4]{Step 4}  for the proof of Theorem \ref{thm:globalcontrol}.
\item Next in Section \ref{sec:local-control-stabilization} we provide ``quantitative local controllability and stabilization" properties that is also devoted to \hyperref[Step5]{Step 5}.    It includes the proofs of the local null controllability Theorem \ref{thm:nullcontrol} as well as the local rapid stabilization \ref{thm-rapid-stab}.
\item Finally, we conclude the paper by presenting the ``global exact controllability between harmonic maps"  in Section \ref{sec:GECHM}. This section is also related to \hyperref[Step6]{Step 6} and \hyperref[Step7]{Step 7} that contains Theorem \ref{lem:excconhm0} and Theorem \ref{lem:excconhm}. Thus we also finish the proof of Theorem \ref{thm:globalcontrol}.
\end{itemize}
\begin{figure}[t]
\centering
\includegraphics[width=0.99\linewidth, trim={0cm 0.0cm 0cm 0.0cm},clip]{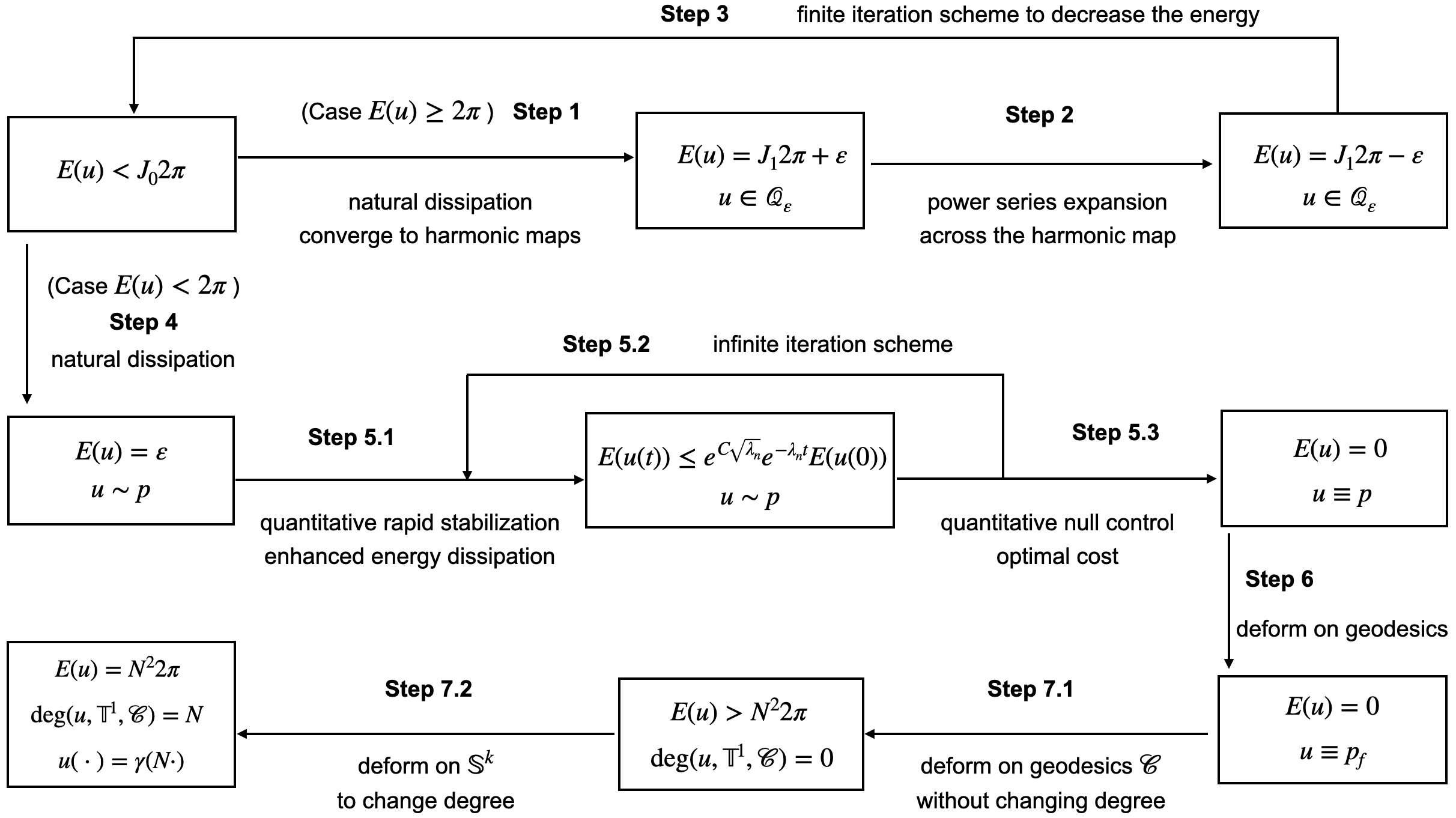}
\caption{The strategy for the proof of the global controllability to harmonic maps: Theorem  \ref{thm:globalcontrol}  \protect  \\
\hyperref[Step1]{Step 1} is the consequence of  Proposition \ref{thm:homopotyconver}, and is illustrated by  Figure \ref{pic2};   \protect \\\hyperref[Step2]{Step 2} is the result of  Proposition \ref{Prop:pse}, and is illustrated by  Figure \ref{pic3};    \protect \\
\hyperref[Step3]{Step 3} is the iteration of the preceding two steps;   \protect \\
\hyperref[Step4]{Step 4} is  the consequence of  Corollary \ref{thm-expo-decay}, and is again illustrated by  Figure \ref{pic2};   \protect \\\hyperref[Step5]{Step 5} is deduced from Theorem  \ref{thm:nullcontrol}, and we refer to Figure \ref{pic4};   \protect \\
\hyperref[Step6]{Step 6} is the characterized by  Theorem \ref{lem:excconhm0}, and we observe this deformation in  Figure \ref{pic5};   \protect \\\hyperref[Step7]{Step 7} is  shown  using  Theorem \ref{lem:excconhm}, and is illustrated by  Figure \ref{pic6}.
}
\label{picstrategy}
\end{figure}
\subsection{Some further questions}
We believe that there are many interesting follow-up questions to consider.  Here are some of them:
\begin{itemize}
    \item In Theorem \ref{thm:globalcontrol} we proved the global controllability of the harmonic map heat flow. However, this proof relies on the natural dissipation of the flow, so it is necessary to wait for a while to prepare the control.  It is known that using the return method introduced by the first author one prove the global controllability of some nonlinear partial differential equations in a short time, for example for the Euler equations \cite{Coron-Euler-1996, Glass-2000} or  the Navier-Stokes equations  \cite{Coron-Marbach-Sueur-Zhang}.  Is it possible to get small-time global controllability for this geometric model?
    \item In Theorem \ref{thm-rapid-stab}, thanks to the frequency Lyapunov method introduced by the second author \cite{Xiang-heat-2020, Xiang-NS-2020}, we obtained local rapid stabilization. Moreover, this quantitative rapid stabilization result also yields the small-time local finite-time stabilization. However, this technique does not apply to semi-global rapid stabilization.  Can we construct an explicit feedback law to rapidly stabilize the system in $\mathbf{H}(2\pi- \varepsilon)$ for any small $\varepsilon$?   Notice that due to Theorem \ref{THM-nounifdecay}, it is impossible to achieve uniform asymptotic
    stabilization in $\mathbf{H}(2\pi)$.
 We refer to the papers by the authors concerning the topic of (global) finite-time stabilization \cite{Coron-Xiang-2021}.
\item In Theorem \ref{lem:excconhm} and Theorem \ref{th-small-time-global-harmonic}, thanks to the geometric feature of the equation, we have proved the small-time global controllability between harmonic maps within the same homotopy class for general compact Riemannian manifold targets, which is to be compared with the analogous but longstanding problem for the nonlinear heat equations. Recently, Hartmann--Kellay--Tucsnak developed a novel method to characterize the reachable set of the heat equation \cite{MR4153111}. Is it possible to combine their ideas to  discover  the reachable set of this geometric equation and even to find a new pheonomenon? 
\item What is the situation for  more complicated geometric models? For example, what if the sphere target is replaced by a general compact Riemannian manifold? What if the space $\mathbb{T}^1$ is further replaced by a general domain $\Omega$ or even a general compact Riemannian manifold? Can we expect similar controllability and stabilization results?  Recall that in \cite{Liu-2018} the local controllability of the flow from $\Omega$ to $\Sph^2$ has been considered. Here we mainly refer to global control problems.

\end{itemize}

\section{Preliminary part}\label{sec:prep}
In this preliminary section, we provide some  definitions, well-posedness results, and basic properties of the controlled harmonic map heat flow that is related to this paper, parts of which are inspired by the authors' previous works on the controlled wave maps equations \cite{Coron-Krieger-Xiang-1, Krieger-Xiang-2022}.

\subsection{Some basic properties of the  controlled harmonic map heat flow}

\subsubsection{ Rotation invariance}
The controlled harmonic map heat flow is  invariant under  the action of the orthogonal group. To be more precise,  for any given solution, $(u, f)$, of the controlled harmonic map heat flow form $\mathbb{T}^1$ to $\Sph^k$ (with $k\geq 1$):
\begin{equation} \notag
\begin{cases}
\partial_t u- \Delta u= |\partial_x u|^2 u+ \mathbf{1}_{\omega} f^{u^{\perp}},  \\
u(0, x)= u_0(x),
\end{cases}
\end{equation}
and for any matrix $A$ belongs to $O(k+1)$, the pair $(\bar u, \bar f):= (A u, Af)$ is also  a solution of
\begin{equation} \notag
\begin{cases}
\partial_t \bar u- \Delta \bar u= |\partial_x \bar u|^2 \bar u+  \mathbf{1}_{\omega} \bar f^{\bar u^{\perp}},  \\
\bar u(0, x)= A u_0(x).
\end{cases}
\end{equation}
This is a consequence of straightforward calculation.
In particular,  every harmonic map from $\mathbb{T}^1$ to $\mathbb{S}^k$ can be characterized as follows: for any harmonic map $u: \mathbb{T}^1\rightarrow \mathbb{S}^k\subset \mathbb{R}^{k+1}$, there exist an integer $N$ and a  matrix $A\in O(k+1)$ such that
\begin{equation*}
   E(u)= 2\pi N^2 \textrm{ and }   A u(x)= \varphi (Nx) \; \forall x\in \mathbb{T}^1,
\end{equation*}
 where $\varphi$ is the simplest non-trivial harmonic map:
\begin{equation}
\label{def-varphi}
    \varphi(x):= (\cos{x}, \sin{x}, 0,\ldots , 0)^T \; \forall x\in \mathbb{T}^1.
\end{equation}

\subsubsection{Well-posedness and continuous dependence results}
Now we present some related well-posedness results. This covers  the inhomogeneous heat equation and the controlled harmonic map heat flow equations, while the closed-loop systems are introduced later on in
Section \ref{subsec:stabi:pre}.

\begin{lem}[The inhomogeneous heat equation]\label{lem:freeheat}
Let $  T >0$.  For any initial state  $v_0: \mathbb{T}^1\rightarrow \mathbb{R}^{k+1}$ in $H^1(\mathbb{T}^1)$ and for any force term $f: [0, T]\times \mathbb{T}^1\rightarrow \mathbb{R}^{k+1}$ in $L^2(0, T; L^2(\mathbb{T}^1))$, the inhomogeneous heat  equation
\begin{equation}\label{eq:inhomoheat}
  \partial_t u- \Delta u= f \; \textrm{ with } u(0, \cdot)= v_0(\cdot),
\end{equation}
admits a unique solution in $C^0([0,T];H^1(\mathbb{T}^1))$. This solution satisfies
\begin{align}
\label{evol-energy}
\frac{1}{2} \frac{d}{dt}\|u\|_{L^2}^2+\int_{\mathbb{T}^1}|u_x|^2dx&=\int_{\mathbb{T}^1} \langle f, u \rangle dx
\text{ in the sense of distributions on $(0,T)$},
\\
\|u(t)\|_{L^2(\mathbb{T}^1)}&\leq  \|v_0\|_{L^2(\mathbb{T}^1)}+ \|f\|_{L^1(0, t;L^2(\mathbb{T}^1))} \;  \forall t\in [0, T], \\
\|u_{x}\|^2_{L^2(0, t; L^2(\mathbb{T}^1))}&\leq  \|v_{0}\|^2_{L^2(\mathbb{T}^1)}+ 2\|f\|^2_{L^1(0, t;L^2(\mathbb{T}^1))} \;  \forall t\in [0, T], \\
\|u_x(t)\|^2_{L^2(\mathbb{T}^1)}+ \|u_{xx}\|^2_{L^2(0, t; L^2(\mathbb{T}^1))}&\leq  \|v_{0x}\|^2_{L^2(\mathbb{T}^1)}+ \|f\|^2_{L^2(0, t;L^2(\mathbb{T}^1))} \;  \forall t\in [0, T].
\end{align}
\end{lem}
The existence and uniqueness of a solution to \eqref{eq:inhomoheat} are classical; see for example the book by Brezis \cite[Theorem 10.11]{Brezis-book}. Concerning the other statements,  it suffices to observe that
\begin{gather*}
 \frac{1}{2} \frac{d}{dt}\|u\|_{L^2}^2= \int_{\mathbb{T}^1} \langle u_{t}, u \rangle dx= \int_{\mathbb{T}^1} \langle u_{xx}+ f, u \rangle dx =-\|u_{x}\|_{L^2}^2+ \int_{\mathbb{T}^1}\langle f, u \rangle dx, \\
    \frac{1}{2} \frac{d}{dt}\|u_x\|_{L^2}^2= \int_{\mathbb{T}^1} \langle u_{tx}, u_x \rangle dx= -\int_{\mathbb{T}^1} \langle u_{t}, u_{xx} \rangle dx\leq -\frac{1}{2} \|u_{xx}\|_{L^2}^2+ \frac{1}{2} \|f\|_{L^2}^2.
\end{gather*}

Concerning the controlled harmonic map heat flow one has  similar results:
\begin{lem}[The controlled harmonic map heat flow]\label{lem:forcehmhf}
There exists an  explicit constant $C>0$ such that for any  $T >0$, for any initial state  $v_0: \mathbb{T}^1\rightarrow \mathbb{S}^{k}\subset \mathbb{R}^{k+1}$ in $H^1(\mathbb{T}^1)$, and for any force term $f: [0, T]\times \mathbb{T}^1\rightarrow \mathbb{R}^{k+1}$ in $L^2(0, T; L^2(\mathbb{T}^1))$, the controlled harmonic maps  heat  equation
\eqref{eq:controlhmhf}
admits a unique solution in $C^0([0,T];H^1(\mathbb{T}^1;\Sph^k))$, which satisfies
\begin{gather*}
\|u_{x}(t)\|^2_{L^2(\mathbb{T}^1)}+ \|u_t\|_{L^2(0, t; L^2(\mathbb{T}^1))}^2\leq  \|v_{0x}\|^2_{L^2(\mathbb{T}^1)}+ \|f\|^2_{L^2(0, t;L^2(\mathbb{T}^1))} \;  \forall t\in [0, T], \\
 \|u_{xx}\|^2_{L^2(0, T; L^2(\mathbb{T}^1))}\leq  \|v_{0x}\|^2_{L^2(\mathbb{T}^1)}+ C \left(T\|v_{0x}\|^6_{L^2(\mathbb{T}^1)}+ \|f\|^2_{L^2(0, T;L^2(\mathbb{T}^1))}+ T\|f\|^6_{L^2(0, T;L^2(\mathbb{T}^1))}\right).
\end{gather*}
\end{lem}
\begin{proof}[Proof of Lemma \ref{lem:forcehmhf}]
We first concentrate on the \textit{a priori} estimates.    Thanks to direct integration by parts and the geometric condition, there is
\begin{align*}
     \frac{1}{2} \frac{d}{dt}\|u_x\|_{L^2}^2= \int_{\mathbb{T}^1} \langle u_{tx}, u_x \rangle dx= - \int_{\mathbb{T}^1} \langle u_t, u_t- \mathbf{1}_{\omega} f^{u^{\perp}} \rangle dx\leq - \frac{1}{2}\|u_t\|_{L^2}^2+ \frac{1}{2} \|f\|_{L^2}^2
\end{align*}
which implies
\begin{equation*}
   \|u_x\|_{C([0, T]; L^2)}^2\leq \|v_{0x}\|_{L^2}^2+ \|f\|_{L^2(0, T; L^2)}^2.
\end{equation*}
We also have
\begin{align*}
     \frac{1}{2} \frac{d}{dt}\|u_x\|_{L^2}^2&=  - \int_{\mathbb{T}^1} \langle u_{xx}+ |u_x|^2 u+ \mathbf{1}_{\omega} f^{u^{\perp}}, u_{xx} \rangle dx\\
    &\leq - \|u_{xx}\|_{L^2}^2+ \|f\|_{L^2} \|u_{xx}\|_{L^2}+ \int_{\mathbb{T}^1} \langle |u_x|^2 u, -u_t+ |u_x|^2 u+ \mathbf{1}_{\omega} f^{u^{\perp}} \rangle dx\\
    &\leq - \frac{3}{4}\|u_{xx}\|_{L^2}^2+ 2\|f\|^2_{L^2}+  2\|u_{x}\|_{L^4}^4 \\
    &\leq  -\frac{1}{2}\|u_{xx}\|_{L^2}^2+ 2\|f\|^2_{L^2}+ C \|u_x\|_{L^2}^6,
\end{align*}
where we have used the inequality
\begin{equation*}
    \|u_x\|_{L^4}^4\lesssim
\|u_x\|_{L^2}^2 \|u_x\|_{L^{\infty}}^2  \lesssim \|u_x\|_{L^2}^3 \|u_{xx}\|_{L^{2}}.
\end{equation*}
Therefore,
\begin{align*}
    \|u_{xx}\|_{L^2(0, T; L^2)}^2&\leq \|v_{0x}\|_{L^2}^2+ 4 \|f\|_{L^2(0, T; L^2(\mathbb{T}^1))}^2+ C \int_0^T \|u_x(s)\|_{L^2}^6 ds \\
    &\leq \|v_{0x}\|_{L^2}^2+ 4 \|f\|_{L^2(0, T; L^2(\mathbb{T}^1))}^2+ C T \left(\|v_{0x}\|_{L^2}^6+ \|f\|_{L^2(0, T; L^2)}^6 \right).
\end{align*}
This finishes the proof of desired estimates.

Let us now deal with the proof of the well-posedness result.
Since $f^{u^{\perp}}$ can be written as  $f- \langle f, u \rangle u$, the  geometric equation \eqref{eq:controlhmhf}  is equivalent to the following nonlinear heat equation:
  \begin{equation}
    \partial_t u- \Delta u= |\partial_x u|^2 u+ \mathbf{1}_{\omega} f-  \langle \mathbf{1}_{\omega} f, u \rangle u. \notag
\end{equation}
By regarding $\mathbf{1}_{\omega} f$ as some source term $g$, the preceding equation becomes a standard  nonlinear (subcritical) heat equation. This observation, to be combined with the \textit{a priori} estimates given in Lemma \ref{lem:freeheat} and the standard Banach fixed-point argument, leads to the existence and uniqueness of the solution to the system.  Indeed, it suffices to consider  the  map
\begin{align*}
\Gamma: \mathcal{X}_T&\rightarrow \mathcal{X}_T\\
v&\mapsto u
\end{align*}
where the Banach space $\mathcal{X}_T$ is defined in the following in \eqref{def:mcXT}, and the state $u= \Gamma v$ is the unique solution of
\begin{equation*}
\begin{cases}
\partial_t u- \Delta u= |\partial_x v|^2 v+ \mathbf{1}_{\omega} f-  \langle \mathbf{1}_{\omega} f, v \rangle v, \\
 u(0, \cdot)= u_0(\cdot).
\end{cases}
\end{equation*}
This map admits a unique  fixed-point in $\mathcal{X}_T$. Thus we finish the proof of the well-posedness result.
\end{proof}

To state the continuous dependence result we work with the following standard  space $\mathcal{X}_T$ in space-time:
\begin{gather}\label{def:mcXT}
\mathcal{X}_T:= \{f: f\in  C([0, T]; H^1(\mathbb{T}^1)), f_{xx}\in   L^2(0, T; L^2(\mathbb{T}^1))\}, \\
    \|f\|_{\mathcal{X}_T}:= \|f\|_{C([0, T]; H^1(\mathbb{T}^1))}+ \|f_{xx}\|_{L^2(0, T; L^2(\mathbb{T}^1))}\; \forall f\in \mathcal{X}_T.
\end{gather}
\begin{lem}[Continuous dependence of the controlled harmonic map heat flow]\label{lem:cont-dep}
    Let $T_0>0$, let $M>0$. There exists some effectively computable constant $C_d$ depending on the values of $(T_0, M)$ such that for any initial states $v_1(x), v_2(x)\in H^1(\mathbb{T}^1; \Sph^k)$ and any force terms $f_1, f_2\in L^{\infty}(0, T_0; L^2(\mathbb{T}^1))$ satisfying
    \begin{equation*}
        \|v_i\|_{H^1}+ \|f_i\|_{L^{\infty}(0, T_0; L^2(\mathbb{T}^1))}\leq M\; \forall i\in \{1, 2\},
    \end{equation*}
    the unique solutions of the controlled harmonic map heat flow $\forall i\in \{1, 2\}$,
\begin{equation} \notag
\begin{cases}
\partial_t u_i- \Delta u_i= |\partial_x u_{i}|^2 u_i+ \mathbf{1}_{\omega} f_i^{u_i^{\perp}},  \\
u_i(0, x)= v_i(x),
\end{cases}
\end{equation}
    satisfy
    \begin{equation*}
     \|u_1- u_2\|_{\mathcal{X}_{T_0}}\leq C_d \left(\|v_1- v_2\|_{H^1(\mathbb{T}^1)}+ \|f_1- f_2\|_{L^2(0, T_0; L^2(\mathbb{T}^1))} \right).
    \end{equation*}
\end{lem}
\begin{proof}[Proof of Lemma \ref{lem:cont-dep}]
    The proof is a direct combination of the preceding well-posedness result (Lemma \ref{lem:forcehmhf}) and a bootstrap argument as we are now going to detail. For some $T\in (0, T_0]$ to be fixed later on, we consider the equations on time interval $(0, T)$. Compare the equations on $u_1$ and $u_2$, we know that  $w(t, x):= u_1(t, x)- u_2(t, x)$ satisfies
    \begin{align*}
     \partial_t w- \Delta w&= |u_{1x}|^2 w+ \langle w_x, u_{1x}+ u_{2x}\rangle u_2\\
     &\;\;\;\;\; + \mathbf{1}_{\omega} (f_1- f_2)+ \mathbf{1}_{\omega} \langle  f_1, u_1\rangle w+ \mathbf{1}_{\omega}\langle  f_1, w\rangle u_2+ \mathbf{1}_{\omega} \langle f_1- f_2, u_2\rangle u_2, \\
     &=: L(w)+  \mathbf{1}_{\omega}(f_1- f_2)+ \mathbf{1}_{\omega} \langle (f_1- f_2), u_2\rangle u_2.
    \end{align*}
Thanks to Lemma \ref{lem:forcehmhf}, we know that
\begin{equation*}
    \|u_1\|_{\mathcal{X}_{T_0}}+   \|u_2\|_{\mathcal{X}_{T_0}}\lesssim 1,
    \end{equation*}
which further implies  that for $i= 1, 2$,
\begin{equation*}
    \|u_{ix}\|^2_{L^2(0, T; L^{\infty})}\lesssim \int_0^T \|u_{ix}(t)\|_{L^2} \|u_{ixx}(t)\|_{L^2} dt\lesssim T^{\frac{1}{2}}.
\end{equation*}
By plugging these estimates into the preceding equation on $w$, according to Lemma    \ref{lem:freeheat}, we obtain
\begin{align*}
\|w\|_{\mathcal{X}_{T}}&\lesssim \|w(0)\|_{H^1}+ \|f_1- f_2\|_{L^2(0, T; L^2)}+ \| L(w)\|_{L^2(0, T; L^2)}.
\end{align*}

Successively, we have
\begin{align*}
    \| |u_{1x}|^2 w\|_{L^2(0, T; L^2)} &\lesssim \|u_{1x}\|_{L^{\infty}(0, T; L^2)} \|u_{1x}\|_{L^{2}(0, T; L^{\infty})} \|w\|_{L^{\infty}(0, T; L^{\infty})}\lesssim T^{\frac{1}{4}} \|w\|_{\mathcal{X}_{T}}, \\
  \| \langle w_x, u_{1x}\rangle u_2\|_{L^2(0, T; L^2)} &\lesssim \|w_{x}\|_{L^{\infty}(0, T; L^2)} \|u_{1x}\|_{L^{2}(0, T; L^{\infty})} \|u_2\|_{L^{\infty}(0, T; L^{\infty})}\lesssim T^{\frac{1}{4}} \|w\|_{\mathcal{X}_{T}}, \\
   \| \langle w_x, u_{2x}\rangle u_2\|_{L^2(0, T; L^2)} &\lesssim \|w_{x}\|_{L^{\infty}(0, T; L^2)} \|u_{2x}\|_{L^{2}(0, T; L^{\infty})} \|u_2\|_{L^{\infty}(0, T; L^{\infty})}\lesssim T^{\frac{1}{4}} \|w\|_{\mathcal{X}_{T}}, \\
    \| \langle f_1, u_1\rangle w\|_{L^2(0, T; L^2)} &\lesssim \|f_{1}\|_{L^{2}(0, T; L^2)} \|u_{1}\|_{L^{\infty}(0, T; L^{\infty})} \|w\|_{L^{\infty}(0, T; L^{\infty})}\lesssim T^{\frac{1}{2}} \|w\|_{\mathcal{X}_{T}}, \\
    \| \langle f_1, w\rangle u_2\|_{L^2(0, T; L^2)} &\lesssim \|f_{1}\|_{L^{2}(0, T; L^2)} \|u_{2}\|_{L^{\infty}(0, T; L^{\infty})} \|w\|_{L^{\infty}(0, T; L^{\infty})}\lesssim T^{\frac{1}{2}} \|w\|_{\mathcal{X}_{T}}.
\end{align*}

Therefore, by selecting $T$ sufficiently small, one obtains
\begin{align*}
\|w\|_{\mathcal{X}_{T}}&\lesssim \|w(0)\|_{H^1}+ \|f_1- f_2\|_{L^2(0, T; L^2)}.
\end{align*}
Then, we iterate this procedure to achieve the required estimates.

\end{proof}

\subsection{Controllability within each homotopy class}

Concerning controllability, due to a simple topological argument, one shall consider the deformation of the state within its homotopy class. Indeed, for any given initial state $u(t=0, \cdot)\in H^1(\mathbb{T}^1; \Sph^k)$ and any force $f \in L^2(0, T; L^2(\mathbb{T}^1))$, the unique solution of \eqref{eq:controlhmhf} satisfies $u\in C([0, T]; H^1(\mathbb{T}^1))$
and therefore is in  $C([0, T]; C_x(\mathbb{T}^1))$. Hence $u(t, \cdot)= u(t)(\cdot)$ can be regarded as a continuous deformation of a closed curve on $\Sph^k$, which forces $u(0)(\cdot)$ and $u(T)(\cdot)$ to be homotopic. This leads to the following definition.  See Figure \ref{pictorus} to illustrate this natural definition.

\begin{defn}\label{def-con-homo}
    Let $\mathcal{N}$ be a compact Riemannian submanifold of $\R^m$. The controlled harmonic map heat flow equation $u: \mathbb{R}^+\times \mathbb{T}^1\rightarrow \mathcal{N}$ is said to be  ``globally  null controllable in each homotopy class" if  for any pair of initial and final  states $(v_0, v_1)$ in  $H^1(\mathbb{T}^1;\mathcal{N})$  that  are homotopic, where $v_1$ is a harmonic map from $\mathbb{T}^1$ to $\mathcal{N}$, there exists a control $f_0\in L^2_{t, x}([0, T]\times \mathbb{T}^1)$ for some $T>0$ depending on the given pair of states such that, the unique solution of the controlled harmonic map heat flow equation with  initial state $u(0, \cdot)= v_0(\cdot)$ and  control $f_0$ satisfies $u(T, \cdot)= v_1(\cdot)$.
\end{defn}

\subsection{Stabilization and closed-loop systems}\label{subsec:stabi:pre}
A map $F$ as defined in  \eqref{feedbacklaw-def} is called  a Carath\'{e}odory map if the following three properties are satisfied:
\begin{itemize}
\item[\label{P1}($\mathcal{P}_1$)] There exists a nondecreasing function $R\in(0,+\infty)\mapsto C_B(R)\in (0,+\infty)$ such that, for every $R\in(0,+\infty)$ and for every $t\in\R$,   $ \|u_0\|_{\dot H^1}\leq R \Rightarrow \|F(t; u_0)\|_{L^2}\leq C_B(R)$;
   \item[\label{P2}($\mathcal{P}_2$)] For all $u_0 \in H^1(\mathbb{T}^1; \mathbb{S}^k)$, the function $t\in \mathbb{R}\mapsto F(t, u_0)\in L^2(\mathbb{T}^1)$ is measurable;
   \item[\label{P3}($\mathcal{P}_3$)] for almost every $t\in \mathbb{R}$, the function $ u_0 \in H^1_x (\mathbb{T}^1; \mathbb{S}^k)\mapsto F(t, u_0)\in L^2(\mathbb{T}^1; \mathbb{R}^{k+1})$ is continuous.
\end{itemize}
This map $F$ is further called a Lipschitz map if
\begin{itemize}
    \item[\label{P4}($\mathcal{P}_4$)]  For every $R>0$, there exists $K(R)>0$ such that
\begin{gather*}
    \left(\|u_1\|_{\dot H^1}\leq R, \|u_2\|_{\dot H^1}\leq R \right) \Rightarrow
    \left( \|F(t, u_1)- F(t, u_2)\|_{L^2}\leq K(R) \|u_1- u_2\|_{H^1}\right).
\end{gather*}
\end{itemize}

This motivates the definition of solutions to the closed-loop system \eqref{closed-loop-timevarying}:
\begin{defn}\label{def:sol:clos}
      Let $F$ be a map in the form of \eqref{feedbacklaw-def} that satisfies conditions \hyperref[P1]{($\mathcal{P}_1$)}, \hyperref[P2]{($\mathcal{P}_2$)}, and \hyperref[P3]{($\mathcal{P}_3$)}.  Let $T_1\in \mathbb{R}$.  Let $u_0\in L^2(\mathbb{T}^1; \mathbb{S}^k)$. A function $u$  is a solution to  the Cauchy problem
\begin{equation}\label{eq:cauchyclosedloop}
\begin{cases}
\partial_t u(t,x)- \Delta u(t,x)=  |u_{ x}|^2(t,x) u(t,x)+ \mathbf{1}_{\omega}  (F(t; u(t)))^{u(t,x)^{\perp}} \\
 u(T_1, x)= u_0(x),
     \end{cases}
\end{equation}
if there exists an interval $I$ with a non-empty interior satisfying $I \cap (-\infty, T_1]= \{T_1\}$ such that
 $u$ is defined on $I\times \mathbb{T}^1$ and is such that, for every $T_2>T_1 $ such that $[T_1,T_2]\subset I$, the restriction of $u$ to $[T_1,T_2]\times \mathbb{T}^1$ is in $C([T_1, T_2]; H^1 (\mathbb{T}^1; \mathbb{S}^k))$ and is a solution to the Cauchy problem \eqref{eq:controlhmhf} with $f(t, x):= F(t, u(t, \cdot))(x)$. The interval $I$ is denoted by $D(u)$.
\\
We say that a solution $u$ to the Cauchy problem \eqref{eq:cauchyclosedloop} is maximal if, for every solution $\tilde u$ to the Cauchy problem \eqref{eq:cauchyclosedloop} such that
\begin{gather*}
    D(u)\subset D(\tilde u), \\
   u(t, \cdot)=  \tilde u(t, \cdot) \textrm{ for every } t\in D(u),
\end{gather*}
one has $ D(u)= D(\tilde u)$.
\end{defn}

\begin{defn}
   Let  $F$ be a map in the form of \eqref{feedbacklaw-def} that satisfies conditions \hyperref[P1]{($\mathcal{P}_1$)}, \hyperref[P2]{($\mathcal{P}_2$)}, and \hyperref[P3]{($\mathcal{P}_3$)}. Let $I$ be a nonempty interval of $\mathbb{R}$. A function $\phi$ is a solution of the closed-loop system \eqref{closed-loop-timevarying} on $I$ if, for every $[T_1, T_2]\subset I$, the restriction of $\phi$ to $[T_1, T_2]\times \mathbb{T}^1$ is a solution of the Cauchy problem \eqref{eq:cauchyclosedloop} with initial state $\phi(T_1, \cdot)$.
\end{defn}

Similar to the open-loop system, one obtains the well-posedness of this closed-loop system.
\begin{prop}\label{lem:wellclosegene}
Let $F$ be a map in the form of \eqref{feedbacklaw-def} that satisfies conditions \hyperref[P1]{($\mathcal{P}_1$)}, \hyperref[P2]{($\mathcal{P}_2$)}, \hyperref[P3]{($\mathcal{P}_3$)}, and \hyperref[P4]{($\mathcal{P}_4$)}. Then, \\
1) For every $R\in (0, +\infty)$, there exist a time $T(R)>0$ and a constant $L(R)>0$ such that
\begin{itemize}
    \item For every $T_1\in \mathbb{R}$  and for every initial state $u_0$ satisfying $\|u_0\|_{\dot H^1}\leq R$, the Cauchy problem \eqref{eq:cauchyclosedloop}  with initial state $u_0$ at time $T_1$  has a unique solution on $[T_1, T_1+ T(R)]$;
    \item  For every $T_1\in \mathbb{R}$, for every $u_0$ (\textit{resp}. $\tilde u_0$) satisfying $\|u_0\|_{\dot H^1}\leq R$ (\textit{resp}. $\|\tilde u_0\|_{\dot H^1}\leq R$), the unique solution $u$ (\textit{resp}. $\tilde u $) of the Cauchy problem \eqref{eq:cauchyclosedloop} with initial state $u_0$ (\textit{resp}. $\tilde u_0$) at time $T_1$  satisfies
    \begin{equation*}
        \|u(t)- \tilde u(t)\|_{H^1}\leq L(R) \|u_0- \tilde u_0\|_{H^1}\; \forall t\in [T_1, T_1+ T(R)].
    \end{equation*}
\end{itemize}
2) For every $T_1\in \mathbb{R}$, for every initial state $ u_0\in H^1(\mathbb{T}^1; \mathbb{S}^k)$, the Cauchy problem   has a unique maximal solution $u$. If $D(u)$ is not equal to $[T_1, +\infty)$, then there exists some $\tau\in (T_1,+\infty)$ such that $D(u)= [T_1, \tau)$ and one has
\begin{equation*}
    \lim_{t\rightarrow \tau^-}\|u(t)\|_{H^1}= +\infty.
\end{equation*}
\end{prop}
Since the proof of this proposition for closed-loop systems is closely related to the ones of Lemma \ref{lem:forcehmhf}--\ref{lem:cont-dep} for open-loop systems (i.e., it is essentially based on bootstrap arguments and direct energy estimates), we omit it. Note that the transition of well-posedness and continuous dependence theory from open-loop systems to closed-loop systems is a standard problem in the study of stabilization problems since, as typically treated in \cite[Lemma 3 for open-loop systems and Theorems Theorem 7-8 for closed-loop systems]{2017-Coron-Rivas-Xiang-APDE} for KdV equations.

\section{Global approximate controllability}\label{sec:GAC}

\subsection{Part 1: global stability of the harmonic map heat flow}
\label{sec:globalconvergence}

This section is devoted to the proof of Proposition \ref{thm:homopotyconver} concerning the uniform convergence of the flow to harmonic maps. This proposition is also related to \hyperref[Step1]{Step 1} of Section \ref{sec:strategy}. First we present the following auxiliary lemma.
\begin{lem}\label{lem:flux}
    Let $T>0$.  Let $N\geq 0$. There exists a non-decreasing function $b: (0, 1)\rightarrow (0, 1)$ satisfying $\lim_{\delta\rightarrow 0^+} b(\delta)= 0$,    such that for any initial state $u_0$ satisfying, for some $\delta\in(0,1)$,
    \begin{equation}\label{eq:cond:1}
        \int_{\mathbb{T}^1} |u_{0x}(x)|^2 dx \in (2\pi N^2 +\delta, 2\pi (N+1)^2- \delta),
    \end{equation}
    the unique solution of the harmonic map heat flow \eqref{eq:freehmhfs} satisfies
    \begin{equation*}
        \int_0^T \int_{\mathbb{T}^1} |u_t(t, x)|^2 dx dt\geq  b(\delta).
    \end{equation*}
\end{lem}
\begin{proof}[Proof of Lemma \ref{lem:flux}]
    The proof is based on a standard compactness argument. Indeed, suppose that this property is not true, then we can find a sequence of initial states  $\{v_n(x)\}_{n\in \mathbb{N}}$ satisfying the condition \eqref{eq:cond:1} such that the corresponding solutions $\{u_n(t, x)\}_{n\in \mathbb{N}}$ of \eqref{eq:freehmhfs} for the initial condition $u_n(0,\cdot)= v_n$ satisfy
      \begin{equation}\label{assum-unt}
        \int_0^T \int_{\mathbb{T}^1} |u_{nt}|^2 dx dt\leq \frac{1}{n}.
    \end{equation}
    Thanks to Lemma \ref{lem:forcehmhf}, we know that $\{u_n\}_{n\in \mathbb{N}}$ are uniformly bounded in $C([0, T]; H^1(\mathbb{T}^1))\cap L^2(0, T; H^2(\mathbb{T}^1))$, which immediately implies that
$\{\partial_t u_n\}_{n\in \mathbb{N}}$ are uniformly bounded in $L^2(0, T; L^2(\mathbb{T}^1))$.
    Hence, there exists some $u\in C([0, T]; H^1(\mathbb{T}^1))\cap L^2(0, T; H^2(\mathbb{T}^1)) $ such that
\begin{align}
\label{w-un02}
u_n&\rightharpoonup u \textrm{ weakly in } L^2(0,T; H^2(\mathbb{T}^1)),
\\
\label{w-dtun00}
 \partial_t u_n&\rightharpoonup \partial_t u \textrm{ weakly in } L^2(0,T; L^2(\mathbb{T}^1)).
\end{align}
The preceding convergence results, combined with the Aubin-Lions lemma and a standard diagonal argument, imply that up to a selection of a subsequence, still denoted by $\{u_n\}_{n\in \mathbb{N}}$,
\begin{equation}
\label{un02-r}
 u_n\rightarrow u \textrm{ strongly in } L^2(0,T; H^{2-r}(\mathbb{T}^1))\; \forall r\in (0,2].
\end{equation}
Since the sequence of functions $\{u_n\}_{n\in \mathbb{N}}$ is uniformly bounded in  $C([0, T]; H^1(\mathbb{T}^1))\cap L^2(0, T; H^2(\mathbb{T}^1))$, we know from direct interpolation that it is also uniformly bounded in
\begin{equation}\label{unLp2-r-new}
    L^p(0, T; H^{2-r}(\mathbb{T}^1))\; \forall r\in (0, 1), \; \forall p\in [2, \frac{2}{1-r}).
\end{equation}
From \eqref{w-dtun00}, \eqref{unLp2-r-new}, and the Aubin-Lions lemma one gets that, up to a selection of a subsequence,
\begin{equation}
\label{un1/4-r}
 u_n\rightarrow u \textrm{ strongly in }  L^p(0, T; H^{2-r}(\mathbb{T}^1))\; \forall r\in (0, 1), \; \forall p\in [2, \frac{2}{1-r}).
\end{equation}
In particular, by selecting $p= 4$ we know that
\begin{equation}
\label{un1/4-r-new}
 u_{nx}\rightarrow u_x \textrm{ strongly in }  L^4(0, T; H^{\frac{1}{2}- \varepsilon}(\mathbb{T}^1))\; \forall \varepsilon\in (0, \frac{1}{2}).
\end{equation}
Thanks to the Sobolev embedding theorem
$H^{1/4}(\mathbb{T}^1)\hookrightarrow L^4(\mathbb{T}^1)$, 
one gets
\begin{equation}
\label{unx44}
 u_{nx}\rightarrow u_x \textrm{ strongly in } L^4(0,T; L^4(\mathbb{T}^1)).
\end{equation}
Due to assumption \eqref{assum-unt} and \eqref{w-dtun00}, one knows that $\partial_t u= 0$, thus $u(t, \cdot)= u(\cdot)$ is time independent. Then,  from \eqref{w-un02}, \eqref{un02-r} with $r=1$, the Sobolev embedding $H^{1}(\mathbb{T}^1)\hookrightarrow L^\infty(\mathbb{T}^1)$, and \eqref{unx44},
\begin{equation*}
    0= -\partial_t u_{n}+ \partial_{xx} u_n+ |u_{nx}|^2 u_n\rightharpoonup  \partial_{xx} u+ |u_{x}|^2 u\textrm{  weakly in } L^1(0,T; L^{2}(\mathbb{T}^1)).
\end{equation*}
Therefore,  $u$ is a harmonic map, which implies that, for some $\tilde N\in \mathbb{N}$,
\begin{equation}
\label{energyu}
\int_{\mathbb{T}^1} |u_x|^2(t,x) dx =2\pi \tilde N^2 \; \forall t \in[0,T].
\end{equation}
Recall that, for every $n$, one has by \eqref{evol-energy}
\begin{equation}
\label{evolution-energy}
\int_{\mathbb{T}^1} |u_{nx}|^2(t, x) dx- \int_{\mathbb{T}^1} |v_{nx}|^2(x) dx= - 2\int_0^t \int_{\mathbb{T}^1} |u_{nt}|^2(t, x)dx dt,
\end{equation}
 which, together with \eqref{eq:cond:1} for $v_n$, \eqref{assum-unt}, and  \eqref{un02-r} with $r=1$, implies that
\begin{equation*}
    \int_0^T \int_{\mathbb{T}^1} |u_x|^2 dx dt= \lim_{n}  \int_0^T \int_{\mathbb{T}^1} |u_{nx}|^2 dx dt\in  T[2\pi N^2 +\delta, 2\pi (N+1)^2- \delta],
\end{equation*}
which leads to a contradiction with \eqref{energyu}. This concludes the proof of Lemma \ref{lem:flux}.
\end{proof}

Armed with this auxiliary lemma we come back to the proof of Proposition \ref{thm:homopotyconver}.
\begin{proof}[Proof of Proposition \ref{thm:homopotyconver}]
The preceding lemma indicates that for any initial state $v(x)\in H^1 (\mathbb{T}^1)$ there is $N\in \mathbb{N}$ such that
\begin{equation}\label{ineq:cnvE}
    E(t)\rightarrow 2\pi N^2.
\end{equation}
Therefore, for any given $\varepsilon>0$,  there exists some $T_0>0$ such that
\begin{gather*}
    E(t)\in [2\pi N^2, 2\pi N^2+ \varepsilon)\; \forall t> T_0.
\end{gather*}
Notice that on $t\in [T_0, T_0+1]$ one has
\begin{equation*}
     \partial_t u- \Delta u- |\partial_x u|^2 u=0,
\end{equation*}
thus
\begin{equation*}
    \frac{d}{dt} E(t)= -2 \langle u_t, u_{xx} \rangle= -2 \langle u_t, u_t \rangle.
\end{equation*}
Therefore,
  \begin{equation*}
        \int_{T_0}^{T_0+1} \int_{\mathbb{T}^1} |u_t|^2 dx dt= \frac{1}{2} \left( E(T_0)- E(T_0+1)\right)\leq \frac{\varepsilon}{2}.
    \end{equation*}
Thus there exists some $t_0\in [T_0, T_0+1]$ such that
\begin{equation*}
    \int_{\mathbb{T}^1} |u_t|^2(t_0, x) dx\leq \varepsilon/2,
\end{equation*}
which, to be combined with the harmonic map heat equation satisfied by $u$, further implies that
\begin{equation}
\label{uxx+ux}
    \int_{\mathbb{T}^1} |u_{xx}+ |u_x|^2 u|^2(t_0, x) dx\leq \varepsilon/2.
\end{equation}

We also know that for any $x_0, x_1\in \mathbb{T}^1$,
\begin{equation*}
    |u_x|^2(t_0, x_0)- |u_x|^2(t_0, x_1)= 2\int_{x_1}^{x_0} u_x \cdot u_{xx} (t_0, x) dx= 2\int_{x_1}^{x_0} u_x \cdot u_{t} (t_0, x) dx= O(\varepsilon^{\frac{1}{2}}).
\end{equation*}
Since
\begin{equation*}
    \int_{\mathbb{T}^1} |u_x|^2(t_0, x) dx\in [2\pi N^2, 2\pi N^2+ \varepsilon),
\end{equation*}
one concludes that
\begin{equation}
    |u_x|^2(t_0, x)= N^2+ O(\varepsilon^{\frac{1}{2}})\; \forall x\in \mathbb{T}^1.
\end{equation}
Therefore, using also \eqref{uxx+ux},
\begin{equation*}
    \int_{\mathbb{T}^1} |u_{xx}+ N^2 u|^2(t_0, x) dx\lesssim \varepsilon.
\end{equation*}
Since $u(t_0, x)$ belongs to $H^2(\mathbb{T}^1)$, we can express it via Fourier series:
\begin{equation*}
    u(t_0, x)= \sum_{n\in \mathbb{Z}} a_n e^{i nx} \; \textrm{ where } a_n= \overline{a_{-n}}\in \mathbb{C}^{k+1}
\end{equation*}
with
\begin{equation*}
    \sum_{n\in \mathbb{Z}} |a_n|^2 n^4< +\infty.
\end{equation*}
One mimics the framework in \cite[Proposition 4.4]{Coron-Krieger-Xiang-1}, and obtains
\begin{equation*}
    \sum_n (N^2- n^2)^2 a_n^2\lesssim \varepsilon.
\end{equation*}
This yields
\begin{equation*}
    |a_n|\lesssim \frac{\varepsilon^{\frac{1}{2}}}{n^2} \text{ uniformly for $n\in \mathbb{Z}\setminus \{-N,N\}$.}
\end{equation*}
Combining these inequalities with the fact that $|u(t_0, x)|= 1$, one has
\begin{equation*}
    |a_N e^{iNx}+ a_{-N} e^{-iNx}|= 1+ O(\varepsilon^{\frac{1}{2}})\;  \forall x\in \mathbb{T}^1.
\end{equation*}
Hence, there exists some harmonic maps $\gamma(x)$ in the form of
\begin{gather*}
    \gamma(x)=b_N e^{iNx}+ \overline{b_{N}} e^{-iNx},
\\
b_N=\frac{1}{2}\left(\alpha_N+ i\beta _N\right)\text{ with }\alpha_N\in \R^{k+1},\, \beta_N\in \R^{k+1}, \, |\alpha_N|=|\beta_N|=1, \text{ and } \alpha_N\cdot \beta_N=0,
\end{gather*}
such that
\begin{equation*}
    |(a_N e^{iNx}+ a_{-N} e^{-iNx})- (b_N e^{iNx}+ \overline{b_N} e^{-iNx})|\lesssim \varepsilon^{\frac{1}{2}}\; \forall x\in \mathbb{T}^1.
\end{equation*}
Hence
\begin{equation*}
    \|u(t_0, x)- \gamma(x)\|_{H^1(\mathbb{T}^1)}\leq \|(a_N e^{iNx}+ a_{-N} e^{-iNx})- \gamma(x)\|_{H^1(\mathbb{T}^1)}+ \|\sum_{n\neq \pm N} a_n e^{inx}\|_{H^{1}(\mathbb{T}^1)}\lesssim \varepsilon^{\frac{1}{2}}.
\end{equation*}
This implies that $u(t_0, x)$ is a ``$O(\varepsilon^{\frac{1}{2}})$-approximate harmonic map", and ends the proof of Theorem \ref{thm:homopotyconver}.
\end{proof}

\subsection{Part 2: the power series expansion method to construct controls to cross the critical energy levels.}\label{subsec:pse}
This part is  related to \hyperref[Step2]{Step 2} of Section \ref{sec:strategy}. The proof is based on the power series expansion method of nonlinear control theory.
We adapt the strategy introduced in  \cite{Coron-Krieger-Xiang-1} concerning the wave maps equation, where we recall that the function $\varphi$ is defined in \eqref{def-varphi}.

\begin{lem}\label{lem_decrea}
Let $N\in \mathbb{N}^*$. Let $T> 0$. There exist an effectively computable constant $\varepsilon_0>0$ and an  explicit function $f_1(t, x)\in L^{\infty}(0, T; L^2(\mathbb{T}^1))$
such that, for any $\varepsilon\in (0, \varepsilon_0]$,
the unique solution of the controlled harmonic map heat flow
\begin{equation}\label{eq_full}
\begin{cases}
   \partial_t \bar u- \Delta \bar u= |\partial_x \bar u|^2 \bar u+  \mathbf{1}_{\omega}(\varepsilon f_1)^{\bar u^{\perp}}, \\
   \bar u(0, x)= \varphi(Nx),
\end{cases}
\end{equation}
satisfies
\begin{equation}  \notag
    E(T)\in (2\pi N^2- 3\pi N^2 \varepsilon^2, 2\pi N^2- \pi N^2\varepsilon^2).
\end{equation}
\end{lem}
\begin{figure}[t]
\centering
\includegraphics[width=0.7\linewidth, trim={0cm 0.0cm 0cm 0.0cm},clip]{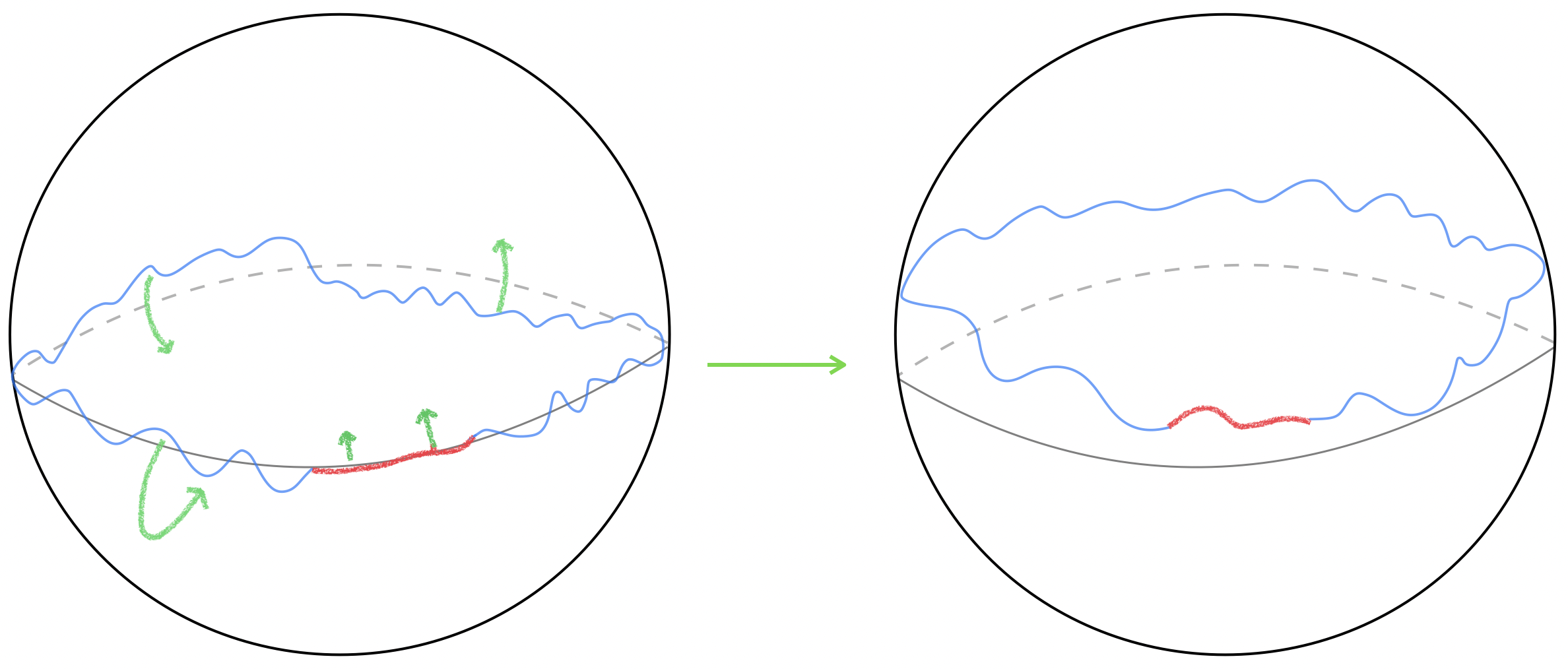}
\caption{The process of using explicit controls to cross the critical energy level set. This picture is also related to  \hyperref[Step2]{Step 2} of Section \ref{sec:strategy}.}
\label{pic3}
\end{figure}
This lemma indicates that with the help of a well-designed control one can reduce the energy of the system if it is not $0$.
This observation, together with the rotation invariance of the forced harmonic map heat equation as well as the continuous dependence property of Lemma \ref{lem:cont-dep}, immediately implies the following:
\begin{prop}\label{Prop:pse}
    Let $N\in \mathbb{N}^*$. Let $T>0$. There exist some effectively computable constant $\nu_1, \varepsilon_1>0$ and an explicit function $h\in L^{\infty}(0, T; L^2(\mathbb{T}^1; \mathbb{R}^{k+1}))$  such that, for any {\it $\nu_1$-approximate harmonic maps}, namely $v\in \mathcal{Q}_{\nu_1}$, such that $E(v)\in (2\pi N^2-1, 2\pi N^2+1)$ we can find  an explicit rotation matrix $A\in O(k+1)$ such that the solution of the controlled harmonic map heat flow equation
    \begin{equation*}
        \begin{cases}
            \partial_t u- \Delta u= |\partial_x u|^2 u+ \mathbf{1}_{\omega} (A h)^{u^{\perp}}, \\
            u(0, x)= v(x),
        \end{cases}
    \end{equation*}
    satisfies
    \begin{equation*}
        \int_{\mathbb{T}^1} |u_x(T, x)|^2 dx\in (2\pi N^2- 10 \varepsilon_1, 2\pi N^2- \varepsilon_1).
    \end{equation*}
\end{prop}

Now we turn to the proof of  Lemma \ref{lem_decrea}.
\begin{proof}[Proof of Lemma \ref{lem_decrea}]
Notice that the control $\varepsilon f_1$ is spatially supported in $\omega$, thus $\mathbf{1}_{\omega}(\varepsilon f_1)^{\bar u^{\perp}}$ equals $(\varepsilon f_1)^{\bar u^{\perp}}$ and the equation \eqref{eq_full} becomes
\begin{equation}\label{eq_full_re}
\begin{cases}
   \partial_t \bar u- \Delta \bar u= |\partial_x \bar u|^2 \bar u+  (\varepsilon f_1)^{\bar u^{\perp}}, \\
   \bar u(0, x)= \varphi(Nx).
\end{cases}
\end{equation}
This proof relies on  the  power series expansion method. We first formally illustrate the main idea, while the rigorous proof will be rendered later on. Roughly speaking, for  $ \bar u$ and $f$ we assume that
\begin{align*}
     \bar u&:= \bar u_0+ \varepsilon \bar u_1+ \varepsilon^2 \bar u_2+\ldots  \;  \textrm{ with }  \bar u_i= ( \bar u_i^1,  \bar u_i^2,  \bar u_i^3,\ldots , \bar u^{k+1}_i)^T, \\
     f&:= \varepsilon f_1
 \;\textrm{ with }  \; f_1= (0, 0, f_1^3, \ldots , f_1^{k+1})^T.
\end{align*}
By substituting these expansions into equation \eqref{eq_full_re} one deduces the equation that governs the zeroth  order term $\bar u_0$:
\begin{equation*}
\begin{cases}
 \partial_t \bar u_0- \Delta \bar u_0= |\partial_x \bar u_{0}|^2 \bar u_0, \\
 u_0(0, \cdot)= \varphi(N \cdot),
\end{cases}
\end{equation*}
indicating that $u_0(t, x)= \varphi(Nx) \; \forall x\in \mathbb{T}^1$.  One further finds that the first order term  $\bar u_1$ satisfies:
\begin{equation*}
\begin{cases}
\partial_t \bar u_1- \Delta \bar u_1
    =  2 ( \bar u_{0x} \cdot  \bar u_{1x})  \bar u_0+ |\bar u_{0x}|^2 u_1+   f_1- (f_1 \cdot  \bar u_0)  \bar u_0, \\
     \bar u_1(0)= 0.
     \end{cases}
\end{equation*}
Thus
\begin{equation*}
\partial_t \begin{pmatrix}
    \bar u_1^1\\\bar u_1^2\\\bar u_1^3\\ \ldots  \\\bar u_1^{k+1}
\end{pmatrix}
- \Delta \begin{pmatrix}
    \bar u_1^1\\\bar u_1^2\\\bar u_1^3\\ \ldots  \\\bar u_1^{k+1}
\end{pmatrix}=
N\left(-(\sin{N x}) u^1_{1x}+ (\cos{N x}) u^2_{1x}\right) \begin{pmatrix}
    \cos{Nx}\\ \sin{Nx}\\0\\ \ldots  \\0
\end{pmatrix}+ N^2\begin{pmatrix}
    \bar u_1^1\\\bar u_1^2\\\bar u_1^3\\ \ldots  \\\bar u_1^{k+1}
\end{pmatrix}+
\begin{pmatrix}
    0\\0\\f^3_1\\ \ldots  \\f^{k+1}_1
\end{pmatrix}
\end{equation*}
This implies that $\bar u^1_1= \bar u^2_1=0$ and that
\begin{equation*}
    \partial_t \bar u^i_1- \Delta \bar u^i_1- N^2 \bar u^i_1= f^i_1 \;  \forall i=3, \ldots , k+1.
\end{equation*}
Observe the following approximate controllability result on the one-dimensional heat equation. 
 \begin{lem}
 Let $T>0$. There exists an explicit control  $g\in L^{\infty}(0, T; L^2(\mathbb{T}^1; \mathbb{R}))$  such that  the unique solution $\bar w$ of
 \begin{equation}\label{eq_v0}
 \begin{cases}
     \partial_t \bar w- \Delta \bar w- N^2 \bar w= \mathbf{1}_{\omega} g, \\
     \bar w(0, x)= 0,
\end{cases}
 \end{equation}
 satisfies $\bar w(T, x)= 1$. Moreover,
 \begin{equation*}
     \int_0^T \int_{\mathbb{T}^1} \langle \bar w_t, -\bar w_t+ g\rangle dx dt= -\pi N^2.
 \end{equation*}
 \end{lem}

\begin{proof}
Again, we first remove the $\mathbf{1}_{\omega}$ term in the equation \eqref{eq_v0} as the control $g$ is supported in the spatial domain $\omega$.
The preceding energy equality comes from direct energy estimate:
\begin{equation*}
    \frac{1}{2}\frac{d}{dt} \int_{\mathbb{T}^1} \left(\bar w_x^2- N^2 \bar w^2 \right)dx= \int_{\mathbb{T}^1} \langle \bar w_t, -\bar w_t+ g\rangle dx.
\end{equation*}
Thus, it suffices to show the existence of control that steers the state from 0 to 1. Notice that since $1$ is not a stationary state of the system \eqref{eq_v0},  this control result is not a direct consequence of the well-known null controllability of heat equations. By defining
\begin{equation*}
    y(t, x):= \bar w(t, x) e^{-N^2 t} \; \forall t\in [0, T], \;\forall x\in \mathbb{T}^1,
\end{equation*}
one obtains
\begin{equation*}
    y_t- \Delta y= e^{- N^2 t} g  \; \forall t\in [0, T],\; \forall x\in \mathbb{T}^1.
\end{equation*}
Hence it suffices to control the state $y(t, \cdot)$ from 0 to $e^{- N^2 T}$. We further define
\begin{equation*}
    z(t, x):= y(t, x)- e^{- N^2 T} \; \forall t\in [0, T], \;\forall x\in \mathbb{T}^1,
\end{equation*}
and obtains
\begin{equation*}
\begin{cases}
 z_t- \Delta z= e^{- N^2 t} g  \; \forall t\in [0, T],\; \forall x\in \mathbb{T}^1, \\
 z(0, x)= - e^{- N^2 T}.
\end{cases}
\end{equation*}
Thanks to the null controllability of the heat equation, see for example \cite{Lebeau-Robbiano-CPDE, Xiang-heat-2020}, we can construct an explicit control $g$ supported in $[0, T]\times \omega$ such that the unique solution satisfies $z(T, \cdot)= 0$.  This finishes the proof of the lemma.
\end{proof}

By selecting
\begin{equation*}
    f^3_1:= g, f^4_1=\ldots = f^{k+1}_1= 0,
\end{equation*}
one obtains
\begin{equation*}
    \bar u^3_1(T, x)= 1, \bar u^4_1(T, x)=\ldots  = \bar u^{k+1}_1(T, x)= 0
\end{equation*}
and one can further prove the following energy estimate:
\begin{lem}\label{lem-en-es}
The solution $\bar u$ of the equation \eqref{eq_full} satisfies
\begin{equation}\label{en-es-varep}
    E(T)- E(0)= -2\pi N^2 \varepsilon^2 + O(\varepsilon^3).
\end{equation}
\end{lem}
\begin{proof}[Proof of Lemma \ref{lem-en-es}]
  Let us define the reminder term $R:= \bar u- \bar u_0- \varepsilon \bar u_1$. One can show that
  \begin{equation}
      \|R\|_{\mathcal{X}_T}\lesssim \varepsilon^2.
  \end{equation}
  Indeed, one easily obtains from the definition of $\bar u, \bar u_0, \bar u_1, R$ and Lemmas \ref{lem:freeheat}--\ref{lem:forcehmhf}   that
  \begin{equation*}
      \|\bar u\|_{\mathcal{X}_T}+   \|\bar u_0\|_{\mathcal{X}_T}+  \|\bar u_1\|_{\mathcal{X}_T}+   \|R\|_{\mathcal{X}_T}\lesssim 1
  \end{equation*}
  and that
\begin{equation*}
    \|(\bar u_t, \bar u_{1t}, R_t)\|_{L^2(0, T; L^2(\mathbb{T}^1))}\lesssim 1.
\end{equation*}

By comparing the equations on $\bar u, \bar u_0$ and $\bar u_1$ we know that $R$ satisfies, where, to simplify the notations we denote $(\bar u, \bar u_0, \bar u_1)$ by $(u, u_0, u_1)$,
\begin{align*}
    \partial_t R- \Delta R&= |u_x|^2u+ (\varepsilon f_1)^{u^{\perp}}- |u_{0x}|^2 u_0- \varepsilon \left( 2(u_{0x\cdot u_{1x}}) u_0+ |u_{0x}|^2 u_1+ f_1- (f_1\cdot u_0) u_0\right),\\
    &= \langle u_{0x}+ \varepsilon u_{1x}+ R_x,  u_{0x}+ \varepsilon u_{1x}+ R_x\rangle ( u_{0x}+ \varepsilon u_{1x}+ R)+ (\varepsilon f_1)^{u^{\perp}}- |u_{0x}|^2 u_0\\
    &\;\;\;\;\;\;\;\;\;\; - \varepsilon \left( 2(u_{0x\cdot u_{1x}}) u_0+ |u_{0x}|^2 u_1+ f_1- (f_1\cdot u_0) u_0\right),\\
    &= \left(|R_x|^2+ 2R_x\cdot (u_{0x}+ \varepsilon u_{1x})\right) u+ |u_{0x}+ \varepsilon u_{1x}|^2 R- \varepsilon \langle f_1, u\rangle R- \varepsilon \langle f_1, R\rangle (u_0+ \varepsilon u_1)\\
    & \;\;\;\;\;\;+ \varepsilon^2 \left(   |u_{1x}|^2 (u_0+ \varepsilon u_1)+ 2 (u_{0x}\cdot u_{1x}) u_1- (f_1\cdot u_1)(u_0+ \varepsilon u_1)- (f_1\cdot u_0) u_1 \right)\\
    &= N(R)+ N_2(\varepsilon),
\end{align*}
where
\begin{align*}
    N(R)&:= \left(|R_x|^2+ 2R_x\cdot (u_{0x}+ \varepsilon u_{1x})\right) u+ |u_{0x}+ \varepsilon u_{1x}|^2 R- \varepsilon \langle f_1, u\rangle R- \varepsilon \langle f_1, R\rangle (u_0+ \varepsilon u_1), \\
    N_2(\varepsilon)&:= \varepsilon^2 \left(   |u_{1x}|^2 (u_0+ \varepsilon u_1)+ 2 (u_{0x}\cdot u_{1x}) u_1- (f_1\cdot u_1)(u_0+ \varepsilon u_1)- (f_1\cdot u_0) u_1 \right).
\end{align*}
According Lemma \ref{lem:freeheat}, for any $T_1\in (0, T]$  the reminder term $R$ satisfies
\begin{align*}
    \|R\|_{\mathcal{X}_{T_1}}\lesssim \|R(0)\|_{H^1}+ \|N(R)+ N_2(\varepsilon)\|_{L^2(0, T_1; L^2(\mathbb{T}^1))}.
\end{align*}
Recall that
\begin{equation*}
    \|(u_x, u_{0x}, u_{1x}, R_x)\|_{L^2(0, T_1; L^{\infty}(\mathbb{T}^1))}\lesssim T_1^{\frac{1}{4}}.
\end{equation*}
Concerning $N(R)$ one has
\begin{gather*}
    \|\left(|R_x|^2+ 2R_x\cdot (u_{0x}+ \varepsilon u_{1x})\right) u\|_{L^2(0, T_1; L^2(\mathbb{T}^1))}\lesssim  T_1^{\frac{1}{4}}\|R\|_{\mathcal{X}_{T_1}}, \\
     \||u_{0x}+ \varepsilon u_{1x}|^2 R\|_{L^2(0, T_1; L^2(\mathbb{T}^1))}\lesssim T_1^{\frac{1}{4}}\|R\|_{L^{\infty}((0, T_1)\times \mathbb{T}^1)} \lesssim  T_1^{\frac{1}{4}}\|R\|_{\mathcal{X}_{T_1}},\\
      \|\varepsilon \langle f_1, u\rangle R\|_{L^2(0, T_1; L^2(\mathbb{T}^1))} \lesssim \varepsilon\|f_1\|_{L^{\infty}(0, T_1; L^2)}\|u\|_{L^{2}(0, T_1; L^{\infty})}\|R\|_{L^{\infty}((0, T_1)\times \mathbb{T}^1)} \lesssim  T_1^{\frac{1}{4}}\|R\|_{\mathcal{X}_{T_1}},\\
       \|\varepsilon \langle f_1, R\rangle (u_0+ \varepsilon u_1)\|_{L^2(0, T_1; L^2(\mathbb{T}^1))} \lesssim  T_1^{\frac{1}{4}}\|R\|_{\mathcal{X}_{T_1}},
\end{gather*}
while for $N_1(\varepsilon)$ one has
\begin{gather*}
    \| |u_{1x}|^2 (u_0+ \varepsilon u_1)\|_{L^2(0, T_1; L^2(\mathbb{T}^1))}\lesssim T_1^{\frac{1}{4}}, \\
     \| (u_{0x}\cdot u_{1x}) u_1\|_{L^2(0, T_1; L^2(\mathbb{T}^1))}\lesssim T_1^{\frac{1}{4}},\\
     \| (f_1\cdot u_1)(u_0+ \varepsilon u_1)\|_{L^2(0, T_1; L^2(\mathbb{T}^1))}\lesssim T_1^{\frac{1}{4}},\\
     \|(f_1\cdot u_0) u_1\|_{L^2(0, T_1; L^2(\mathbb{T}^1))}\lesssim T_1^{\frac{1}{4}}.
\end{gather*}
Hence,
\begin{align*}
     \|R\|_{\mathcal{X}_{T_1}}&\lesssim  \|R(0)\|_{H^1}+ \|N(R)+ N_2(\varepsilon)\|_{L^2(0, T_1; L^2(\mathbb{T}^1))} \\
     &\lesssim \|R(0)\|_{H^1}+ T_1^{\frac{1}{4}}\left( \varepsilon^2+  \|R\|_{\mathcal{X}_{T_1}}\right).
\end{align*}
Selecting $T_1$ small enough, one has
 \begin{equation*}
\|R\|_{\mathcal{X}_{T_1}}\lesssim \|R(0)\|_{H^1}+ T_1^{\frac{1}{4}} \varepsilon^2.
 \end{equation*}
 We iterate this bootstrap argument  and get the required estimate on $R$:
 \begin{equation}
     \|R\|_{\mathcal{X}_{T}}\lesssim \|R(0)\|_{H^1}+ T\varepsilon^2\lesssim \varepsilon^2.
 \end{equation}
By substituting the preceding estimate into the equation on $R$ one further obtains
\begin{align*}
    \|R_t\|_{L^2(0, T; L^2(\mathbb{T}^1))}= \|R_{xx}+ N(R)+ N_2(\varepsilon)\|_{L^2(0, T; L^2(\mathbb{T}^1))}\lesssim \varepsilon^2.
\end{align*}

Now, we come back to the proof of the required  estimate on $E(t)$. One has
\begin{align*}
\frac{1}{2} \frac{d}{dt} E(t)= - \int_{\mathbb{T}^1} \langle u_t, u_{xx} \rangle dx
= \int_{\mathbb{T}^1} -|u_t|^2+ \varepsilon u_t\cdot f_1^{u^{\perp}} dx,
\end{align*}
which implies that
\begin{align*}
   &\;\;\;\;\frac{1}{2} \Big(E(T)- E(0)\Big)\\
   &=  \int_0^T \int_{\mathbb{T}^1} -|u_t|^2+ \varepsilon u_t\cdot f_1^{u^{\perp}} dx dt \\
   &= \int_0^T \int_{\mathbb{T}^1}  - |\varepsilon u_{1t}+ R_t|^2+ \varepsilon \langle \varepsilon u_{1t}+ R_t, f_1- (f_1 \cdot u) u\rangle      dx dt \\
   &= \int_0^T \int_{\mathbb{T}^1}  -|\varepsilon u_{1t}|^2+ \varepsilon^2 \langle u_{1t}, f_1\rangle  - 2(\varepsilon u_{1t}\cdot  R_t)- |R_t|^2- \varepsilon \langle \varepsilon u_{1t}+ R_t,  (f_1 \cdot u) u\rangle+ \varepsilon (R_t\cdot f_1)      dx dt \\
    &= \int_0^T \int_{\mathbb{T}^1}  -|\varepsilon u_{1t}|^2+ \varepsilon^2 \langle u_{1t}, f_1\rangle  - \varepsilon \langle \varepsilon u_{1t},  (f_1 \cdot u) u\rangle     dx dt+ O(\varepsilon^3)\\
    &= \int_0^T \int_{\mathbb{T}^1}  -|\varepsilon u_{1t}|^2+ \varepsilon^2 \langle u_{1t}, f_1\rangle  - \varepsilon \langle \varepsilon u_{1t},  (f_1 \cdot (\varepsilon u_1+ R)) u\rangle     dx dt+ O(\varepsilon^3) \\
     &= \int_0^T \int_{\mathbb{T}^1}  -|\varepsilon u_{1t}|^2+ \varepsilon^2 \langle u_{1t}, f_1\rangle       dx dt+ O(\varepsilon^3) \\
      &= \int_0^T \int_{\mathbb{T}^1}  -|\varepsilon \bar w_{t}|^2+ \varepsilon^2 \langle \bar w_{t}, g\rangle       dx dt+ O(\varepsilon^3)\\
      &= -\frac{\pi}{2} N^2 \varepsilon^2 + O(\varepsilon^3).
\end{align*}
This ends the proof of Lemma \ref{lem-en-es}.
\end{proof}

Finally, thanks to Lemma \ref{lem-en-es}, one immediately gets Lemma \ref{lem_decrea}.
\end{proof}

\section{Quantitative local null controllability and rapid stabilization}\label{sec:local-control-stabilization}
This section is devoted to the local null controllability and the local quantitative rapid stabilization, and  is  also related to \hyperref[Step5]{Step 5} of Section \ref{sec:strategy}.
Using the frequency Lyapunov method introduced in \cite{Xiang-NS-2020, Xiang-heat-2020}, we first show that our geometric (nonlinear)  system \eqref{eq:controlhmhf} is locally rapidly stabilizable. Then, relying on our time-periodic  feedback control, we further obtain the local null controllability. This local controllability result, to be combined with the previous proved global approximate controllability, lead to the global null controllability. As commented in Remark \ref{rmk:Liu}, the local null controllability problem of the harmonic map heat flow has been studied in \cite[Proposition 2.6]{Liu-2018}. Here we introduce a different method to prove this null controllability. This new approach shares the advantage of presenting quantitative cost estimates.

\subsection{An iteration approach}
It sounds natural to rely on the iteration argument introduced in \cite{Krieger-Xiang-2022} for the local exact controllability of the wave maps equations. Heuristically speaking, one shall construct a sequence of ``approximate solutions" $\{u^n\}_{n}$ having  initial state  $u^n(t= 0)= u_0$ and final state $u^n(t= T)= \mathfrak{N}+ e^n$, where $\mathfrak{N}:= (0,\ldots ,0, 1)^T$ is the constant map equal to the North pole and $e^n\rightarrow 0$,
\begin{equation*}
   \begin{cases}
        u^{n}_t- \Delta u^{n}= |u^n_x|^2 u^n- \mathbf{1}_{\omega}\langle f^n, u^n\rangle u^n+ \mathbf{1}_{\omega}f^{n}, \\
        u^{n}(0, \cdot)= u_0(\cdot), \; u^{n+1}(T, \cdot)= \mathfrak{N}+ e^{n}
    \end{cases}
\end{equation*}
and find an exact solution by passing to the limit $n\rightarrow +\infty$.   The idea is that for the $n$-th iterate $(u^n, f^n)$ we first consider a correction $(v^n, h^n)$ which satisfies the linear controlled heat equation such that it eliminates the error term of $(u^n, f^n)$, namely $e^n$. Thus the state $(u^n+ v^n, f^n+ h^n)$ almost solves the controlled harmonic map heat equation, moreover the error is smaller than the preceding one. Then we correct the error source term to find an exact solution $(u^{n+1}, f^{n+1})$ such that its final state $e^{n+1}$ becomes smaller than $e^n$, where $u^{n+1}\approx u^n+ v^n$ and $f^{n+1}= f^n+ h^n$.  This strategy was successfully applied on the controlled wave equation.
However, in the current setting, due to strong smoothing effect of the heat equation the heat equation is null controllable but not exactly controllable. Consequently, it becomes more complicated when one aims to construct the correction term $v_n$.  Independently of the approximate solutions illustrated above, it is equally natural to consider iterations of the form:
\begin{equation*}
    \begin{cases}
        u^{n+1}_t- \Delta u^{n+1}= |u^n_x|^2 u^n- \mathbf{1}_{\omega}\langle f^n, u^n\rangle u^n+ \mathbf{1}_{\omega} f^{n+1}, \\
        u^{n+1}(0,  \cdot)= u_0( \cdot),
    \end{cases}
\end{equation*}
or of the form
\begin{equation*}
    \begin{cases}
        u^{n+1}_t- \Delta u^{n+1}- |u^n_x|^2 u^{n+1}= - \mathbf{1}_{\omega}\langle f^n, u^n\rangle u^n+ \mathbf{1}_{\omega}f^{n+1}, \\
        u^{n+1}(0,  \cdot)= u_0( \cdot).
    \end{cases}
\end{equation*}
Nevertheless, in these circumstances, some uniform decay properties of the source term should be involved, making the analysis more delicate.

\subsection{The stereographic projection }
Instead of the iteration method used above, we now use the stereographic projection from the South pole onto the tangent space $T_{\mathfrak{N}} \Sph^k$ of the North pole $\mathfrak{N}= (0,\ldots,0, 1)^T$ to the sphere $\Sph^k$ ($T_{\mathfrak{N}} \Sph^k=\{(x_1,\ldots,x_k,1)^T; (x_1,\ldots,x_k)^T\in \R^k\, \}\simeq \R^k$). It maps the state $u$ which belongs to the geometric target $\Sph^k$ onto this tangent space. Let us remark that the stereographic projection is a standard method for the study of the harmonic map heat flow with sphere targets. Recently, this technique has  been used by Liu to study of the local controllability of harmonic map heat flow  \cite[Proposition 2.6]{Liu-2018}.  Taking advantage of this projection, the controlled harmonic map heat flow equation becomes  the controlled nonlinear heat equation (without any constraint), the reason for which will be rendered later on,
\begin{equation}
    \partial_t v- \Delta v+ \frac{2 s_x}{4+s} v_x- \frac{2|v_x|^2}{4+s} v= \mathbf{1}_{\omega}g \; \textrm{ where } v, g\in \mathbb{R}^k.   \notag
\end{equation}
More precisely, we define the bijection
\begin{align}\label{sphe_project}
    \mathbb{P}: \Sph^k\setminus \{(0,0,\ldots , -1)^T\}&\longrightarrow \mathbb{R}^k, \\
    (u^1,\ldots , u^{k+1})^T&\mapsto (v^1, \ldots , v^k)^T,   \notag
\end{align}
for which, by denoting
$$s:= \sum_{i=1}^k (v^i)^2,$$
one has
\begin{equation*}
    (u^1, \ldots , u^{k+1})^T= \mathbb{P}^{-1}(v^1,\ldots , v^k)^T= \left(\frac{4v^1}{4+ s},\ldots , \frac{4v^k}{4+s}, \frac{4-s}{4+s}\right)^T.
\end{equation*}
Therefore, for any given function $v(t, x)\in \mathbb{R}^k$ the corresponding function $u(t, x)$ on $\Sph^k\subset \mathbb{R}^{k+1}$ satisfies that for  $\forall i\in \{1, \ldots , k\}$,
\begin{align*}
    \partial_t u^i&= 4 \left(
\frac{ v^i_t}{4+s}- \frac{v^i s_t}{(4+s)^2} \right), \\
 \partial_x u^i&= 4 \left(
\frac{ v^i_x}{4+s}- \frac{v^i s_x}{(4+s)^2} \right), \\
 \partial_{xx} u^i&= 4 \left(
\frac{ v^i_{xx}}{4+s}- \frac{2v^i_x s_x+ v^i s_{xx}}{(4+s)^2}+ \frac{2 v^i (s_x)^2}{(4+s)^3} \right), \\
\partial_t u^{k+1}&= -8 \frac{s_t}{(4+s)^2}, \\
\partial_x u^{k+1}&= -8 \frac{s_x}{(4+s)^2}, \\
\partial_{xx} u^{k+1}&= -8 \left(\frac{s_{xx}}{(4+s)^2}- \frac{2(s_x)^2}{(4+s)^3}  \right).
\end{align*}
Notice that
\begin{align*}
    |u_x|^2&= \sum_{i= 1}^k (u^i_x)^2+  (u^{k+1}_x)^2\\
    &= 16 \sum_{i+1}^k \left( \frac{ v^i_x}{4+s}- \frac{v^i s_x}{(4+s)^2}\right)^2+ 16 \frac{4 (s_x)^2}{(4+s)^4} \\
    &= \frac{16 |v_x|^2}{(4+ s)^2}.
\end{align*}
Defining $\bar u:= (u^1, \ldots , u^k)^T$, thanks to the preceding equations, one has
\begin{align*}
    &\;\;\;\; \bar u_t- \Delta \bar u- |u_x|^2 \bar u\\
    &= \frac{4}{4+s} v_t- \frac{4 s_t}{(4+ s)^2} v- \frac{4}{4+s} \Delta v+ \frac{8 s_x}{(4+ s)^2} v_x+ \frac{4s_{xx}}{(4+s)^2} v- \frac{8 (s_x)^2}{(4+s)^3} v- \frac{64 |v_x|^2}{(4+ s)^3} v\\
    &=  \frac{4}{4+s} v_t- \frac{4}{4+s} \Delta v+ \frac{8 s_x}{(4+ s)^2} v_x- \frac{8 |v_x|^2}{(4+ s)^2} v\\
    &\;\;\; + \left(- \frac{4 s_t}{(4+ s)^2} + \frac{4s_{xx}}{(4+s)^2} - \frac{8 (s_x)^2}{(4+s)^3} - \frac{8 (4-s)|v_x|^2}{(4+ s)^3}\right) v,
\end{align*}
and
\begin{align*}
 u_t^{k+1}- \Delta  u^{k+1}- |u_x|^2 u^{k+1}=
       - \frac{8 s_t}{(4+ s)^2} + \frac{8 s_{xx}}{(4+s)^2} - \frac{16 (s_x)^2}{(4+s)^3} - \frac{16 (4-s)|v_x|^2}{(4+ s)^3}.
\end{align*}
Inspired by the preceding formulas, we define
\begin{equation}
    I:=  - \frac{8 s_t}{(4+ s)^2} + \frac{8 s_{xx}}{(4+s)^2} - \frac{16 (s_x)^2}{(4+s)^3} - \frac{16 (4-s)|v_x|^2}{(4+ s)^3},
\end{equation}
and obtain
\begin{align*}
\begin{cases}
\bar u_t- \Delta \bar u- |u_x|^2 \bar u&= \displaystyle \frac{4}{4+s} \left(\partial_t v- \Delta v+ \frac{2 s_x}{4+s} v_x- \frac{2|v_x|^2}{4+s} v\right)+ \frac{1}{2} I v, \\
 u_t^{k+1}- \Delta  u^{k+1}- |u_x|^2 u^{k+1}& =I.
\end{cases}
\end{align*}

Conversely, suppose that for some control function $g$ supported in $[0, T]\times \omega$, the function $v$ satisfies
\begin{equation}\label{eq:transeqv}
    \partial_t v- \Delta v+ \frac{2 s_x}{4+s} v_x- \frac{2|v_x|^2}{4+s} v= \mathbf{1}_{\omega}g= g.
\end{equation}
Remark that for any given initial state $v(0, \cdot)$ and control function $g$, the preceding equation admits a unique solution.
By plugging this equation into the definition of $I$ one obtains
\begin{align*}
    -\frac{1}{8} I&=   \frac{ s_t}{(4+ s)^2} - \frac{ s_{xx}}{(4+s)^2} + \frac{2 (s_x)^2}{(4+s)^3} + \frac{2 (4-s)|v_x|^2}{(4+ s)^3} \\
    &= \frac{1}{(4+s)^2} \left(s_t- s_{xx}+ \frac{2(s_x)^2}{4+s}+ \frac{2(4- s) |v_x|^2}{4+s}   \right) \\
    &= \frac{1}{(4+s)^2} \left(2 v\cdot ( \Delta v- \frac{2 s_x}{4+s} v_x+ \frac{2|v_x|^2}{4+s} v+ g)- 2|v_x|^2- 2 v v_{xx}+ \frac{2(s_x)^2}{4+s}+ \frac{2(4- s) |v_x|^2}{4+s}   \right) \\
    &= \frac{2 vg}{(4+ s)^2}.
\end{align*}
This further implies that
$u= \mathbb{P}^{-1} v$ satisfies
\begin{align*}
\begin{cases}
\bar u_t- \Delta \bar u- |u_x|^2 \bar u&= \frac{4}{4+s} g-  \frac{8 vg}{(4+ s)^2}  v, \\
 u_t^{k+1}- \Delta  u^{k+1}- |u_x|^2 u^{k+1}& = -\frac{16 vg}{(4+ s)^2}.
\end{cases}
\end{align*}
One observes from this preceding equation that in the uncontrolled domain, where the value of $g$ vanishes, one has
\begin{equation*}
    u_t-\Delta u- |u_x|^2 u= 0,
\end{equation*}
and that in the controlled domain $\omega$ one has
\begin{equation*}
    u_t-\Delta u- |u_x|^2 u=
    \begin{pmatrix}
        \frac{4}{4+s} g-  \frac{8 vg}{(4+ s)^2}  v\\ -\frac{16 vg}{(4+ s)^2}
    \end{pmatrix}
\end{equation*}
since
\begin{equation*}
    \begin{pmatrix}
        \frac{4}{4+s} g-  \frac{8 vg}{(4+ s)^2}  v\\ -\frac{16 vg}{(4+ s)^2}
    \end{pmatrix}  \cdot u= \begin{pmatrix}
        \frac{4}{4+s} g-  \frac{8 vg}{(4+ s)^2}  v\\ -\frac{16 vg}{(4+ s)^2}
    \end{pmatrix}  \cdot
     \begin{pmatrix}
       \frac{4}{4+s} \bar u\\ \frac{4- s}{4+s}
    \end{pmatrix}= 0.
\end{equation*}
Namely, it suffices to investigate the control system \eqref{eq:transeqv} on $v$, since this automatically provides an explicit trajectory to the control system on $u$.

\subsection{Quantitative rapid stabilization and null controllability}
As illustrated above, it suffices to study the transformed system \eqref{eq:transeqv}. Actually, we are able to prove the following quantitative result:
\begin{prop}\label{prop:nullcontrol}
    Let $T\in (0, 1]$. The controlled system
   \begin{equation}
    \partial_t v- \Delta v+ \frac{2 v \cdot v_x}{4+ |v^2|} v_x- \frac{2|v_x|^2}{4+ |v|^2} v= \mathbf{1}_{\omega} g \notag
\end{equation}
is locally null controllable in the sense that, there exists some effectively computable $C>1$ which is independent of the value of $T\in (0, 1]$ such that for any initial state $v(0, \cdot)$ satisfying
\begin{equation*}
    \|v(0, \cdot)\|_{H^1(\mathbb{T}^1)}\leq e^{-\frac{C}{T}},
\end{equation*}
we can construct an explicit control $g\in L^{\infty}(0, T; L^2(\mathbb{T}^1))$ satisfying
\begin{equation}
 \|g\|_{L^{\infty}(0, T; L^2(\mathbb{T}^1))}\leq e^{\frac{C}{T}} \|v(0, \cdot)\|_{H^1(\mathbb{T}^1)},  \notag
\end{equation}
 such that the unique solution $v\in C([0, T]; H^1(\mathbb{T}^1))\cap L^2(0, T; H^2(\mathbb{T}^1))$ of the controlled harmonic map heat flow satisfies $v(T, \cdot)= 0$.
\end{prop}
We use the Frequency Lyapunov method introduced in \cite{Xiang-NS-2020, Xiang-heat-2020} to prove this local null controllability result.  This stabilization-based approach provides  constructive and explicit controls to achieve both quantitative rapid stabilization and finite time stabilization (thus consequently null controllability).

For any given constant $\lambda>0$, we define  the low-frequency projection
\begin{align}
\label{defPlambda}
    P_{\lambda} v&:= \frac{1}{\pi}\left(\sum_{1\leq n_2\leq \sqrt{\lambda}} \langle v, \cos(n_2x)\rangle_{L^2(\mathbb{T}^1)} \cos (n_2 x) \right)+ \frac{1}{2\pi} \langle v, 1\rangle_{L^2(\mathbb{T}^1)}   \notag \\
     & \;\;\;\;\;\;\; \;\;\; \;\;\;\;\;\;\;\;\; \;\;\; \;\;\;\;\;\;\;\;\; + \frac{1}{\pi} \left(\sum_{1\leq n_1\leq \sqrt{\lambda}} \langle v, \sin(n_1x)\rangle_{L^2(\mathbb{T}^1)} \sin (n_1x)\right).
\end{align}
Recall that $\{\sin (nx)/\sqrt{\pi}\}_{n\in \mathbb{N}\setminus\{0\}}\cup \{\cos (nx)\sqrt{\pi}\}_{n\in \mathbb{N}\setminus\{0\}}\cup\{1/\sqrt{2\pi}\}$ is an orthonormal basis made of eigenfunctions of the Laplace operator on the torus $\mathbb{T}^1$.
The essential step of the frequency Lyapunov  method is the following quantitative rapid stabilization result:
\begin{lem}\label{lem-rapid-stab}
    There exist effectively computable constants $C_0, C_1>1$   such that for any $\lambda>1$, the unique solution of  the closed-loop system
    \begin{equation}
    \partial_t v- \Delta v+ \frac{2 v \cdot v_x}{4+ |v^2|} v_x- \frac{2|v_x|^2}{4+ |v|^2} v= -  \lambda e^{C_0 \sqrt{\lambda}} \mathbf{1}_{\omega} P_{\lambda} v,
\end{equation}
decays exponentially as
\begin{equation}
   \|v(t, \cdot)\|_{H^1}\leq e^{2C_0 \sqrt{\lambda}} e^{-\frac{\lambda}{4} t}  \|v(0, \cdot)\|_{H^1} \; \forall t\in (0, +\infty),  \notag
\end{equation}
provided that the initial state $v(0, \cdot)$ satisfies
\begin{equation}
    \|v(0, \cdot)\|_{H^1}\leq C_1^{-1} e^{-6 C_0 \sqrt{\lambda}}.   \notag
\end{equation}
\end{lem}
\begin{remark}\label{rem:rapid-lem-thm}
This lemma is devoted to the rapid stabilization of $v= \mathbb{P} u$ (recall the definition of $\mathbb{P}$ in equation \eqref{sphe_project}, and the relation between the equation on $u$ and the equation on $v$). After applying the inverse of the stereographic projection, the function $u$ satisfies the following closed-loop system
\begin{equation*}
    u_t-\Delta u- |u_x|^2 u= \mathbf{1}_{\omega}  F(u):=
     \mathbf{1}_{\omega}
    \begin{pmatrix}
        \frac{4}{4+s} g-  \frac{8 vg}{(4+ s)^2}  v\\ -\frac{16 vg}{(4+ s)^2}
    \end{pmatrix}
\end{equation*}
where
\begin{equation*}
    v= \mathbb{P}u, \; s= |\mathbb{P} u|^2,\;  g= -  \lambda e^{C_0 \sqrt{\lambda}} \mathbf{1}_{\omega} P_{\lambda} (\mathbb{P} u).
\end{equation*}
Since $v$ decays rapidly so does $u$. This automatically gives the proof of Theorem \ref{thm-rapid-stab} on the local rapid stabilization of the harmonic map heat flow around $\mathfrak{N}= (0,\ldots,0,1)^T$. The stabilization around other points is a direct consequence of the rotation invariance of the forced harmonic map heat flow.
\end{remark}
We first provide the proof of Proposition \ref{prop:nullcontrol} with the help of Lemma \ref{lem-rapid-stab}, while this auxiliary lemma will be shown later.
\begin{proof}[Proof of Proposition \ref{prop:nullcontrol}]
Armed with the preceding lemma, the construction of  controls leading to null controllability of the system is an explicit iteration procedure.  We refer to Section 4 of the article \cite[Theorem 4.1]{Xiang-NS-2020} by the last author  concerning this precise construction. Roughly speaking, one shall cut the time interval $(0, T)$ by infinitely many pieces:
   \begin{equation}
    (0, T)= \cup_{k= 1}^{\infty} (t_k, t_{k+1}],  \notag
    \end{equation}
and on each piece construct a feedback law as $ \mathbf{1}_{\omega} g:= -  \lambda_k e^{C_0 \sqrt{\lambda_k}} \mathbf{1}_{\omega} P_{\lambda_k} v$  such that the solution decays exponentially fast with rate $\lambda_k/4$ in the period $t\in (t_k, t_{k+1})$.  After a careful selection of the values of $\{t_k, \lambda_k\}_{k\in \mathbb{N}^*}$ as well as a constant $C>1$, we  finally show that for  any  initial state $v(0, \cdot)$ satisfying
\begin{equation}
\|v(0, \cdot)\|_{H^1(\mathbb{T}^1)}\leq e^{- \frac{C}{T}} \notag
\end{equation}
the unique solution $v(t)|_{t\in (0, T)}$ of the equation
 \begin{equation*}
\partial_t v- \Delta v+ \frac{2 v \cdot v_x}{4+ |v^2|} v_x- \frac{2|v_x|^2}{4+ |v|^2} v= -  \lambda_k e^{C_0 \sqrt{\lambda_k}} \mathbf{1}_{\omega} P_{\lambda_k} v \; \forall t\in (t_k, t_{k+1}], \forall k\in \mathbb{N}^*,
\end{equation*}
satisfies:
\begin{gather*}
\|v(t_k, \cdot)\|_{H^1(\mathbb{T}^1)}\leq C_1^{-1} e^{-6 C_0 \sqrt{\lambda_k}} \; \forall k\in \mathbb{N}^*, \\
 \|v(t, \cdot)\|_{H^1(\mathbb{T}^1)}\leq e^{2C_0 \sqrt{\lambda_k}} e^{-\frac{\lambda_k}{4} t}  \|v(t_k, \cdot)\|_{H^1(\mathbb{T}^1)} \; \forall t \in (t_k, t_{k+1}], \forall k\in \mathbb{N}^*,\\
\|\mathbf{1}_{\omega} g(t)\|_{L^2(\mathbb{T}^1)}= \| \lambda_k e^{C_0 \sqrt{\lambda_k}} \mathbf{1}_{\omega} P_{\lambda_k} v\|_{L^2(\mathbb{T}^1)}\leq e^{\frac{C}{T}}\; \forall t\in (t_k, t_{k+1}], \forall k\in \mathbb{N}^*,\\
v(t, \cdot)\xrightarrow{t\rightarrow T^{-}} 0 \textrm{ in } H^1(\mathbb{T}^1).
\end{gather*}
This finishes the proof of the quantitative null controllability property.
\end{proof}

\begin{proof}[Proof of Lemma \ref{lem-rapid-stab}]
    According to the frequency Lyapunov method,  one  directly considers
the frequency-based Lyapunov function
\begin{equation}
    V_{\lambda}(v):= \mu_{\lambda} \|P_{\lambda} v\|_{L^2(\mathbb{T}^1)}^2+ \|\partial_x P_{\lambda}^{\perp} v\|_{L^2(\mathbb{T}^1)}^2,
\end{equation}
as well as the following choice of  feedback law
\begin{equation}
    g(v):= - \gamma_{\lambda} P_{\lambda} v,
\end{equation}
with the exact value of $\gamma_{\lambda}, \mu_{\lambda}\in [\lambda, +\infty)$  to be fixed later on.
 Remark that the Lyapunov function $V_{\lambda}$ is equivalent to the $H^1$-norm, since
\begin{equation}
    \frac{1}{2} \|v\|_{H^1(\mathbb{T}^1)}^2\leq V_{\lambda}(v)\leq \mu_{\lambda} \|v\|_{H^1(\mathbb{T}^1)}^2.
\end{equation}

In the following context, if we do not emphasize $L^2(\omega)$, then the  $L^2$ notation infers to $L^2(\mathbb{T}^1)$.  Let us consider the variation of $V_{\lambda}(v)$ with respect to time:
\begin{align*}
    \frac{d}{dt} V_{\lambda}(v)&= 2 \mu_{\lambda} \big\langle  P_{\lambda} v, \frac{d}{dt}  P_{\lambda} v\big\rangle_{L^2}+ 2  \big\langle \partial_x P_{\lambda}^{\perp} v, \frac{d}{dt} \partial_x P_{\lambda}^{\perp} v\big\rangle_{L^2} \\
    &= 2 \mu_{\lambda} \big\langle  P_{\lambda} v, \frac{d}{dt}
 v\big\rangle_{L^2} - 2  \big\langle \Delta P_{\lambda}^{\perp} v, \frac{d}{dt} v\big\rangle_{L^2}.
\end{align*}

By denoting
\begin{equation}
    a(v, v_x):= - \frac{2 v \cdot v_x}{4+ |v^2|} v_x+ \frac{2|v_x|^2}{4+ |v|^2} v,    \notag
\end{equation}
there is
\begin{equation}
    \|a(v, v_x)\|_{L^2}\lesssim \|v_x\|_{L^4}^2\lesssim \|v_x\|_{L^2}^{\frac{3}{2}}\|\Delta v\|_{L^2}^{\frac{1}{2}}.    \notag
\end{equation}
We also recall the so-called spectral inequality (see \cite{Lebeau-Robbiano-CPDE} for a general result on this inequality):
\begin{lem}[Spectral inequality, \cite{Lebeau-Robbiano-CPDE}]
There exists some constant $C_0$ independent of $\lambda>1$ such that
\begin{equation}
    \|P_{\lambda} f\|_{L^2(\omega)}\geq e^{- C_0\sqrt{\lambda}}  \|P_{\lambda} f\|_{L^2(\mathbb{T}^1)} \;  \forall f\in L^2(\mathbb{T}^1),\;  \forall \lambda>1.
\end{equation}
\end{lem}
Lastly, we recall the Young's inequalities
\begin{equation*}
    a^{\frac{3}{2}} b^{\frac{1}{2}}\leq \frac{3 a^2+ b^2}{4},
 \; a^{\frac{3}{2}} b^{\frac{3}{2}}\leq \frac{3 a^2+ b^6}{4}
 \; \textrm{ and } \; a^{\frac{1}{2}} b^{\frac{5}{2}}\leq \frac{ a^2+ 3b^{\frac{10}{3}}}{4}.
\end{equation*}

On the one hand,
\begin{align*}
    - 2  \big\langle \Delta P_{\lambda}^{\perp} v, \frac{d}{dt} v\big\rangle_{L^2}&= - 2  \big\langle \Delta P_{\lambda}^{\perp} v, \Delta v+ a(v, v_x)- \mathbf{1}_{\omega} \gamma_{\lambda} P_{\lambda} v\big\rangle_{L^2}\\
    &\leq -2 \|\Delta P_{\lambda}^{\perp} v\|_{L^2}^2+ 2\|\Delta P_{\lambda}^{\perp} v\|_{L^2} \|a(v, v_x)\|_{L^2}+ \gamma_{\lambda} \|\Delta P_{\lambda}^{\perp} v\|_{L^2}\| P_{\lambda} v\|_{L^2} \\
    &\leq -2 \|\Delta P_{\lambda}^{\perp} v\|_{L^2}^2+ C\|\Delta v\|_{L^2}^{\frac{3}{2}} \|v_x\|_{L^2}^{\frac{3}{2}}+ \gamma_{\lambda} \|\Delta P_{\lambda}^{\perp} v\|_{L^2}\| P_{\lambda} v\|_{L^2} \\
    &\leq -\frac{3}{2} \|\Delta P_{\lambda}^{\perp} v\|_{L^2}^2+  \gamma_{\lambda}^2 \| P_{\lambda} v\|_{L^2}^2+ C\|\Delta v\|_{L^2}^{\frac{3}{2}} \|v_x\|_{L^2}^{\frac{3}{2}}.
\end{align*}

On the other hand, thanks to the spectral inequality and the standard energy estimates,
\begin{align*}
 2  \big\langle  P_{\lambda} v, \frac{d}{dt} v\big\rangle_{L^2}
 &=  2  \big\langle  P_{\lambda} v, \Delta v+ a(v, v_x)- \mathbf{1}_{\omega} \gamma_{\lambda} P_{\lambda} v\big\rangle_{L^2} \\
 &= -2 \|\partial_x P_{\lambda} v\|_{L^2}^2- 2\gamma_{\lambda} \big\langle P_{\lambda} v, P_{\lambda} v\big\rangle_{L^2(\omega)}+ 2\big\langle  P_{\lambda} v, a(v, v_x)\big\rangle_{L^2} \\
 &\leq - 2\gamma_{\lambda} e^{- C_0\sqrt{\lambda}}  \|P_{\lambda} v\|_{L^2}^2- 2 \|\partial_x P_{\lambda} v\|_{L^2}^2+ C\|P_{\lambda} v\|_{L^2} \|v_x\|_{L^2}^{\frac{3}{2}}\|\Delta v\|_{L^2}^{\frac{1}{2}} \\
  &\leq - 2\gamma_{\lambda} e^{- C_0\sqrt{\lambda}}  \|P_{\lambda} v\|_{L^2}^2- 2 \|\partial_x P_{\lambda} v\|_{L^2}^2+ C\| v\|_{H^1}^{\frac{5}{2}} \|\Delta v\|_{L^2}^{\frac{1}{2}}.
\end{align*}

Recall that by the definition of  the projection $P_{\lambda}$
\begin{equation}
    -\|\Delta P_{\lambda}^{\perp} v\|_{L^2}^2\leq -  \lambda \|\partial_x P_{\lambda}^{\perp} v\|_{L^2}^2.   \notag
\end{equation}
By combing the preceding estimates together one obtains
\begin{align*}
 \frac{d}{dt} V_{\lambda}(v)&\leq     - 2 \mu_{\lambda} \gamma_{\lambda} e^{- C_0\sqrt{\lambda}}  \|P_{\lambda} v\|_{L^2}^2- 2 \mu_{\lambda}\|\partial_x P_{\lambda} v\|_{L^2}^2+ C\mu_{\lambda}\| v\|_{H^1}^{\frac{5}{2}} \|\Delta v\|_{L^2}^{\frac{1}{2}} \\
 &\;\;\;\; \;\;\; -\frac{3}{2} \|\Delta P_{\lambda}^{\perp} v\|_{L^2}^2+  \gamma_{\lambda}^2 \| P_{\lambda} v\|_{L^2}^2+ C\|\Delta v\|_{L^2}^{\frac{3}{2}} \|v_x\|_{L^2}^{\frac{3}{2}} \\
 &= -\|\Delta P_{\lambda}^{\perp} v\|_{L^2}^2 - 2 \mu_{\lambda} \gamma_{\lambda} e^{- C_0\sqrt{\lambda}}  \|P_{\lambda} v\|_{L^2}^2   + \gamma_{\lambda}^2 \| P_{\lambda} v\|_{L^2}^2
\\
&  -\frac{1}{2} \|\Delta P_{\lambda}^{\perp} v\|_{L^2}^2- 2 \mu_{\lambda}\|\partial_x P_{\lambda} v\|_{L^2}^2+ C\mu_{\lambda}\| v\|_{H^1}^{\frac{5}{2}} \|\Delta v\|_{L^2}^{\frac{1}{2}} + C\|\Delta v\|_{L^2}^{\frac{3}{2}} \|v_x\|_{L^2}^{\frac{3}{2}} \\
&\leq -\|\Delta P_{\lambda}^{\perp} v\|_{L^2}^2 - 2 \mu_{\lambda} \gamma_{\lambda} e^{- C_0\sqrt{\lambda}}  \|P_{\lambda} v\|_{L^2}^2   + \gamma_{\lambda}^2 \| P_{\lambda} v\|_{L^2}^2
\\
&  \;\;\;\;\;\;\;\; -\frac{1}{2} \|\Delta  v\|_{L^2}^2+ C\mu_{\lambda}\| v\|_{H^1}^{\frac{5}{2}} \|\Delta v\|_{L^2}^{\frac{1}{2}} + C\|\Delta v\|_{L^2}^{\frac{3}{2}} \|v_x\|_{L^2}^{\frac{3}{2}} \\
&\leq -\lambda \|\partial_x P_{\lambda}^{\perp} v\|_{L^2}^2 - 2 \mu_{\lambda} \gamma_{\lambda} e^{- C_0\sqrt{\lambda}}  \|P_{\lambda} v\|_{L^2}^2   + \gamma_{\lambda}^2 \| P_{\lambda} v\|_{L^2}^2
\\
& \;\;\;\;\;\;\;\; + C\mu_{\lambda}^{\frac{4}{3}}\| v\|_{H^1}^{\frac{10}{3}}+ C \|v\|_{H^1}^{6},
\end{align*}
where both $C$ and $C_0$ are independent of the choice of $\lambda\in (1, +\infty)$.
By selecting
\begin{equation}
    \gamma_{\lambda}= \lambda e^{C_0 \sqrt{\lambda}} \textrm{ and } \mu_{\lambda}= \lambda e^{2 C_0 \sqrt{\lambda}},
\end{equation}
we have
\begin{equation}
    \frac{d}{dt} V_{\lambda}(v)\leq- \lambda V_{\lambda} + C \left(e^{4C_0 \sqrt{\lambda}} V_{\lambda}^{\frac{5}{3}}+ V_{\lambda}^3 \right).
\end{equation}
Using this {\it a priori} energy estimate one can easily conclude the local (in time) existence and uniqueness of the solution for the closed-loop system.  Moreover, for any initial state that satisfies
\begin{equation}
    \|v\|_{H^1}\leq C_1^{-1} e^{-6 C_0 \sqrt{\lambda}},
\end{equation}
this unique solution  decays exponentially:
\begin{equation*}
    V_{\lambda}(v(t))\leq e^{- \frac{\lambda}{2} t} V_{\lambda}(v(0)),
\end{equation*}
and
\begin{equation*}
    \|v(t)\|_{H^1}\leq e^{2C_0 \sqrt{\lambda}} e^{-\frac{\lambda}{4} t}  \|v(0)\|_{H^1} \; \forall t\in (0, +\infty).
\end{equation*}

Thus we finish the proof of Lemma \ref{lem-rapid-stab}.
\end{proof}

\section{Small-time global exact controllability between harmonic maps}\label{sec:GECHM}

\subsection{The sphere target case}
This part is devoted to the proofs of Theorem \ref{lem:excconhm0} and Theorem  \ref{lem:excconhm}.

We first investigate the simplest case, namely, the harmonic map heat flow $\mathbb{T}^1\rightarrow \Sph^1$, for which, benefiting from the polar coordinates, the system is transformed into the linear heat equation with an interior control.
Next, we show that for a sphere $\Sph^k$, if the initial state is in a given closed geodesic and the control force is tangent to this geodesic, then the unique solution to the harmonic map heat flow remains in the same closed geodesic.
Moreover, under the polar coordinates, the system becomes again the controlled heat equation, for which the  controllability is well-investigated.

Since $\gamma$ is a closed constant-speed geodesic on $\mathbb{S}^k$ with energy $2\pi$, there exists an orthogonal matrix $A\in O(k+1)$ such that
\begin{equation*}
    A \gamma (x)= \varphi (x) \;
 \textrm{ where } \varphi(x):= (\cos{x}, \sin{x}, 0,\ldots , 0)^T.
\end{equation*}
In the following, without loss of generality, we assume that $\gamma= \varphi$, and that $\omega= (2\pi- \delta, 2\pi)$. We denote by $N$ the integer such that $\deg (v_0, \mathbb{T}^1, \mathcal{C})= N$.\\

(i).  We first prove Theorem \ref{lem:excconhm0}. Recall that this result  is  related to \hyperref[Step6]{Step 6} of Section \ref{sec:strategy}. We express  the controlled harmonic map heat flow in term of  the polar coordinates as follows:
\begin{gather*}
    u(t, x)= (\cos{\theta}, \sin{\theta}, 0,\ldots , 0)^{T} \; \textrm{ with } \theta= \theta(t, x)\in \mathbb{R},\\
    f^{u^{\perp}}(t, x)= h(t, x) (-\sin{\theta}, \cos{\theta},0,\ldots , 0)^T \;  \textrm{ with $h(t, x)\in \mathbb{R}$ and supp } [0, T]\times  \omega.
\end{gather*}
Thus the function $\theta(t, x)|_{t\in [0, T], x\in [0, 2\pi]}$ satisfies
\begin{equation*}
\begin{cases}
        \theta_t(t, x)-  \theta_{xx}(t, x)= h(t, x), \\
        \theta(t, 2\pi)= \theta(t, 0) \textrm{ mod } 2\pi,\\
          \theta_x(t, 2\pi)= \theta_x(t, 0),\\
        \theta(0, x)= \theta_0(x)
        \end{cases}
\end{equation*}
with
\begin{gather*}
    h(t, x)\in L^{\infty}(0, T; L^2([0, 2\pi]; \mathbb{R})).
\end{gather*}
Clearly, for initial state $\theta_0$ satisfying $\theta_0(2\pi)= \theta_0(0)+ 2\pi N$, the unique solution $\theta$ satisfies
\begin{gather*}
        \theta\in C([0, T]; H^1([0, 2\pi]; \mathbb{R}))\cap L^2([0, T]; H^2([0, 2\pi]; \mathbb{R})), \\
        \theta(t, 2\pi)= \theta(t, 0)+ 2\pi N\; \forall t\in [0, T].
\end{gather*}
By considering
\begin{equation*}
    w(t, x):= \theta(t, x)-  Nx- r,
\end{equation*}
we obtain
\begin{equation*}
\begin{cases}
       w_t(t, x)-  w_{xx}(t, x)= h(t, x), \\
        w(t, 0)= w(t, 2\pi), \\
        w_x(t, 0)= w_x(t, 2\pi),
        \end{cases}
\end{equation*}
  which is a controlled heat equation on $\mathbb{T}^1$. Similarly to the proof of Proposition \ref{prop:nullcontrol}, we can construct an explicit control $h\in L^{\infty}(0, T; L^2)$ such that $w(T, x)= 0$. Actually, we can directly cite the results given in  \cite{1971-Fattorini-Russell-ARMA,  1995-Guo-Littman, 1977-Jones-Frank-JMAA,  1978-Littman-ASNSP,  2014-Martin-Rosier-Rouchon-A,  2016-Coron-Nguyen-ARMA, 2016-Martin-Rosier-Rouchon-SICON, MR4153111} concerning the one-dimensional heat equations. Therefore, the final state of the forced harmonic map heat flow equation is
  \begin{equation*}
      u(T, x)= (\cos \theta(T, x), \sin \theta(T, x),0,\ldots , 0)^T= (\cos (Nx+r), \sin (Nx+r),0,\ldots , 0)^T.
  \end{equation*}
This finishes the proof of Theorem \ref{lem:excconhm0}.
\\

(ii).  Next we turn to the proof of  Theorem \ref{lem:excconhm} keeping in mind  that it is  related to \hyperref[Step7]{Step 7} of Section \ref{sec:strategy}.
Again, we express the controlled harmonic map heat flow equation in terms of the polar coordinates. Without loss of generality, let $r=0$.
Assume that the initial state is $v_0(x)= (\cos (\theta_0(x)), \sin (\theta_0(x)), 0,\ldots , 0)^T$ with
\begin{equation*}
\theta_0\in H^1(0, 2\pi) \textrm{ and } \theta_0(2\pi)= \theta_0(0)+ 2\pi N.
\end{equation*}
We define a $C^2([0, 2\pi])$ function $\theta_1(\cdot)$ as follows:
\begin{equation*}
    \theta_1(x)=
    \begin{cases}
        N_1 x, \textrm{ if } x\in [0, 2\pi- \delta], \\
        N_1 x- (N_1- N) 2\pi, \textrm{ if } x\in [2\pi- \delta/2, 2\pi].
    \end{cases}
\end{equation*}
This function satisfies $(\theta_1)_x(0)= (\theta_1)_x(2\pi)$.
Then we show that there exists a control $h\in L^{\infty}(0, T; L^2(0, 2\pi))$ supported in  $[0, T]\times \omega$ such that the unique solution of
\begin{equation*}
\begin{cases}
        \theta_t(t, x)-  \theta_{xx}(t, x)= h(t, x), \\
        \theta(t, 2\pi)= \theta(t, 0)+ 2\pi N,\\
          \theta_x(t, 2\pi)= \theta_x(t, 0),\\
        \theta(0, x)= \theta_0(x),
        \end{cases}
\end{equation*}
satisfies $\theta(T, \cdot)= \theta_1(\cdot)$. Indeed, since the function $\theta_1$ satisfies
\begin{equation*}
\begin{cases}
       -  \theta_{1xx}(x)= h_1(x), \\
        \theta_1(2\pi)= \theta_1(0)+ 2\pi N,\\
          \theta_{1x}(2\pi)= \theta_{1x}(0)= N_1,
          \end{cases}
\end{equation*}
with $h_1$ supported in  $(2\pi- \delta, 2\pi- \delta/2)$,
it suffices to control the function $w(t, x):= \theta(t, x)- \theta_1(x)$ that is governed by the controlled equation on $\mathbb{T}^1$:
\begin{equation*}
\begin{cases}
        w_t(t, x)-  w_{xx}(t, x)= (h- h_1)(t, x), \\
        w(t, 2\pi)= w(t, 0),\\
          w_x(t, 2\pi)= w_x(t, 0),\\
        w(0, x)= \theta_0(x)- \theta_1(x).
        \end{cases}
\end{equation*}
Since the preceding equation is null controllable, there exists an explicit function $h$ supported in $(0, T)\times \omega$  such that the unique solution satisfies $w(T, x)= 0$. Hence, $\theta(T, x)= \theta_1(x)$.  Define
\begin{gather*}
    u_1(x):= (\cos(\theta_1(x)), \sin(\theta_1(x)), 0,\ldots , 0)^T\; \forall x\in \mathbb{T}^1,\\
    u_2(x):= (\cos(N_1 x), \sin(N_1 x), 0,\ldots , 0)^T\; \forall x\in \mathbb{T}^1.
\end{gather*}

Next, using the assumption $k\geq 2$, we can construct a continuous deformation $u(t, x): [T, 2T]\times \mathbb{T}^1\rightarrow \mathbb{S}^k$  such that
\begin{gather*}
    u(t, x)|_{t\in [T, 2T], x\in \mathbb{T}^1}\in C^1([T, 2T]; C^2(\mathbb{T}^1; \mathbb{S}^k)), \\
    u(T, x)= u_1(x), \; u(2T, x)= u_2(x), \\
    u(t, x)= (\cos(N_1 x), \sin(N_1 x), 0,\ldots , 0)^T\;  \forall t\in [T, 2T], \; \forall x\in [0, 2\pi- \delta].
\end{gather*}
The above constructed function $u(t, x)|_{t\in [T, 2T], x\in \mathbb{T}^1}$ satisfies
\begin{equation*}
    u_t- u_{xx}= |u_x|^2 u+ f,
\end{equation*}
with
\begin{equation*}
    f:=  u_t- u_{xx}- |u_x|^2 u \textrm{ having a support included in  } [T, 2T]\times \omega.
\end{equation*}
Therefore, we have constructed a control $f$ supported in $[0, 2T]\times \omega$ and  a function $u(t, x)|_{t\in [0, 2T], x\in \mathbb{T}^1}$ as a solution to the controlled harmonic map heat flow equation such that
\begin{equation*}
    u(0, \cdot)= v_0(\cdot),  \; u(T, \cdot)= u_1(\cdot)
 \textrm{ and } u(2T, \cdot)= u_2(\cdot).
\end{equation*}
We refer to Figure \ref{pic6} concerning this deformation process.  This finishes the proof of Theorem \ref{lem:excconhm}.

\subsection{The general compact Riemannian manifold target case}\label{sebsec:genR}
Here we turn to  cases where the $\mathbb{S}^k$-target is replaced by   a  general  compact  Riemannian submanifold of $\R^m$: Theorem \ref{th-small-time-global-harmonic}.
In fact, by the Nash embedding theorem, it then covers the case where the target manifold is any compact Riemannian manifold.
The main ingredient of our proof of  Theorem \ref{th-small-time-global-harmonic} is the following lemma.

\begin{lem}
\label{heat-geodesic}
Let $\Gamma$ be a complete geodesic of  $\mathcal{N}$ and let $q_0, q_1$ be two points in this geodesic.  Let  $0<a<a_1<a_0< \pi$ and the controlled  domains be $\omega= [0, a_0]\cap [2\pi- a_0, 2\pi]$,  $\omega_0= [a, a_0]\cap [2\pi- a_0, 2\pi -a]$ and  $\omega_1= [a, a_1]\cap [2\pi- a_1, 2\pi- a]$.  Let $T>0$.

$(i)$ We consider the harmonic map heat flow $[0, T]\times \mathbb{T}^1\rightarrow \mathcal{N}$:
\begin{equation}\label{eq:conhmf:whole}
\begin{cases}
 u_t- u_{xx}= \beta_u (u_x, u_x)+ \mathbf{1}_{\omega}f^{u^{\perp}},\\
 u(0, \cdot)= u_0(\cdot)\subset \Gamma
 \end{cases}
\end{equation}
Assume that the initial state has values inside the geodesic $\Gamma$ and also that the control is always tangent to this geodesic. Then the controlled harmonic map heat flow has values in $\Gamma$ and becomes a linear controlled heat equation.

$(ii)$  We consider the harmonic map heat flow $[0, T]\times (a, 2\pi- a)\rightarrow \mathcal{N}$:
\begin{equation}\label{eq:Inter:conhmf}
\begin{cases}
 u_t- u_{xx}= \beta_u (u_x, u_x)+ \mathbf{1}_{\omega_0}f^{u^{\perp}},\\
 u(t, a)= q_0, \\
 u(t, 2\pi- a)= q_1,\\
 u(0, \cdot)= u_0(\cdot)\subset \Gamma,
 \end{cases}
\end{equation}
where the initial state satisfies the compatibility condition $u_0(a)= q_0$ and $u_0(2\pi- a)= q_1$.
Assume that the control is always tangent to the geodesic $\Gamma$. Then the controlled harmonic map heat flow has values in $\Gamma$ and becomes a linear controlled heat equation.\\
Moreover, for any point $q_2\in \Gamma$ and any initial state $u_0\in C^2([a, 2\pi- a])$, there exists a control $f\in C^0([0, T]\times [a, 2\pi- a])$ such that the unique strong solution satisfies $u\in C^0([0, T]; C^2[a, 2\pi- a])\cap C^1([0, T]; C^0[a, 2\pi- a])$
and
\begin{equation}
u(T, x)= q_2\; \forall x\in [a_1, 2\pi- a_1].
\end{equation}
\end{lem}

The second part of Lemma \ref{heat-geodesic} can be combined with the following gluing lemma to construct $2\pi$-periodic strong solutions of the controlled harmonic map heat flow.
\begin{lem}\label{lem:glue}
Let $\gamma$ be a closed curve on $\mathcal{N}$. Assume that the map  $u: [0, T]\times (a, 2\pi- a)\rightarrow \mathcal{N}$ belongs to $ C^0([0, T]; C^2[a, 2\pi- a])\cap C^1([0, T]; C^0[a, 2\pi- a])$ space. Then we can extend $u$ to $\tilde u: [0, T]\times [0, 2\pi] \rightarrow \mathcal{N}$ so that the extension $\tilde u$ satisfies
\begin{gather}
\tilde u\in C^0([0, T]; C^2[0, 2\pi])\cap C^1([0, T]; C^0[0, 2\pi]), \\
\tilde u(t, 0)= \tilde u(t, 2\pi), \forall t\in [0, T],\\
\tilde u_x(t, 0)= \tilde u_x(t, 2\pi), \forall t\in [0, T],\\
\tilde u_{xx}(t, 0)= \tilde u_{xx}(t, 2\pi), \forall t\in [0, T], \\
\tilde u(t, \cdot) \textrm{ is homotopic to  } \gamma, \forall t\in [0, T].
\end{gather}
\end{lem}

\begin{proof}[Proof of Lemma~\ref{heat-geodesic}]
The statement concerning  the values of the flow follows by checking that the solution to the controlled harmonic map heat flow with values in the Riemannian submanifold $\Gamma\subset \R^m$ is also a solution to the controlled harmonic map heat flow with values in the Riemannian submanifold $\mathcal{N}\subset\R^m$. Concerning the second statement, suppose that
\begin{align*}
\bar u: \mathbb{R}&\rightarrow \Gamma \subset \mathcal{N} \\
s& \mapsto \bar u(s)
\end{align*}
 is a non-constant harmonic map:
\begin{equation}
    -\bar u_{ss}= \beta_{\bar u} (\bar u_s, \bar u_s) \; \forall s\in \R,
\end{equation}
where $\beta$ is the second fundamental form of $\mathcal{N}$. Note that $\bar u (\R)=\Gamma$.
We also note that for any two points $p$ and $p_f$ in the same connected component of $\mathcal{N}$, there exists a complete geodesic $\Gamma$ containing both points.

$(i)$ We know that $u$ is a solution to the controlled harmonic map heat flow if and only if
\begin{equation}
    u_t- u_{xx}= \beta_u (u_x, u_x)+  \mathbf{1}_{\omega_0} f^{u^{\perp}}.
\end{equation}
Assume that $u(t, x)= \bar u(\varphi(t, x))$ for every $(t, x)\in [0, T]\times \mathbb{T}^1$, with
\begin{equation*}
    \varphi: [0, T]\times \mathbb{T}^1\rightarrow \mathbb{R},
\end{equation*}
and that  $f=  f_0 \bar u_s(\varphi)\in T_{\bar u(\varphi)}\Gamma\subset  T_{\bar u(\varphi)}\mathcal{N}$, where $f_0:[0,T]\times \mathbb{T}^1\rightarrow \R$.
Let
\begin{equation}
    A:= u_t- u_{xx}- \beta_u (u_x, u_x)- \mathbf{1}_{\omega}f^{u^{\perp}}.
\end{equation}
It satisfies
\begin{align*}
    A&= \bar u_s(\varphi) \varphi_t- \bar u_{ss}(\varphi) |\varphi_x|^2- \bar u_s(\varphi) \varphi_{xx}- \beta_{\bar u} (\bar u_s(\varphi), \bar u_s(\varphi)) |\varphi_x|^2- \mathbf{1}_{\omega}f_0 \bar u_s(\varphi)\\
    &= \bar u_s(\varphi) \left(\varphi_t- \varphi_{xx}- \mathbf{1}_{\omega}f_0\right)- \Big(\bar u_{ss}(\varphi)+ \beta_{\bar u} (\bar u_s(\varphi), \bar u_s(\varphi))\Big) |\varphi_x|^2\\
    &= \bar u_s(\varphi) \left(\varphi_t- \varphi_{xx}- \mathbf{1}_{\omega}f_0\right).
\end{align*}
Note that, since $\bar u$ is a non-constant harmonic map, $\bar u_s(\varphi)$ does not vanish. Hence $u$ is a solution to the controlled harmonic map heat flow  if and only if $\varphi$ is a solution to the controlled heat equation
\begin{equation}\label{eq:heat-general}
    \varphi_t- \varphi_{xx}= \mathbf{1}_{\omega}f_0.
\end{equation}

$(ii)$ Assume that
\begin{equation}
q_0= \bar u(A), q_1= \bar u(B) \textrm{ and } q_2= \bar u(D).
\end{equation}
 Similarly,  assume that $u(t, x)= \bar u(\varphi(t, x))\;  \forall (t, x)\in [0, T]\times [a, 2\pi- a]$, with
\begin{equation*}
    \varphi: [0, T]\times [a, 2\pi- a]\rightarrow \mathbb{R},
\end{equation*}
and that  $f=  f_0 \bar u_s(\varphi)\in T_{u(\varphi)}\Gamma\subset  T_{u(\varphi)}\mathcal{N}$, where $f_0:[0,T]\times [a, 2\pi- a]\rightarrow \R$.
Then
\begin{align*}
    u_t- u_{xx}- \beta_u (u_x, u_x)- \mathbf{1}_{\omega}f^{u^{\perp}}= \bar u_s(\varphi) \left(\varphi_t- \varphi_{xx}- \mathbf{1}_{\omega_0}f_0\right).
\end{align*}
 Hence $u$ is a solution to the controlled harmonic map heat flow  if and only if $\varphi$ is a solution to the controlled heat equation
\begin{equation}\label{eq:heat-general-interval}
\begin{cases}
 \varphi_t- \varphi_{xx}= \mathbf{1}_{\omega_0}f_0, \\
 \varphi(t, a)= A, \\
 \varphi(t, 2\pi- a)= B, \\
 \varphi(0, \cdot)\in C^2([a, 2\pi- a]; \mathbb{R}).
\end{cases}
\end{equation}

Next, we show that there exists a continuous control $f$ such that the solution satisfies $u(T, x)= q_3$ for every $x\in [a_1, 2\pi- a_1]$. It suffices to show that there exists $f_0$ such that  the solution of \eqref{eq:heat-general-interval} satisfies $\varphi(T, x)= D\; \forall x\in [a_1, 2\pi- a_1]$.

We construct a steady state $\bar \varphi\in C^2([a, 2\pi- a]; \mathbb{R})$ such that, for some function $g_0\in C^2([a, 2\pi- a]; \mathbb{R})$ with a support contained in $\omega_1$,
\begin{equation}
\begin{cases}
- \bar \varphi_{xx}= \textbf{1}_{\omega_1} g_0,\\
\bar \varphi(a)= A,\\
 \bar \varphi(2\pi- a)= B, \\
 \bar \varphi(x)= D\; \forall x\in [a_1, 2\pi- a_1],
\end{cases}
\end{equation}
and consider the state $\phi:= \varphi- \bar \varphi$ which satisfies
\begin{equation}\label{eq:heat-general-interval_phi}
\begin{cases}
 \phi_t- \phi_{xx}= \mathbf{1}_{\omega_0}f_0- \textbf{1}_{\omega_1} g_0, \\
 \phi(t, a)= 0, \\
 \phi(t, 2\pi- a)= 0, \\
 \phi(0, \cdot)= \varphi(0, \cdot)- \bar \varphi(0, \cdot)\in C^2([a, 2\pi- a]; \mathbb{R}).
\end{cases}
\end{equation}
The preceding system on $\phi$ is null controllable with $\mathbf{1}_{\omega_0}f_0\in C^0$. Thus we have constructed a solution $\varphi$ that satisfies  $\varphi(T, x)= D\; \forall x\in [a_1, 2\pi- a_1]$.
This concludes the proof of Lemma~\ref{heat-geodesic}.

\end{proof}

Thanks to the previous lemmas, one can adapt the ideas of Theorems \ref{lem:excconhm0}--\ref{lem:excconhm} to get the small-time global controllability between homotopic harmonic maps, i.e.  Theorem~\ref{th-small-time-global-harmonic}. More precisely, the proof consists of the following five stages which are illustrated in Figure \ref{picma}.
\begin{proof}[The proof of Theorem~\ref{th-small-time-global-harmonic}]
 We select $p_0\in \gamma_0$, $p_1\in \gamma_1$ and further find a constant speed complete geodesic $\Gamma= \{\bar u(s): s\in \mathbb{R}\}$ such that $\bar u(a)= p_0, \bar u(2\pi- a)= p_1$.

\noindent \textit{ Stage 1.} In this first stage we use the idea of Theorems \ref{lem:excconhm0}, or Lemma \ref{heat-geodesic} for the special case that $\Gamma$ is a closed geodesic,  to deform the state on the geodesic $\gamma_0$. It suffices to select  controls that are tangent to this geodesic. Since the initial state is the harmonic map $\gamma_0$, which is smooth,  we can find a control $f\in C^0([0, T/5]\times \mathbb{T}^1)$  to steer the controlled harmonic map heat flow from $\gamma_0(\cdot)$ to  $v_1\in C^2(\mathbb{T}^1; \mathcal{N})$ satisfying
\begin{gather*}
v_1(x)\in \{\gamma(s): s\in \mathbb{T}^1\} \; \forall x\in \mathbb{T}^1, \\
v_1(x)= p_0\; \forall x\in [a, 2\pi- a_1].
\end{gather*}

\noindent \textit{ Stage 2.}    Using the idea of Theorem \ref{lem:excconhm} we can find a control $f\in C^2([T/5, 2T/5]\times \mathbb{T}^1)$ to move the state from $v_1$ to any state  $v_2\in C^2(\mathbb{T}^1; \mathcal{N})$ satisfying
\begin{gather*}
v_2(x)\in \Gamma \; \forall x\in [a, 2\pi- a], \\
v_2(x)= p_0\; \forall x\in [a, 2\pi- a_1], \\
v_2(2\pi- a)= p_1, \\
v_2 \textrm{ is homotopic to } v_1.
\end{gather*}
However, currently we only fix the value of $v_2$ on $[a, 2\pi- a]$, while the explicit choice of $v_2$ on $\mathbb{T}^1\setminus [a, 2\pi- a]$ will be fixed later on in the next step.

\noindent \textit{ Stage 3.} Notice that $v_2|_{x\in [a, 2\pi- a]}$ has values in the complete geodesic $\Gamma$. We adapt the idea of Lemma \ref{heat-geodesic} $(ii)$ to  the harmonic map heat flow $[2T/5, 3T/5]\times (a, 2\pi- a)\rightarrow \mathcal{N}$:
\begin{equation}
\begin{cases}
 w_t- w_{xx}= \beta_w (w_x, w_x)+ \mathbf{1}_{\omega_0}g^{w^{\perp}},\\
 w(t, a)= p_0, \\
 w(t, 2\pi- a)= p_1,\\
 w(2T/5, \cdot)|_{(a, 2\pi- a)}= v_2(\cdot)|_{(a, 2\pi- a)}\subset \Gamma.
 \end{cases}
\end{equation}
Thus there exists a control $g\in C^0([2T/5, 3T/5]\times \omega_0)$ such that the solution  $w\in C^0([2T/5, 3T/5]; C^2[a, 2\pi- a])\cap C^1([2T/5, 3T/5]; C^0[a, 2\pi- a])$ satisfies
\begin{gather*}
w(3T/5, x)\in \Gamma \; \forall x\in [a, 2\pi- a], \\
w(3T/5, x)= p_1\;  \forall x\in [a_1, 2\pi- a_1], \\
w(3T/5, a)= p_0, \\
w(3T/5, 2\pi- a)= p_1.
\end{gather*}
Thanks to the gluing Lemma \ref{lem:glue}, we can extend the control $g$ to $f$ on $[2T/5, 3T/5]\times \omega$ and the state $w$ to $u$ on $[2T/5, 3T/5]\times \mathbb{T}^1$ such that
\begin{equation}
\begin{cases}
 u_t- u_{xx}= \beta_u (u_x, u_x)+ \mathbf{1}_{\omega}f^{u^{\perp}},\\
 u(t, x)= w(t, x)\; \forall (t, x)\in [2T/5, 3T/5]\times (a, 2\pi- a), \\
u\in C^0([2T/5, 3T/5]; C^2(\mathbb{T}^1))\cap C^1([2T/5, 3T/5]; C^0(\mathbb{T}^1)), \\
 u(t, \cdot) \textrm{ is homotopic to  } \gamma_0, \forall t\in [0, T].
 \end{cases}
\end{equation}
Hence, we will choose the exact value of  $v_2(\cdot)$  as $u(2T/ 5, \cdot)\in C^2(\mathbb{T}^1; \mathcal{N})$. We denote the value of $u(3T/5)$ as $v_3\in C^2(\mathbb{T}^1; \mathcal{N})$ which satisfies
\begin{gather*}
v_3(x)\in \Gamma \; \forall x\in [a, 2\pi- a], \\
v_3(x)= p_1\in \gamma_1\;  \forall x\in [a_1, 2\pi- a_1], \\
v_3(a)= p_0, \\
v_3(2\pi- a)= p_1.
\end{gather*}

\noindent \textit{ Stage 4.}  Similar to Step 2, we use the idea of Theorem \ref{lem:excconhm} to  find a control $f\in C^2([T/5, 2T/5]\times \mathbb{T}^1)$ that steers the state from $v_3$ to  $v_4\in C^2(\mathbb{T}^1; \mathcal{N})$ satisfying
\begin{gather*}
v_4(x)\in \{\gamma_1(s): s\in \mathbb{T}^1\} \; \forall x\in \mathbb{T}^1, \\
v_4(x)= p_1\;  \forall x\in [a_1, 2\pi- a_1].
\end{gather*}

\noindent \textit{ Stage 5.}  Finally, in the last step, similar to Step 1,  using the idea of Theorem \ref{lem:excconhm0} we can find a control that is tangent to the geodesic $\gamma_1$ to deform the state from $v_4$ to $\gamma_1$.
\end{proof}

\vspace{0.5cm}

\noindent\textbf{Acknowledgments} \;
Shengquan Xiang is financially  supported by “The Fundamental Research Funds for the Central Universities, 7100604200, Peking University” and NSF of China under Grant  12301562.

\bibliographystyle{alpha}
\bibliography{biblio}

\end{document}